\documentclass{article}
\usepackage{arxiv/arxiv_spread}
\usepackage{parskip}

\usepackage[utf8]{inputenc} 
\usepackage[T1]{fontenc}    
\usepackage{microtype}
\usepackage{graphicx}
\usepackage{booktabs} 
\usepackage{pgfplots}
\usepgfplotslibrary{groupplots}
\usepackage{pifont}

\usepackage[colorlinks=true,linkcolor=blue,citecolor=blue,urlcolor=blue]{hyperref}
\usepackage[round]{natbib}

\usepackage{algorithmic}
\usepackage[ruled,vlined]{algorithm2e}

\usepackage{amsmath}
\usepackage{amssymb}
\usepackage{bbm}
\usepackage{mathtools}
\usepackage{amsthm}
\usepackage{multicol}
\usepackage[capitalize,noabbrev]{cleveref}

\usepackage{xcolor}
\definecolor{orange}{rgb}{0.2,0.7,0.4}

\usepackage{graphicx}
\usepackage{wrapfig}
\graphicspath{{figures/}}
\usepackage{caption}
\usepackage{subcaption}
\usepackage{amsfonts}
\usepackage{mwe}
\usepackage{multirow}

\usepackage{enumitem}
\usepackage{makecell}

\usepackage{tabularx}
\usepackage{float}

\usepackage{pgfplots}

\usepackage{enumitem}
\setlist[enumerate,1]{label={\bfseries\arabic*.},ref={\bfseries\theenumii.\arabic*}}
\setlist[enumerate,2]{label={\bfseries\alph*)},ref={\bfseries\theenumi\alph*)}}
\usepackage{makecell}

\newcommand{\st}{\mathop{\mathrm{subject\,\,to}}}

\usepackage{tabularx}
\usepackage{float}

\usepackage{pgfplots}

\usepackage{amsmath,amsfonts,bm,mathrsfs}

\def\1{\bm{1}}

\DeclareMathAlphabet{\mathsfit}{\encodingdefault}{\sfdefault}{m}{sl}
\SetMathAlphabet{\mathsfit}{bold}{\encodingdefault}{\sfdefault}{bx}{n}
\newcommand{\tens}[1]{\bm{\mathsfit{#1}}}

\def\gA{{\mathcal{A}}}

\def\gD{{\mathcal{D}}}
\def\gE{{\mathcal{E}}}
\def\gF{{\mathcal{F}}}

\def\gH{{\mathcal{H}}}

\def\gO{{\mathcal{O}}}

\def\gR{{\mathcal{R}}}

\def\gW{{\mathcal{W}}}
\def\gX{{\mathcal{X}}}

\def\sI{{\mathbb{I}}}

\def\sP{{\mathbb{P}}}

\newcommand{\etens}[1]{\mathsfit{#1}}

\DeclareMathOperator*{\argmin}{arg\,min}

\newcommand\norm[1]{\left\lVert#1\right\rVert}

\newcommand\abs[1]{\left| #1 \right|}

\newcommand{\paren}[1]{\left ( #1 \right )}

\newcommand{\bracket}[1]{\left [ #1 \right ]}

\DeclareMathOperator*{\Prob}{\ensuremath{{\mathbb{P}}}}

\newcommand{\ind}[1]{\sI_{\bracket{#1}}}

\DeclareMathOperator*{\E}{\ensuremath{{\mathbb{E}}}}

\newcommand{\ept}[1]{\E\bracket{#1}}

\NewDocumentCommand{\expect}{ e{^} s o >{\SplitArgument{1}{|}}m }{%
  \E
  \IfValueT{#1}{{\!}^{#1}}
  \IfBooleanTF{#2}{
    \expectarg*{\expectvar#4}%
  }{
    \IfNoValueTF{#3}{
      \expectarg{\expectvar#4}%
    }{
      \expectarg[#3]{\expectvar#4}%
    }%
  }%
}
\NewDocumentCommand{\expectvar}{mm}{%
  #1\IfValueT{#2}{\nonscript\;\delimsize\vert\nonscript\;#2}%
}

\newcommand{\R}{\mathbb{R}}

\newcommand{\set}[1]{\ensuremath{\left\{ #1 \right\}}}

\newcommand{\Span}[1]{\textnormal{Span}\set{#1}}

\theoremstyle{plain}
\newtheorem{theorem}{Theorem}[section]
\newtheorem{lemma}[theorem]{Lemma}
\newtheorem{proposition}[theorem]{Proposition}

\theoremstyle{definition}
\newtheorem{definition}[theorem]{Definition}
\newtheorem{assumption}[theorem]{Assumption}
\theoremstyle{remark}
\newtheorem{remark}[theorem]{Remark}

\newtheorem{example}[]{Example}

\newcommand{\innerprod}[2]{ \Big\langle #1, #2 \Big\rangle }

\newcounter{relctr} 
\everydisplay\expandafter{\the\everydisplay\setcounter{relctr}{0}} 

\AtBeginDocument{\let\originallabel\label}

\def\beq#1\beq{\begin{equation}\begin{aligned}#1\end{aligned}\end{equation}}

\def\bea#1\bea{
\begingroup
\allowdisplaybreaks\begin{align*}#1\end{align*}\endgroup}

\let\oldexp\exp
\renewcommand{\exp}[1]{\oldexp\paren{#1}}

\newtheorem{innercustomgeneric}{\customgenericname}
\providecommand{\customgenericname}{}
\newcommand{\newcustomtheorem}[2]{%
  \newenvironment{#1}[1]
  {%
   \renewcommand\customgenericname{#2}%
   \renewcommand\theinnercustomgeneric{##1}%
   \innercustomgeneric
  }
  {\endinnercustomgeneric}
}

\newcustomtheorem{cthm}{Theorem}
\newcustomtheorem{clemma}{Lemma}
\newcustomtheorem{ccor}{Corollary}
\newcustomtheorem{cassump}{Assumption}
\newcustomtheorem{cprop}{Proposition}

\newcommand{\ubar}[1]{\underaccent{\bar}{#1}}

\makeatletter
\let\amsmath@bigm\bigm

\renewcommand{\bigm}[1]{%
  \ifcsname fenced@\string#1\endcsname
    \expandafter\@firstoftwo
  \else
    \expandafter\@secondoftwo
  \fi
  {\expandafter\amsmath@bigm\csname fenced@\string#1\endcsname}%
  {\amsmath@bigm#1}%
}

\newcommand{\DeclareFence}[2]{\@namedef{fenced@\string#1}{#2}}
\makeatother

\DeclareFence{\mid}{|}

\usepackage{bm}  
\makeatletter
\def\ubar#1{\underline{\sbox\tw@{$#1$}\dp\tw@\z@\box\tw@}}
\makeatother

\title{Quantile Additive Trend Filtering
}

\author{
  Zhi Zhang, \hspace{3pt}Kyle Ritscher, \hspace{3pt}and Oscar Hernan Madrid Padilla\hspace{5pt}  
    \\ \\
  {\normalsize Department of Statistics and Data Science, University of California, Los Angeles}\\
}    

\date{}

\begin{document}

\maketitle


\begin{abstract}
\vspace{-0pt}
This paper introduces and analyzes quantile additive trend filtering, a novel approach to model the conditional quantiles of the response variable given multivariate covariates. Under the assumption that the true model is additive, and that the components are functions whose $r$th order weak derivatives have bounded total variation, our estimator is a constrained version of quantile trend filtering within additive models. The primary theoretical contributions are the error rate of our estimator in both fixed and growing input dimensions. In the fixed dimension case, we show that our estimator attains a rate that mirrors the non-quantile minimax rate for additive trend filtering, featuring the main term $n^{-2r/(2r+1)}$. For growing input dimension ($d$), our rate has an additional polynomial factor $d^{(2r+2)/(2r+1)}$.
We propose a practical algorithm for implementing quantile additive trend filtering using dimension-wise backfitting.  Experiments in both real data and simulations confirm our theoretical findings. We provide a public implementation of the algorithm at \url{https://github.com/zzh237/QATF}.

\end{abstract}

\keywords{Total Variation, nonparametric quantile regression, additive trend filtering.}

\section{INTRODUCTION}

Quantile regression has been widely used for analyzing various types of data, especially in the presence of heteroscedasticity or non-normal error distributions. 
It has been considered robust, with broad applications ranging from economics and finance to environmental science and healthcare \citep{coad2006innovation,sanderson2006statistically, perlich2007high, benoit2009benefits, wasko2014quantile, pata2015astronomical, belloni2023high}.


Additive models and trend filtering have emerged as powerful tools for handling high-dimensional and nonparametric estimation problems. Additive models, introduced by \cite{friedman1981projection}, reduce dimensionality while providing interpretable predictions, and have been utilized in many fields, including finance \citep{hastie1987non}, marketing \citep{breiman1985estimating}, healthcare \citep{hastie2017generalized},  environmental pollution assessment \citep{hastie2017generalized}, policy analysis \citep{hastie2017generalized}, energy demand forecasting\cite{fasiolo2021fast},  social science \citep{fasiolo2021fast}, and politics  \citep{petersen2016fused}.


Trend filtering, proposed by \cite{mammen1997locally,kim2009ell_1}, is a nonparametric method which fits a piecewise polynomial function. It has garnered significant interest in nonparametric research \citep{rudin1992nonlinear, tibshirani2014adaptive, kim2009ell_1,sadhanala2019additive}, and is particularly useful in scenarios with in-homogeneous smoothness or sharp changes in the underlying data, such as image processing \citep{rudin1992nonlinear}, final market crash \citep{gu2021factor}, shifts in the dynamics of macroeconomics \citep{mulvey2016identifying}, and understanding of earthquakes \citep{yano2022l1}.



Motivated by data sets  with heavy-tails and outliers, we build our model based on quantile regression. To accommodate the prevalent application of additive models, we assume the underlying function is additive, and we use trend filtering as a regularization technique.
This leads us to a more robust, flexible and interpretable framework, known as quantile additive trend filtering, which combines the benefits of  quantile regression, additive models, and trend filtering.


\subsection{Summary of Contributions}

The proposed model in this paper is quantile additive trend filtering (QATF), which offers a dual benefit of dimension reduction and robust regression, and has not yet been studied. 
We conduct an examination of risk bounds for QATF and develop an algorithm for learning the QATF model. The techniques we developed could also be used for analyzing the other additive models with different classes of functions. The most related work to our work is the paper \cite{madrid2022risk}, which focuses on the univarite quantile trend filtering, but does not consider additive models. We also provide non-asymptotic analysis, in contrast to \cite{narayanexpected}, which only addresses asymptotic behavior of the estimator. Moreover, the papers  \cite{gasthaus2019probabilistic} and \cite{tagasovska2019single} primarily focus on methodology and experimentation. 


QATF is built around the univariate trend filtering estimator, defined by constraining according to the sum of $\ell_1$ norms of discrete derivatives of the component functions, with a quantile loss applied. When the underlying quantile function is additive, with components whose $(r-1)$th derivatives have total variation bounded by $V$, we derive error rates for $(r-1)$th order quantile additive trend filtering. This rate is $n^{-2r/(2r+1)} V^{2/(2r+1)}\max\{1, V^{(2r-1)/(2r+1)}\}$ for a fixed input dimension $d$ (under weak assumptions). The main term $n^{-2r/(2r+1)}V^{2/(2r+1)}$, is the same as the non-quantile minimax rate in \cite{sadhanala2019additive}. It also includes an extra term, which is of order $\mathcal{O}(1)$ if $V = \mathcal{O}(1)$. Such a value of $V$ is referred to as canonical scaling, as described in \cite{madrid2022risk,sadhanala2016total}. 

Additionally, we explore error rates for QATF where the input dimension ($d$) increases with the sample size ($n$). We show the proposed estimator exhibits an error rate that includes a polynomial factor of $d^{(2r+2)/(2r+1)}$ multiplied by the fixed dimension error rate. When $r$ is large, the exponent $d^{(2r+2)/(2r+1)}$ approaches 1, recovering the linear dependence on $d$ in the standard additive trend filtering model. 

Our rates are minimax up to logarithmic factors in fixed dimension cases, and off by a small factor in the growing dimension setting. This aligns closely with the minimax rates for mean estimation in additive trend filtering, as discussed in \cite{sadhanala2019additive}. However, \cite{sadhanala2019additive} primarily considers sub-Gaussian errors, where we only introducing minor assumptions, proving that the constrained quantile trend filtering estimator is robust
to heavy-tailed (with no moment conditions required) distributions of the errors. In addition, the proof techniques we employ differ substantially from those in the cited work, see our discussion in Section~\ref{app_proof_contribution} of the Appendix.

Finally, we develop an algorithm for performing quantile trend filtering in additive models, which involves a backfitting approach for each dimension. We validate the algorithm in extensive simulation settings consisting of additive models in the presence of noise, and then demonstrate the utility of our method by constructing prediction intervals using the publicly available 2024 World Happiness dataset.

The review of additive models, trend filtering, and quantile regression is given in Section~\ref{app_review_of_things} of Appendixs.

\section{QUANTILE TREND FILTERING FOR ADDITIVE MODELS} 
\subsection{Notations}
\label{notations}
Before we define the estimator, let us introduce the necessary definitions and notation.
Given a distribution $Q$, and a set of  i.i.d. points $X =  \set{X^1, \dots, X^n} \subseteq [0,1]^d$, $i=1,\dots, n$, from $Q$, we define the $L_2$, euclidean and empirical inner products $\innerprod{\cdot}{\cdot}_{L_2}$  $\innerprod{\cdot}{\cdot}_2$  $\innerprod{\cdot}{\cdot}_n$ as below.

Let $\gX$ be $[0,1]^d$, for two functions $f: \gX \rightarrow \R$ and $g: \gX \rightarrow \R$, we define 
\bea
&\innerprod{f}{g}_{L_2} = \int_{\gX} f(x)g(x)dQ(x), \innerprod{f}{g}_2 = \sum_{i=1}^n f(X^i)g(X^i),
\bea
And $\innerprod{f}{g}_n = \frac{1}{n}\innerprod{f}{g}_2$.
Hence, we define the $L_2$ norm $\norm{\cdot}_{L_2}$ as $\norm{f}_{L_2}^2 = \innerprod{f}{f}_{L_2} = \int_{\gX} f(x)^2 dQ(x)$. 
We have empirical norm $\norm{\cdot}_{n}$ as 
$\norm{f}_{n}^2 = \innerprod{f}{f}_n = \frac{1}{n}\sum_{i=1}^n f(X^i)^2$. The $\norm{\cdot}_2$ and $\norm{\cdot}_{\infty}$ are defined as $\norm{f}_{2}^2 = n \norm{\cdot}_{n}^2$, and $\norm{f}_{\infty} = \max_{i=1,\dots n}\abs{f(X^i)}$. We also use the $\norm{\cdot}$ as the common euclidean norm, and we use $\norm{\cdot}_{\infty}$ as the vector $\ell_{\infty}$ norm. For sequences $a_n$ and $b_n$, we will also use the notation $a_n \lesssim b_n$ to denote that $a_n \leq Cb_n$ for an absolute constant $C$. We write    $a_n = O(b_n)$  if there exists   constants $C>0$  and  $n_0 >0$ such that  $n \geq  n_0$ implies   that  $a_n \,\leq \,  C b_n$.  
	Furthermore, if  $a_n =  O(b_n)$ and  $b_n =  O(a_n)$ then we write  $a_n =  \Theta(b_n)$ or  $a_n \asymp b_n$.
	For a sequence of random variables $x^{(n)}$ and a positive sequence $a_n$ we write  $x^{(n)} = O_{\mathrm{pr}}(a_n)$ if  for every $\epsilon>0$ there exists  $M > 0$ such that    $\sP\left(|x^{(n)}| \,\geq\,  M a_n\right) < \epsilon$ for all $n$. Let $\ind{\gA}$ be the indicator function over the set $\gA$ such that
        $
\ind{\gA} =
\begin{cases}
1, & \text{if } a \in \gA, \\
0, & \text{if } a \notin \gA.
\end{cases}
$
For a functional $\nu$,  we let $B_{\nu}(\epsilon)$ for the $\nu$-ball of radius $V > 0$, i.e., $B_{\nu}(V) = \set{f: \nu(f) \le V}$.  So $B_n(V)$ for the $\norm{\cdot}_n$-ball of radius $V$, and $B_{L_2}(V)$ for the $\norm{\cdot}_{L_2}$-ball of radius $V$, and $B_{\infty}(V)$ for the $\norm{\cdot}_{\infty}$-ball of radius $V$.

For an additive function $f = \sum_{j=1}^df_j$, its input is a vector $X^i = (X^i_1, \dots, X^i_d) \in \R^d$, and we have  
$ 
f(X^i) = \sum_{j=1}^d f_j(X^i_j).
$ 
In this paper, we reserve index $i$ for data point, and $j$ for the component, and for any function $f$, we use $f^i$ as a shorthand for $f(X^i)$, and $f^i_j$ for $f_j(X^i_j)$.

Next, we introduce our specific model and data framework.

\subsection{Model and Data Description}
Suppose we are given data \( \{(X^i, Y^i)\}_{i=1}^n \subset \mathcal{X} \times \mathbb{R} \), where \( \mathcal{X} = [0, 1]^d \), generated as
\begin{equation}
\label{eq45}
\begin{aligned}
    Y^i &= f_{0}(X^i) + \epsilon^i = \sum_{j=1}^d f_{0j}(X^i_j) + \epsilon^i,
\end{aligned}
\end{equation}
where \( \mathbb{P}(\epsilon^i \leq 0 \mid X^i) = \tau \), with 
\begin{align*}
f_{0}: [0, 1]^d \rightarrow \mathbb{R}, \quad  f_{0j}: [0, 1] \rightarrow \mathbb{R},   
\end{align*}
and where \( X^1, X^2, \ldots, X^n \) are i.i.d. draws from a distribution \( \mathscr{F} \) in \( [0, 1]^d \).

Thus, $f_{0}(X^i) = \sum_{j=1}^d f_{0j}(X^i_j)$
is the \(\tau\)-conditional quantile of \(Y^i\) given \(X^i = (X^i_1, \ldots, X^i_d)\) of \(d\) dimensions, i.e. $f_0(X^i) = F^{-1}_{Y^i|X^i}(\tau)$, for $i=1, \dots, n$, where $F_{Y^i|X^i}(\cdot)$ denotes the conditional cumulative distribution function of $Y^i$ given $X^i$. 


An alternative way to write our model is to let
\[
\theta^i_0 \coloneqq f_0(X^i) \quad \text{and} \quad \theta_{0j}^i =  f_{0j}(X^i_j),
\]
so that $Y^i$ can be written as 
\begin{equation}
\begin{aligned}
Y^i &= \theta^i_0 + \epsilon^i = \sum_{j=1}^d \theta_{0j}^i + \epsilon^i.
\label{eq:alternative_model}
\end{aligned}
\end{equation}

We find (\ref{eq:alternative_model}) to be helpful when both describing the methods and theory for our proposed model.

Furthermore, we collect the \(j\)th component of the inputs into a vector of \(n\) dimensions, i.e., \(X_j = (X^1_j, X^2_j, \dots, X^n_j)\) for \(j = 1, \dots, d\).

Our goal is to estimate $f_0$ or $\theta_0$. 
In this paper we are interested in scenarios where $f_0$ is (or is close to) a piecewise polynomial function.

\subsection{Discrete Difference Operator}

The discrete difference operator is a fundamental component in defining the trend filtering estimator for a given integer \( r \). 

Let \( r \in \mathbb{N}_{\ge 1} \) and denote \( X =  (X^1, \dots, X^n) \in \R^n \) as a vector of univariate input points. Let $(X^{(1)}, \dots, X^{(n)})$ be the ordered inputs  where \( X^{(1)} < \dots < X^{(n)} \). The operator \( D^{(X,r)}_n \in \R^{(n-r) \times n} \) is defined on the ordered inputs recursively.\\
\textbf{First-Order Difference Operator (\( r = 1 \)).} The first-order difference operator \( D^{(X,1)}_n \) is defined as:
\beq 
\label{eq:first_order}
&D^{(X,1)}_n = \begin{pmatrix}
-1 & 1 & 0 & \dots & 0 & 0 \\
0 & -1 & 1 & \dots & 0 & 0 \\ 
\vdots & \vdots & \vdots & \ddots & \vdots & \vdots \\
0 & 0 & 0 & \dots & -1 & 1 
\end{pmatrix} \in \R^{(n-1) \times n}, \\ 
\quad & D^{(X,1)}_n(\theta) = (\theta^2 - \theta^1, \dots, \theta^n - \theta^{n-1}).
\beq 

\paragraph{Higher-Order Difference Operators (\( r \geq 2 \)).} For higher-order difference operators, we recursively define \( D^{(X,r)}_n \in \R^{(n-r) \times n} \) using the previous order's operator,
\begin{equation}
\label{eq:higher_order}
\begin{aligned}
D^{(X,r)}_n = D^{(X,1)}_{n-r+1} \cdot \text{diag}&\left(\frac{r-1}{X^{(r)} - X^{(1)}}, \dots,\right. \\
&\left.\frac{r-1}{X^{(n)} - X^{(n-r+1)}}\right) D^{(X,r-1)}_n,
\end{aligned}
\end{equation}
where \( D^{(X,1)}_{n-r+1} \in \R^{(n-r) \times (n-r+1)} \) is first-order operator as in Equation \eqref{eq:first_order}. Examples of the discrete difference operator are given in Section~\ref{app_examples_discrete_difference_operator} of the Appendix.

\noindent
In summary, the matrix \( D^{(X,r)}_n \) generalizes the discrete difference operator to account for non-uniformly spaced inputs, making it a powerful tool in trend filtering applications.

\subsection{Univariate Quantile Trend Filtering}
We first introduce the univariate quantile trend filtering, which focuses on signals (quantile sequences) that have bounded \(r\)th order total variation. 

We first introduce a function used in quantile regression to define the loss. The check function, $\rho_{\tau}(u)$, for a given quantile level \(\tau \in (0,1)\), is defined as follows:
\bea \rho_{\tau}(u) =  u\paren{\tau - \ind{u<0}} = \max\set{\tau u, (\tau-1)u}. 
\bea 

In the analysis presented in \cite{madrid2022risk},
the constrained univariate quantile trend filtering estimator $\widehat \theta^{(r)}$ is given by 
\beq 
\label{u_q_t}
&\min_{\theta\in \R^n} \sum_{i=1}^n \rho_{\tau}(Y^i - \theta^i) 
\\
&\text{subject to}\quad \|D^{(r)}(\theta)\|_1 \le V.
\beq 
\begin{remark}
{\bf a.} The $D^{(r)}$ is the $r$th order difference operator, for a vector \(\theta \in \mathbb{R}^n\),  \(D^{(0)}(\theta) = \theta\), \(D^{(1)}(\theta) = (\theta^2 - \theta^1, \ldots, \theta_n - \theta_{n-1})\), and \(D^{(r)}(\theta)\), for \(r \geq 2\), as recursively defined as \(D^{(r)}(\theta) = D^{(1)}(D^{(r-1)}(\theta))\). Note that \(D^{(r)}(\theta) \in \mathbb{R}^{n-r}\).\\
{\bf b.} $D^{(r)}$ in the univariate quantile trend filtering estimator is equal to $D^{(X,r)}_n$ in \eqref{eq:higher_order} when $x$ equals the grid $(1/n, 2/n, \ldots, n/n)$.
\end{remark}


Due to \(\ell_1\) penalization  constrained univariate quantile trend filtering estimator enforce \(D^{(r)}(\widehat{\theta}^{(r)})\) to be sparse. It is known that for \(\theta \in \mathbb{R}^n\), \(D^{(r)}(\theta)\) has \(k\) nonzero entries if and only if \(\theta = (f(1/n), f(2/n), \ldots, f(n/n))^\top\) for a \textit{discrete spline function} \(f\), consisting of \((k+1)\) polynomials of degree \(r-1\) (see Proposition D.3 in \cite{guntuboyina2020adaptive}). For this reason, just like the usual trend filtering estimators, the univariate quantile trend filtering estimator fits discrete splines. For the precise definition of a discrete spline see Section 2 in \cite{mangasarian1971discrete}.



\subsection{Quantile Additive Trend Filtering via Discrete Difference Operator}

Given the definition of the discrete difference operator, the quantile additive trend filtering (QATF) estimator regularizes each component function based on the total variation of its \( r \)th order discrete derivative. To define the estimator, we introduce the constrained set \( K^{(X,r)}(V) \), and the estimator \( \widehat{\theta}^{(r)} \) is derived from the following optimization problem:
\begin{equation}
\label{eq303}
\min_{\theta \in K^{(X,r)}(V)} \sum_{i=1}^n \rho_{\tau} \left( Y^i - \sum_{j=1}^d \theta^i_j \right),
\end{equation}
where the feasible set is defined as:
\beq
\label{space_K}
K^{(X,r)}(V) := \left\{ \theta : \theta = \sum_{j=1}^d \theta_j, \sum_{j=1}^d \left\| D^{(X_j, r)}_n S_j \theta_j \right\|_1 \leq V, \right. \\
\left. \mathbbm{1}^{\top} \theta_j = 0, \, j = 1, \dots, d \right\}.
\beq
The \( S_j \in \mathbb{R}^{n \times n} \) is a permutation matrix such that
\[
S_j \theta_j = \left( \theta^{(1)}_j, \theta^{(2)}_j, \ldots, \theta^{(n)}_j \right), \quad j = 1, \dots, d,
\]
where \( \theta^{(1)}_j < \dots < \theta^{(n)}_j \) among \( \theta_j = (\theta^1_j, \theta^2_j, \ldots, \theta^n_j) \).
The \( \mathbbm{1} \) is a vector with all entries equal to 1, ensuring that \( \mathbbm{1}^{\top} \theta_j = 0 \) for identifiability. The positive constant \( V \) controls the additive \( \ell_1 \) penalization. As shown by \cite{tibshirani2014adaptive} and \cite{sadhanala2019additive}, this regularization results in piecewise polynomial components of degree \( (r-1) \).

Compared to the constrained additive trend filtering formulation in \cite{sadhanala2019additive}, where Sub-Gaussian errors are assumed, the formulation in Equation~\eqref{eq303} is designed to handle models with heavy-tailed distributions, such as Cauchy and t-distributed errors, without assuming the existence of moments. This makes the problem formulation and analysis more challenging.

\subsection{Falling Factorial Representation}
The falling factorial representation offers an alternative method to express QATF estimates. As seen in \cite{tibshirani2014adaptive, wang2015trend}, this representation is useful for deriving rapid error rates in trend filtering.
\begin{definition} [Definition 2.1 in \cite{sadhanala2019additive}]
\label{f_f_r}

Given the knot points 
\((t^1 < \dots < t^n)\) 
where \( t^1 < \dots < t^n \in [0,1] \), we define the falling factorial basis as follows:
\vspace{-6pt}
\beq 
\label{f_e_q}
h_i^{(t)}(t) &= \prod_{l=1}^{i-1}(t - t^l), \quad i = 1, \dots, r+1, \\
h_{i+r+1}^{(t)}(t) &= \prod_{l=1}^r (t-t^{i+l}) \cdot \ind{t > t^{i+r}}, \\
\quad &i = 1, \dots, n-r-1.
\beq 
Then we define a linear subspace of function as 
$$
\gH_j  = \Bigl\{\sum_{i=1}^n \alpha_j^i h_i^{(X_j)}: \alpha_j^1, \dots, \alpha_j^n \in \R \Bigr\}, 
$$  where $\alpha_j^1, \dots, \alpha_j^n$ are coefficients for the $j$th dimension.
\end{definition}
The space \(\gH_j\) is defined as the span of the \( (r-1) \)th order falling factorial basis \( \{h_1^{(X_j)}, \dots, h_n^{(X_j)}\} \) over the \( j \)th dimension. These functions represent piecewise polynomial functions of order \( r-1 \).

If  \( f_j \in \gH_j \), then
$
f_j = \alpha_j^1 h_1^{(X_j)} + \alpha^2 h_2^{(X_j)} + \dots + \alpha_j^n h_n^{(X_j)}.
$
This results in a continuous piecewise polynomial function, with a global polynomial structure determined by \( \alpha_j^1, \dots, \alpha_j^{r+1} \). Notably, when \( \alpha_j^{i+r+1} \neq 0 \), \( f_j \) exhibits changes in its \( r \)th derivative at the knot \( t^{i+r} \), as well as in all lower-order derivatives.


Based on the Definition~\ref{f_f_r}, we will give an equivalent formulation of problem in Equation~\eqref{eq303}. Before we do that, we introduce an important definition. 
\begin{definition}[Total Variation]
  For a function \( f: [0,1] \rightarrow \mathbb{R} \), its total variation is defined as:
\begin{equation}
\label{eq_tv}
\begin{aligned}
TV(f) &= \sup\left\{ \sum_{i=1}^p |f(z_{i+1}) - f(z_i)| : \right. \\
& \left. z_1 < \dots < z_p, \ \text{is a partition of } [0,1] \right\}.
\end{aligned}
\end{equation}
For simplicity, we use \( TV^{(r)} \) to denote the total variation of the \( (r-1) \)th weak derivative of \( f \), i.e., \( TV^{(r)}(f) = TV(f^{(r-1)}) \).

Total Variation is discussed further in Appendix~\ref{app_discussion_TV}.

\end{definition}
Based on the relationship \( \theta^i_j = f_j(X^i_j) \) and the arguments above, the following problem is equivalent to the problem in Equation~\eqref{eq303}. Define the function class \( \gF^{(r)}(V) \) and the estimator \( \widehat{f}^{(r)}_{\gH} \) as follows:
\beq
\label{space_F}
\gF^{(r)}(V) \coloneqq \left \{ f = \sum_{j=1}^d f_j : f_j \in \gH_j, 
\sum_{j=1}^d TV(f_j^{(r-1)}) \leq V, \right. \\
\left. \sum_{i=1}^n f_j(X^i_j) = 0 \right\}, 
\beq

\begin{equation}
\label{eq298}
\min_{f \in \gF^{(r)}(V)} \sum_{i=1}^n \rho_{\tau}\left(Y^i - \sum_{j=1}^d f_j(X^i_j)\right). 
\end{equation}
Here, \( \gH_j \) is the span of the falling factorial basis over the \( j \)th dimension, defined in Equation~\eqref{f_e_q}, and \( \sum_{i=1}^n f_j(X^i_j) = 0 \) is required for identifiability.


\paragraph{Penalized Version of the Estimator.}

Using general Lagrange duality theory, the problem in Equation~\eqref{eq303} can be reformulated as a penalized version by introducing a positive tuning parameter \( \lambda \):
\beq  
\label{eq70}&
\min_{\substack{\theta_1, \dots, \theta_d \in \R^n}} \sum_{i=1}^n \rho_{\tau}(Y^i -  \sum_{j=1}^d \theta^i_j) + \lambda \sum_{j=1}^d  \norm{D^{(X_j,r)}_nS_j\theta_j}_1 \\
&\text{subject to}\quad \boldsymbol{1}^{\top} \theta_j =0, \;\;\; j=1,\ldots,d. 
\beq




For appropriate values of \( V \) and \( \lambda \), Equations~\eqref{eq303} and \eqref{eq70} are equivalent. For practical implementation, \eqref{eq70} will be employed in our experiments.


\section{MAIN RESULTS}
In this section, we present the main results of the paper concerning the theoretical guarantees of the QATF estimator. 
In our analysis, we adopt key assumptions that guide the theoretical foundation of the work. These are only describe them here at a high level, but are provided and explained in full mathematical detail in Section~\ref{app_assumptions_thm1} of the Appendix. 



First, we assume the input data lies within \([0,1]^d\), with the maximum gap between consecutive values controlled by \(O(\frac{\log n}{n})\), ensuring a well-behaved distribution for analysis. Remark~\ref{remark_data_domain} explains that this assumption is not restrictive.

Next, we assume the additive functions are centered, meaning their sum across input dimensions is zero, ensuring model identifiability. This follows standard practice in additive models.

Finally, we introduce an assumption for the conditional distribution function of the outcome variable, which guarantees reliable quantile estimates and is commonly applied in related statistical literature.

\subsection{ Results for QATF Estimator}

To begin, we introduce notation for a loss function,  $\Delta_n^2$. This loss function is not used for training; rather, it is used to quantify the quality of the estimators. The mathematical formulation of this loss function is as follows:
let $X^1, \dots, X^n$ be given data inputs, and consider a function $\delta: [0, 1]^d \rightarrow \mathbb{R}$.  Define $\Delta_n^2$ as a mapping from functions to real numbers
\beq 
\label{huber loss}
\Delta_n^2(\delta) = \frac{1}{n}\sum_{i=1}^n \min\set{\abs{\delta(X^i)}, \delta(X^i)^2},
\beq 
which, up to constants, is the Huber loss  \citep{huber1992robust}.  
 We then provide the risk bounds for estimators in Equation~\eqref{eq303} and \eqref{eq298}. However, the primary challenges encountered in our study arise from the use of quantile loss, as opposed to least squares loss, and the incorporation of an additive model with potentially growing dimension. Both parts present considerable difficulties in proving risk bounds.  
Thus, we give the bounds based on Huber-type loss in \eqref{huber loss}, as this loss naturally appears as a lower bound to the quantile population loss. See \eqref{bound_two_losses} in the Appendix for a more detailed explanation. For these reasons, the Huber loss is a natural loss function in quantile regression \citep{belloni2023high, madrid2022risk, ye2021non}. From a statistical perspective, the Huber loss provides insights into the convergence rate of the estimator toward the true value, as it captures some notion of the "average" error, albeit in a slightly weaker sense than the usual least squares loss.




\begin{theorem}
\label{theorem1}
Let $\set{Y^i}_{i=1}^n$ be any sequence of independent random variables which satisfies Assumption~\ref{assumptionA}
and let $f_0$ be as defined in \eqref{eq45}, where $f_{0}(X^i)$ represents the \(\tau\)-conditional quantile of \(Y^i\) given \(X^i\). Suppose Assumption \ref{assumption1} holds on the data inputs, and Assumption~\ref{assumption4} holds on $f_0$. Assume that the dimension $d$ of the input space is fixed, and that the underlying regression function is additive, \smash{$f_0=\sum_{j=1}^d f_{0j}$}. Here, $f_0 \in \gF^{(r)}(V)$ and the components $f_{0j}$, $j=1,\ldots,d$, are $r$ times weakly differentiable, with $\sum_{j=1}^d TV(f_{0j}^{(r-1)}) = V^*$. If $V$ is chosen such that $V \ge V^* = \sum_{j=1}^d TV(f_{0j}^{(r-1)})$, then 
\begin{equation}
\label{rate_fix_d}
\begin{aligned}
&\Delta_n^2(\widehat f_{\gH}^{(r)} - f_0) = \\  &\quad O_{pr} \Big( n^{-\frac{2r}{2r+1}} V^{\frac{2}{2r+1}}  \max\left\{1, V^{\frac{2r-1}{2r+1}} \right\} \Big).
\end{aligned}
\end{equation}
\end{theorem}
\vspace{-5pt}
\begin{remark}
Theorem~\ref{theorem1} extends Theorem 1 in \cite{sadhanala2019additive} to quantile regression, differing from their Theorem 1 by the inclusion of an additional factor: $\max\set{1, V^{(2r-1)/(2r+1)}}$, 
which can go to infinity if $V$ grows to infinity. However, under the canonical scaling $V^* = \gO(1)$ as mentioned in \cite{madrid2022risk, sadhanala2016total}, one can choose $V = \gO(1)$ as well, thus the above term is also $\gO(1)$. 
Theorem 3 in \cite{sadhanala2019additive} presents a lower bound, detailed in Equation (39). This lower bound aligns with our upper bound rate stated in Equation~\eqref{rate_fix_d} of our main manuscript, which holds for fixed dimension $d$, differing only by logarithmic factors. Hence, in the fixed dimension $d$, our bound is minimax optimal. Furthermore, the lower bound of their paper is for the Sub-Gaussian case, our upper bound is for arbitrary error terms that do not require moment conditions. 
\end{remark}

Theorem~\ref{theorem1} assumes that $\tau$ is fixed. For the cases where $\tau \rightarrow 0$ or $\tau \rightarrow 1$, see the discussion in Section~\ref{app_discussion_tau} of the Appendix.

\subsection{Error Bounds for A Growing Dimension $d$}
In this subsection, we allow the input dimension $d$ to grow with the sample size $n$. 
Furthermore, to achieve an error rate that scales linearly with the dimension \(d\), we assume the input points are i.i.d. from a continuous distribution on \([0,1]^d\) that decomposes into independent marginals. For more details, refer to Section~\ref{app_discussion_assumption6} of the Appendixs.

We now state our main result in the growing \( d \) case, whose proof can be found in Appendix \ref{app_lemmas_theorem2}, \ref{app_theorem2}.

\begin{theorem}
\label{theorem2}
Let $\set{Y^i}_{i=1}^n$ be any sequence of independent random variables which satisfies Assumption~\ref{assumptionA}
and let $f_0$ be as defined in \eqref{eq45}, where $f_{0}(X^i)$ represents the \(\tau\)-conditional quantile of \(Y^i\) given \(X^i\).
Let $\set{X^i}_{i=1}^n$ be the input points which satisfy Assumption~\ref{assumption1} and \ref{assumption6}. Also, suppose that $f_0$ satisfies  Assumption~\ref{assumption4}. Assume the underlying regression function is additive, \smash{$f_0=\sum_{j=1}^d f_{0j}$}, where $f_0 \in \gF^{(r)}(V)$. The components $f_{0j}$, $j=1,\ldots,d$, are $k$ times weakly differentiable, and $\sum_{j=1}^d TV(f_{0j}^{(r-1)}) = V^*$. If $V$ is chosen such that $V \ge V^* = \sum_{j=1}^d TV(f_{0j}^{(r-1)})$, then
\begin{equation}
\label{rate_fix_d}
\begin{aligned}
\Delta_n^2(\widehat f_{\gH}^{(r)} - f_0) = 
& O_{pr} \Big( d^{\frac{2r+2}{2r+1}} n^{-\frac{2r}{2r+1}} V^{\frac{2}{2r+1}} \\
& \cdot \max\left\{ 1, V^{\frac{2r-1}{2r+1}} \right\} \Big).
\end{aligned}
\end{equation}
\end{theorem}
\begin{remark}
Theorem~\ref{theorem2} generalizes Theorem 2 in \cite{sadhanala2019additive} to the quantile regression setting, with a difference that our upper bound contains an extra term.
This is the factor 
$\max\set{1, V^{(2r-1)/(2r+1)}},$  
which can go to infinity if $V$ grows to infinity. However, under the canonical scaling $V^* = \gO(1)$ one can choose $V = \gO(1)$ as well and thus the above term is also $\gO(1)$. The factor $d^{(2r+2)/(2r+1)}$ compared to Theorem 2 in \cite{sadhanala2019additive} is $d$, and for large $r$, these two terms tend to be the same, making our rate essentially minimax up to log factors.  
\end{remark}

\section{BACKFITTING ALGORITHM}
In this section, we describe the algorithm for fitting our method. The algorithm is based on the idea of backfitting, which was introduced in \cite{hardle1993backfitting} for additive models. Two algorithms closely related to backfitting for additive models are the alternating least squares and alternating conditional expectations methods, as mentioned in \cite{van1983non} and \cite{breiman1985estimating}. Other works that have explored backfitting idea include \cite{buja1989linear, nielsen2005smooth, ravikumar2009sparse,petersen2016fused,tibshirani2017dykstra}.

To solve the original problem~\eqref{eq70}, we use the backfitting approach described in Algorithm~\ref{alg-backfitting}. Specifically, the algorithm cycles over $j = 1, \dots, d$, and at each step updates the estimate for component $j$ by applying univariate quantile trend filtering to the $j$th
partial residual (i.e., the current residual excluding component j). The univariate quantile trend filtering utilizes the functions described in  \cite{brantley2020baseline}.

\raggedbottom 




\begin{minipage}{0.85\textwidth}
    \begin{algorithm}[H]
       \caption{Backfitting for quantile additive trend filtering}
       \label{alg-backfitting}
       \begin{algorithmic}[1]
           \STATE {\bfseries } \textbf{Inputs:} Responses $Y^i \in \R$ and input points $X^i \in \R^d$, $i = 1, \dots, n$, and set a value for $\lambda>0$.
           \STATE Set $t=0$ and initialize $\theta_j^{(0)} = 0$, for $j = 1, \dots, d$.
           \FOR{$t = 1, 2, 3, \dots$ (until convergence)}
         \FOR{$j=1,2,\cdots,d$}
                   \STATE (i) Set response $u_j^i$, $u_j^i = Y^i -  \sum_{l<j} \theta^{i(t)}_l - \sum_{l>j} \theta^{i(t-1)}_l$, for $i = 1, \dots, n$.
                   \STATE (ii) $ \theta_j^{(t)} = \underset{\theta_j}{\arg\min} \sum_{i=1}^n \rho_{\tau}(u_j^i - \theta^i_j) +  \lambda \norm{D^{(X_j,r)}_nS_j\theta_j}_1, \text{subject to }  \mathbbm{1}^{\top} \theta_j = 0.$
               \ENDFOR
           \ENDFOR
           \STATE Return $\widehat \theta_j$, $j = 1, \dots, d$ (parameters $\theta_j^{(t)}$ at convergence).
       \end{algorithmic}
    \end{algorithm}
  \end{minipage}

\subsection{Backfitting with ADMM}
The main computational burden of Algorithm~\ref{alg-backfitting} is Line 6. We tackle this by using the alternating directions method of multipliers (ADMM)  algorihtm, see  \cite{boyd2011distributed}, as follows.  
We let 
$ 
u_j^i = Y^i -  \sum_{l<j} \theta^i_l - \sum_{l>j} \theta^i_l,
$ 
then, we solve 
\beq 
\label{73}
&\min_{\theta_j \in \R^n } \sum_{i=1}^n  \rho_{\tau}(u_j^i - \theta^i_j) + \lambda \norm{D^{(X_j,r)}_nS_j\theta_j}_1, \\
&\text{subject to}\quad \boldsymbol{1}^{\top} \theta_j =0.
\beq 

How to solve \eqref{73} is given in Section~\ref{app_backfitting_update} of Appendix.

Algorithm~\ref{alg-backfitting} is equivalent to block coordinate descent (BCD), also called exact blockwise minimization,
applied to Problem \eqref{eq70} over the coordinate blocks $\theta_j$, $j = 1,\dots, d$.
To summarize the broader theoretical framework, a general treatment of the Block Coordinate Descent (BCD) method is given by \cite{tseng2001convergence}. Since Equation \eqref{eq70}
decomposes into smooth plus separable terms, which satisfies a convex criterion of BCD, see Appendix \cref{nonconvex0} for details. Then it immediately holds that the iterates of our algorithm converges to the minimal point of the objection function. 
Hence, our algorithm for quantile additive trend filtering is guaranteed to find the minimizer.

\subsection{Computational Complexity}
\label{complexity}


In Algorithm~\ref{alg-backfitting}, line 3, the outer loop usually takes $20$ iterations to converge. 
In each iteration, the backfitting process is conducted, involving $d$ distinct fits of quantile trend filtering. For each fit of quantile trend filtering, the ADMM procedure in Section~\ref{app_backfitting_update} of the Appendix requires three updates (primal, dual, and slack). In the worst case, each update can be computed in \( O(n) \) time (treating \( r \) as a constant), leading to an overall ADMM complexity of \( O(nm) \), where \( m \) denotes the maximum number of ADMM iterations. 
There are $d$ components, therefore one full iteration of
standard backfitting ADMM updates can be done in linear time $O(dnm)$. In practice, we find $m$ tends to be small, typically around 10. Thus, the overall complexity of Algorithm~\ref{alg-backfitting} amounts to $O(dn)$. The details about the calculation of $O(dn)$ can be found in Section~\ref{app_backfitting_update} of Appendix.  The software and hardware setting for Algorithm \ref{alg-backfitting}  can be found in Appendix~\ref{computation_setting}.




\subsection{Examples}
We conclude this section with two examples, visualized in Figure~\ref{fig:4} and Figure~\ref{fig:3}, illustrating the performance of our quantile additive trend filtering algorithm. Full details are given in Section~\ref{app_algorithm_example} of the Appendix. 

\begin{figure}[htb]
    \setlength{\abovecaptionskip}{-15pt} 
    \centering
    \begin{subfigure}[b]{1\textwidth}  
        \includegraphics[width=\linewidth, height=1.5in]{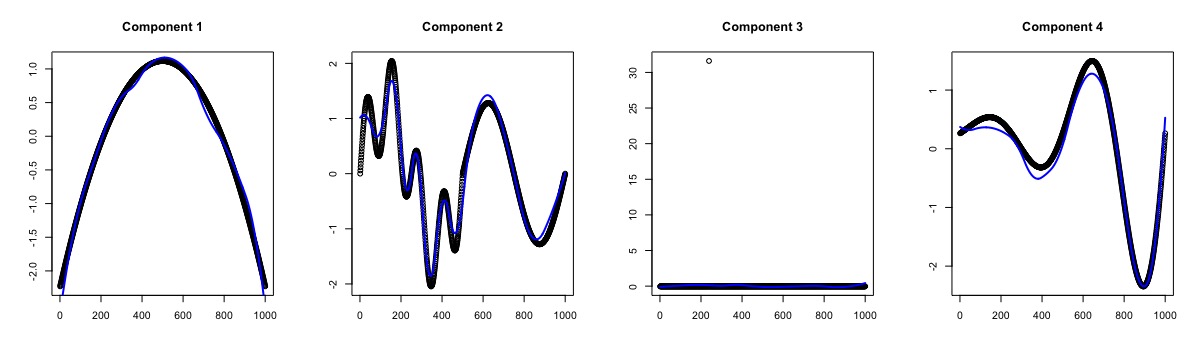}
        \label{fig:4a}
    \end{subfigure}
    \caption[]{Component estimates of Quantile Additive Trend Filtering with $\tau = 0.5$ are plotted in blue, and the true component functions $f_{0j}(x) = a_jg_j(x) - b_j$ in black, with $a_j$ and $b_j$ chosen such that $f_{0j}$ has an empirical mean of zero and empirical norm $\|f_{0j}\|_n = 1$. For this scenario, $g_1(x) = \frac{1}{2}x^2 $, $g_2(x) = \frac{3}{2} \sin(4 \pi x) + \mathbbm{1}_{x \leq \frac{1}{2}} \cdot \sin(16 \pi x)$, $g_3$ is a dummy dimension (where only 1 randomly assigned point takes non-zero value), and $g_4(x) = e^{3 x} \sin(4 \pi x)$. }
    \label{fig:4}
\end{figure}


\begin{figure}[htb]  
    \captionsetup[subfigure]{aboveskip=3pt,belowskip=-0pt,font=small}
    \centering
    \begin{subfigure}[b]{0.3\linewidth}  
        \includegraphics[width=\linewidth,height=1.5in]{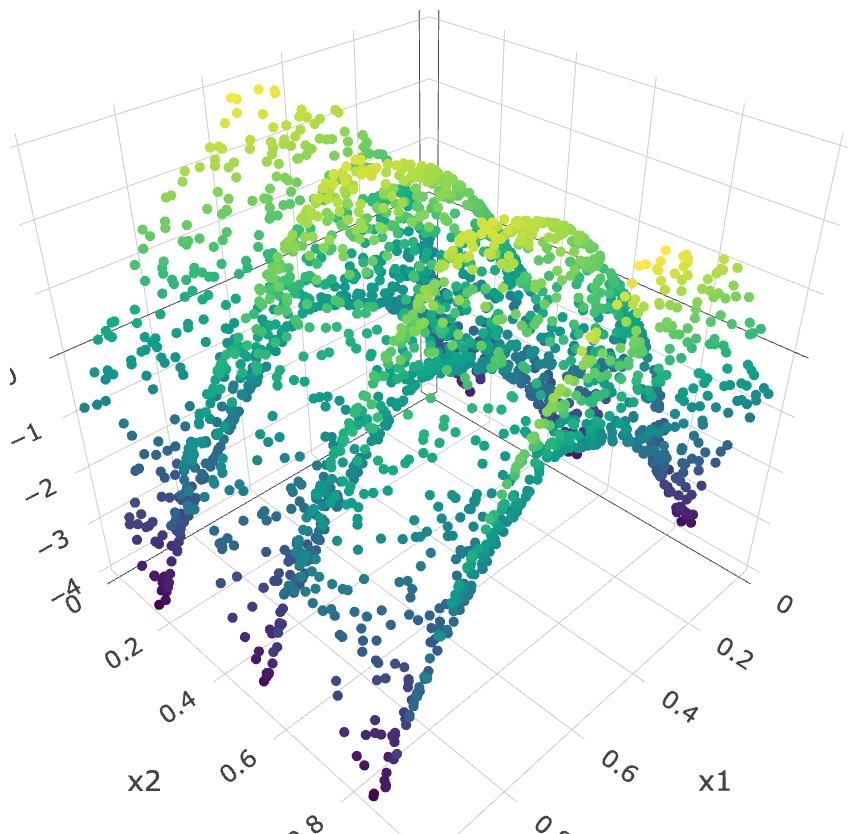}
        \caption{ $\tau = 0.5$, True Signal}
        \label{fig:3a}
    \end{subfigure}  
    \begin{subfigure}[b]{0.3\linewidth}  
        \includegraphics[width=\linewidth,height=1.5in]{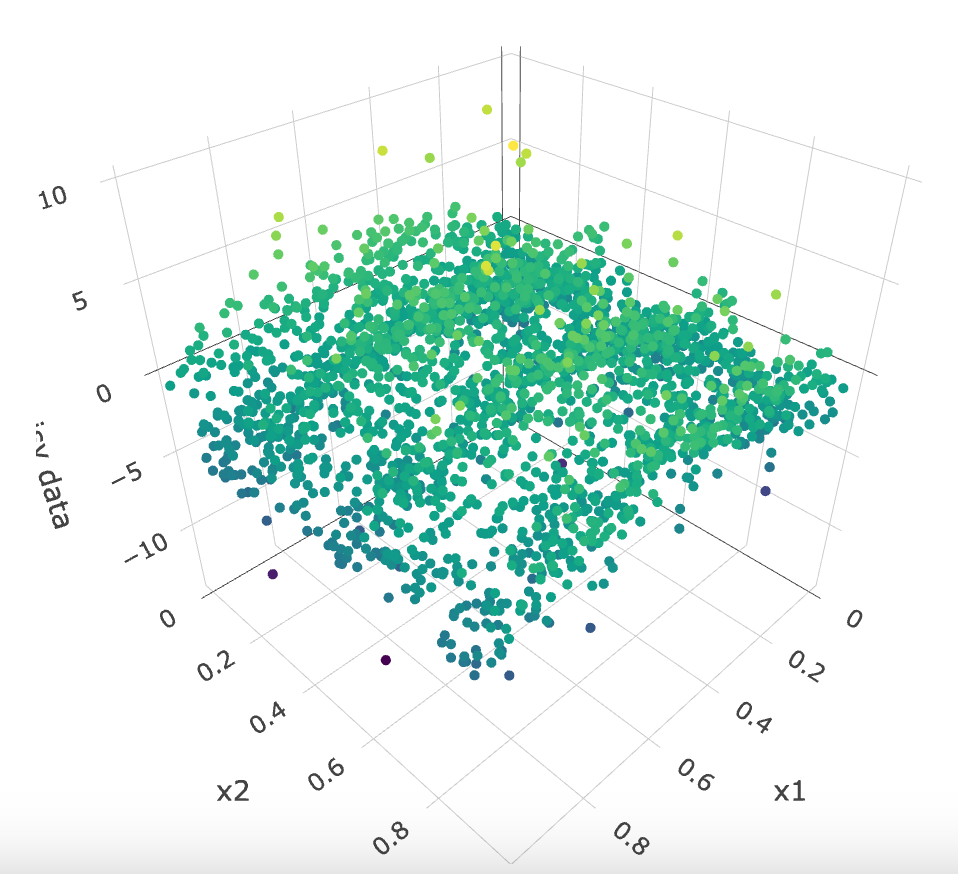}
        \caption{Noisy Data}
        \label{fig:3b}
    \end{subfigure}
    \begin{subfigure}[b]{0.3\linewidth}  
        \includegraphics[width=\linewidth,height=1.5in]{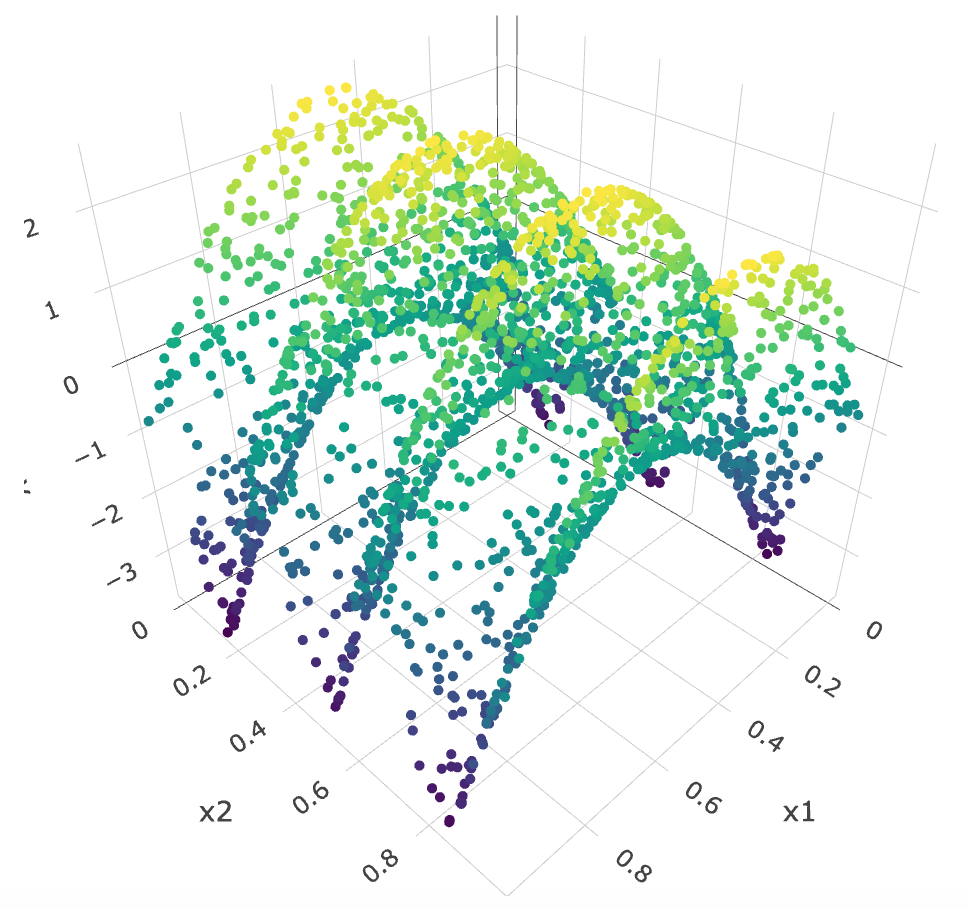}
        \caption{ Reconstructed Signal}
        \label{fig:3c}
    \end{subfigure}
    \caption[]{Figure \subref{fig:3a} plots true function $f_{0j}(x) = a_jg_j(x) - b_j$ on $0.5$ quantile, where $g_1(x) = \frac{1}{2} \cos(6 \pi x) + 0.1)$ and $g_{2}(x) = -(x - \frac{1}{2})^2$. Figure \subref{fig:3b} plots the real data with heavy tailed noise added. Figure \subref{fig:3c} plots QATF estimators $\widehat f$ on this Scenario, for a reconstruction of the true signal. }
    \label{fig:3}
    \vspace{-10pt}
\end{figure}


\section{SIMULATED EXPERIMENTS}
\label{experiment}

In this section, we conduct a comparison study to evaluate the performance of quantile additive trend filtering (QATF) of orders $r = 1$ and $r = 0$, denoted as QATF1 and QATF0 respectively, against existing benchmark methods. 

\subsection{Benchmark Methods}
We compared against the usual mean regression penalized additive trend filtering models of orders $r = 1$ and $r = 0$ (ATF1 and ATF0), as well as quantile smoothing splines (QS) \citep{koenker1994quantile}. Both ATF and QS can be implemented with minor modifications to step (ii) in Algorithm~\ref{alg-backfitting}. Specifically, ATF is implemented by replacing the check function with the least squares loss, following the approach of \cite{sadhanala2019additive}, and QS simply substitutes the trend filtering penalty with the smoothing spline penalty. We implement the univariate fits using the R packages \texttt{trendfilter} \citep{tibshirani2011solution} for ATF, and \texttt{fields} \citep{nychka2017fields} for QS. 

\subsection{Experimental Setup}
We sample $n \in \{500, 1000, 2500\}$ points in $d = 10$ dimensions, assigning inputs $X_j = (X^1_j \dots X^n_j)$, where $X_j^i$ is sampled from a $d$-dimensional uniform distribution within the interval $[0,1]^d$. 

The component functions are defined as $f_{0j} = a_jg_j - b_j$ for $j = 1, \dots, d$. For $g_j(x)$, we use sinusoids with Doppler-like spatially varying frequencies:
\begin{equation}
\label{dropper}
g_j(x) = \sin\left(\frac{2\pi}{(x+0.1)^{j/10}}\right), \quad j = 1, \dots, 10.
\end{equation}
The true data is then generated  as:
\[
Y^i = \sum_{j=1}^d f_{0j}(X^i_j) + \epsilon^i, \quad i = 1, \dots, n.
\]
This structure allows for significant heterogeneity both within and between the component functions. The errors $\{\epsilon^i\}_{i=1}^n$ are independent with $\epsilon^i \sim F^i$, varying across different scenarios.

For feature selection, we consider values of $\lambda$ such that $\log_{10}(\lambda)$ is in a grid of 50 evenly spaced points in $[-2 + r, 4 + r]$ for QATF0 and QATF1, $[-7, 3]$ for ATF0 and ATF1, and $[-16, 0]$ for QS. We then choose the $\lambda$ to be the value that minimizes the estimator $\widehat f$'s mean squared error (MSE), $\frac{1}{n}\sum_{i=1}^n(\widehat f - f_0)^2$, with $f_0$ representing the true vector of quantiles. The grid size was chosen to experimentally encompass all variability in the MSE resulting from changes in $\lambda$, and we chose to use MSE for feature selection in order to ensure fair competition with other methods in the comparison experiments. Feature selection via Cross-Validation is used for the experiments on real data in Section 6.

\begin{table}[ht]
\centering
\caption {Average mean squared error, $\frac{1}{n}\sum_{i=1}^n(f^i_0 - \widehat f^i)^2$, averaging over 50 Monte Carlo simulations for the different methods considered. The best MSE is listed with bold text. S = Scenario.}
\begin{tabular}{rrrrrrrr}
\hline
n    & S & $\tau$ & QATF1                & QATF0                & QS                   & ATF1                 & ATF0               \\
\hline 
500  & 1        & 0.5 & 0.9831          & 1.0136          & 1.0738          & \textbf{0.5800} & 0.6194        \\
1000 & 1        & 0.5 & 0.6132          & 0.7426         & 0.7304           & \textbf{0.4011} & 0.4583        \\
2500 & 1        & 0.5 & 0.3111          & 0.4096          & 0.4325             & \textbf{0.2055}  & 0.2751        \\
\hline
500  & 2        & 0.5 & \textbf{2.0341} & 2.3507          & 2.2444          & 165.2800        & 2794.1800      \\
1000 & 2        & 0.5 & \textbf{1.3132}  & 1.5889          & 1.5223          & 1412.6500       & 78147.2000    \\
2500 & 2        & 0.5 & \textbf{0.6243} & 0.7877          & 0.8078	          & 33.5683        & 2623.1000      \\
\hline 
500 & 3        & 0.5 & \textbf{1.1926} & 1.5098          & 1.3299          & 4.1947          & 4.0913        \\
1000  & 3        & 0.5 & \textbf{0.7359} & 0.9847                & 0.8901                &      3.9147   &   3.8723      \\ 
2500 & 3        & 0.5 & \textbf{0.3468} & 0.4963          & 0.4965         & 3.6066          & 3.4550       \\
\hline 
500  & 4        & 0.5 & \textbf{0.9751} & 1.1981                & 1.0976                & 3.5949          & 5.4504        \\
1000 & 4        & 0.5 & \textbf{0.5600} & 0.7544          & 0.6707          & 1.0105          & 1.1889        \\
2500 & 4        & 0.5 & \textbf{0.2391} & 0.3647                & 0.3604                & 0.6629          & 0.9073       \\ 
500  & 4         & 0.8 & \textbf{1.7977} & 2.4720           & 1.9304           & NA          & NA        \\
1000 & 4         & 0.8 & \textbf{1.0583} & 1.4232          & 1.2408          & NA          & NA        \\
2500 & 4         & 0.8 & \textbf{0.5405} & 0.7116          & 0.6917         & NA          & NA        \\
500  & 4         & 0.2 & \textbf{1.7558} & 2.3276          & 1.7912          & NA          & NA        \\
1000 & 4         & 0.2 & \textbf{1.0611} & 1.4834          & 1.2454          & NA          & NA        \\
2500 & 4         & 0.2 & \textbf{0.5364} & 0.7019          & 0.6888         & NA          & NA         \\ \hline 
500  & 5        & 0.5 & 1.6479 & 1.6549          & \textbf{1.6190}          & 1.8772          & 1.6511        \\
1000 & 5        & 0.5 & 1.0925 & \textbf{0.8612}          & 1.1024          & 1.2632          & 0.9551        \\
2500 & 5        & 0.5 & 0.6382 & \textbf{0.3466}          & 0.6963         & 0.7755          & 0.4624        \\
\hline 
500  & 6         & 0.5    & 0.6457                     & 1.0902                     & 0.9807                  & \textbf{0.6412}                & 0.6713                                   \\ 
1000 & 6         & 0.5    & \textbf{0.3373}                 & 0.7372                     & 0.6635                  & 0.3586                    & 0.4407                                    \\ 
2500 & 6         & 0.5    & 0.2175                     & 0.4265                     & 0.4308                  & \textbf{0.2078}                & 0.2885                                     \\ \hline
500  & 7         & 0.5    & \textbf{1.9526}                 & 2.0708                     & 2.0007                  & 3.5986                    & 3.5436                                    \\ 
1000 & 7         & 0.5    & \textbf{1.3192}                 & 1.5304                     & 1.4419                  & 2.0597                    & 2.3514                                   \\ 
2500 & 7         & 0.5    & \textbf{0.7164}                 & 0.8174                     & 0.8542                  & 2.9768                    & 4.3687                    \\ \hline

\end{tabular}
\vspace{-2pt}
\label{tbexp}
\end{table}

\subsection{Scenarios and Results}
We examine the performance of QATF and benchmark methods across different scenarios:\\
\textbf{Scenario 1: Normal Errors} \leavevmode \\
In this scenario, the distributions $F^i$ are taken as $N(0, 1)$. Given the normal errors and varying smoothness, this scenario is ideally suited for ATF. As shown in Table~\ref{tbexp}, ATF1 performs the best, with ATF0 closely following. QATF1 and QATF0 both outperform QS, likely due to splines' limitations in adapting to heterogeneous smoothness.\\
\textbf{Scenario 2: Cauchy Errors} \leavevmode \\
Here, the distributions $F^i$ are taken as $\text{Cauchy}(0, 1)$, where the errors have no mean. As expected, ATF1 and ATF0 perform poorly in this scenario. Conversely, the quantile methods show robustness, with, consistent with the results from Scenario 1, QATF1 and QATF0 slightly outperforming QS.\\ 
\textbf{Scenario 3: Log-Normal Errors} \leavevmode \\
For this scenario, the distributions $F^i$ are taken as $\text{log-normal}(0, 1)$, which is both right-skewed and heavy-tailed. In real data applications, this is another typical example of where estimating the median is more useful than estimating the mean. We see that the quantile methods perform well at this task, with QATF1 performing the best.  \\
\textbf{Scenario 4: Heteroscedastic $t$ Errors} \leavevmode \\
In this scenario, the distributions $F^i$ are taken as $t(2)$, but with 
$\epsilon^i = i^{1/2}/n^{1/2}v^i$, where $v^i$s are independent draws from $F^i$. The performance of the methods is similar to Scenario 2, but less extreme. We also utilize this scenario to demonstrate fits on other quantiles. This is only possible with a quantile fitting method, so ATF1 and ATF0 are represented by NAs in Table \ref{tab:test_metrics}. \\
\textbf{Scenario 5: Piecewise Constant Components, $t$ Errors} \leavevmode \\
For this scenario, we redefine the function $g_j(x)$ as:
\[
g_j(x) = 
\begin{cases} 
1 & \text{if } x \in [b_1, b_2) \\
-1 & \text{if } x \in [b_2, b_3) \\
\vdots & \vdots \\
(-1)^{j} & \text{if } x \in [b_{j+1}, b_{j+2})
\end{cases}, 
\]
where $b_1 = 0$, $b_{j+2} = 1$, and $b_2 \dots b_{j+1}$ are squares of $j$ evenly spaced breakpoints in $(0, 1)$. This maintains heterogeneity within and between components as in previous scenarios, but introduces a piecewise-constant structure. The distributions $F^i$ are taken as $t(3)$. This is a scenario where we expect the k = 0 models to perform the best, which is indeed the case; QATF0 is the premier method, with ATF0 outperforming QATF1 and QS for second place. \\
\textbf{Scenario 6: Time Series, Normal Errors} \leavevmode \\
For this scenario, we take the function as in \eqref{dropper},  
but now we allow series correlation. We have an autoregressive model the error $\epsilon^i$ depends on the previous $\epsilon^{i-1}$. Specifically, for $\tau = 0.5$, we consider.
\begin{align*}
\epsilon^i &= \frac{0.5\epsilon^{i-1} + v^i}{0.5^2 + 1^2},
v^i \sim F^i, \text{ where } F^i \text{ as } N(0,1).
\end{align*}
With the normal errors, ATF1 performs the best, though by only a slim margin over QATF1. \\
\textbf{Scenario 7: Time Series, $t$ Errors} \leavevmode \\
For this, we consider the same structure as Scenario 6, but with \( F^i \) as \( t(2) \). As expected, in the transition to heavy tails, the quantile methods are more effective, with QATF1 demonstrating the strongest performance. 

Overall, the QATF estimator performs well across all scenarios, and particularly outperforms ATF under heavy-tailed errors, thereby supporting our theoretical findings. We also provide simulation results for $\tau = 0.2$ and 0.8 for Scenarios 1, 2, 3, 5, and 6 in the Table~\ref{table2_more_result} of Appendix. We only consider \( \tau = 0.5 \) for Scenario 7 due to the difficulties in constructing the ground truth data when \( \tau \neq 0.5 \)

\section{WORLD HAPPINESS DATA}
\label{app_real_data}

In this section, we demonstrate the effectiveness of QATF in real data applications by performing statistical inference on the 2024 World Happiness Data. We analyze the conditional relationships between a country-level happiness index, derived from average responses to the Cantril life ladder question in the Gallup World Polls \cite{cantril1965pattern,helliwell2024world}, and nine predictors, such as gross national income, sourced from the 2023 World Happiness Report and World Bank Development Indicators \cite{worldbank2012world}.

Our dataset choice is inspired by \cite{petersen2016fused}, who compared the fused lasso additive model (FLAM) to generalized additive models (GAM) \citep{hastie2013gam} using 2013 Gallup World Polls data with twelve predictors. A key benefit of our quantile method is the ability to construct prediction intervals and assess coverage without domain-specific knowledge.

To assess QATF, we conduct 10-fold cross-validation on 113 countries with complete data, split into training and testing sets. We use $r = 0$, making this an even more direct extension of \cite{petersen2016fused} to the quantile case, as FLAM is equivalent to Additive Trend Filtering when $r = 0$. We evaluate coverage for 60\%, 80\%, and 90\% prediction intervals, deeming a prediction covered if it lies between the estimated quantiles, akin to swapping the lower and upper quantile estimates in the cases where they intersect \cite{chernozhukov2010quantile, vardi2000multivariate}. For example, with a 90\% prediction interval, coverage ($\alpha$) is:
\beq 
\label{coverage}
\alpha = \frac{1}{n} \sum_{i=1}^n \ind{\hat{q}^i_{0.05} \leq Y^i_{\text{test}} \leq \hat{q}^i_{0.95}},
\beq 
where $\hat{q}^i_{0.05}$ and $\hat{q}^i_{0.95}$ are predicted quantile for data $Y^i_{\text{test}}$ based on $X^i_{\text{test}}$. Tuning parameters for the 0.05, 0.1, 0.2, 0.5, 0.8, 0.9, and 0.95 quantiles are selected via 10-fold cross-validation. Results in Table \ref{tab:test_metrics} demonstrate QATF’s superior coverage compared to quantile smoothing splines.






\begin{table}[h]
    \centering
    \caption{Coverage $\alpha$ of various prediction intervals, averaged over 10 train-test splits.}
    \begin{tabular}{@{}lcc@{}}
        \toprule
        Interval Width & QATF & QS \\ 
        \midrule
        90\% &  0.894 & 0.798 \\ 
        80\% & 0.790 & 0.716 \\ 
        60\% & 0.634 & 0.602 \\ 
        \bottomrule
    \end{tabular}
    \label{tab:test_metrics}
\end{table}

Finally, we examine the component-wise results, demonstrating the interpretability advantages of additive models. The 0.1, 0.5, and 0.9 quantiles are trained on the entire dataset, and the component-wise results are plotted in Appendix \ref{sec:component_plots}, Figure \ref{fig:component_plots}. With these individual component results, we can see, for example, that conditional on the other predictors, our model finds that Log-GDP per Capita is associated with increased happiness index scores (which is consistent with the findings in \cite{petersen2016fused}). Figure \ref{fig:6} displays this association. When considering the quantiles specifically, we see that construction of prediction intervals relies almost exclusively on the Log-GDP per Capita and Perceptions of Corruption predictor variables, since the other variables demonstrate little to no difference in their component fits between the 0.1 and 0.9 quantiles.

\begin{figure}[htbp]
    \centering
    \includegraphics[scale=0.5]{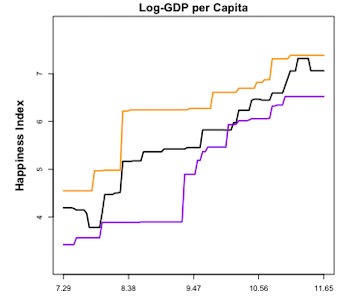}
    \caption{Log-GDP per Capita component fit in Quantile Additive Trend Filtering with $\tau = 0.1, 0.5, \text{and } 0.9$, plotted in purple, black, and orange, respectively.}
    \label{fig:6}
\end{figure}

\section{CONCLUSION}

In total, we analyzed risk bounds for quantile additive trend filtering in both fixed and growing input dimensions. We also proposed and validated a practical algorithm using dimension-wise backfitting, effectively addressing challenges in additive models with heavy-tailed distributions.


Future work could extend this method to high-dimensional settings, incorporating sparsity penalties for enhanced performance when input dimensions are comparable to or larger than the sample size.

\clearpage
\bibliographystyle{apalike}
\bibliography{references}

\clearpage 
\section*{Checklist}



 \begin{enumerate}

 \item For all models and algorithms presented, check if you include:
 \begin{enumerate}
   \item A clear description of the mathematical setting, assumptions, algorithm, and/or model. [Yes]
   \item An analysis of the properties and complexity (time, space, sample size) of any algorithm. [Yes]
   \item (Optional) Anonymized source code, with specification of all dependencies, including external libraries. [Yes/No/Not Applicable]
 \end{enumerate}

 \item For any theoretical claim, check if you include:
 \begin{enumerate}
   \item Statements of the full set of assumptions of all theoretical results. [Yes]
   \item Complete proofs of all theoretical results. [Yes]
   \item Clear explanations of any assumptions. [Yes]     
 \end{enumerate}

 \item For all figures and tables that present empirical results, check if you include:
 \begin{enumerate}
   \item The code, data, and instructions needed to reproduce the main experimental results (either in the supplemental material or as a URL). [Yes]
   \item All the training details (e.g., data splits, hyperparameters, how they were chosen). [Yes]
         \item A clear definition of the specific measure or statistics and error bars (e.g., with respect to the random seed after running experiments multiple times). [Yes]
         \item A description of the computing infrastructure used. (e.g., type of GPUs, internal cluster, or cloud provider). [Yes]
 \end{enumerate}

 \item If you are using existing assets (e.g., code, data, models) or curating/releasing new assets, check if you include:
 \begin{enumerate}
   \item Citations of the creator If your work uses existing assets. [Yes]
   \item The license information of the assets, if applicable. [Yes]
   \item New assets either in the supplemental material or as a URL, if applicable. [Yes]
   \item Information about consent from data providers/curators. [Not Applicable]
   \item Discussion of sensible content if applicable, e.g., personally identifiable information or offensive content. [Not Applicable]
 \end{enumerate}

 \item If you used crowdsourcing or conducted research with human subjects, check if you include:
 \begin{enumerate}
   \item The full text of instructions given to participants and screenshots. [Not Applicable]
   \item Descriptions of potential participant risks, with links to Institutional Review Board (IRB) approvals if applicable. [Not Applicable]
   \item The estimated hourly wage paid to participants and the total amount spent on participant compensation. [Not Applicable]
 \end{enumerate}

 \end{enumerate}

\newpage

\onecolumn
\begin{center}
    \textbf{\Large Appendix}
\end{center}

\appendix
\section{BROADER IMPACT AND ETHICAL STATEMENT}

This paper contributes to the ongoing development of Machine Learning, adhering to standard ethical guidelines in research. Our work, aligning with common advancements in the field, does not present any unique ethical dilemmas or societal consequences that require special emphasis. We recognize the importance of responsible use and development of Machine Learning technologies and their potential impact on society. However, our research does not delve into areas that might raise significant ethical or societal concerns. We commit to ethical research practices and consider the broader implications of our work to be aligned with the typical advancements in Machine Learning.
\section{CONTRIBUTION OF PROOF TECHNIQUES}
\label{app_proof_contribution}
Our proof process establishes the Huber type loss (defined in \eqref{huber loss}) is upper bounded by the localized Rademacher width, as delineated in Lemmas~\ref{lemma12} to Theorem~\ref{theorem8}.  We then establish the upper bound of local Rademacher width in terms of Huber type metric by the local Rademacher width in terms of $\ell_2$ type metric, which needs us 
to convert the $\ell_2$ norm in local Rademacher width to the Huber loss function in its local Rademacher width. This conversion is detailed in \ref{definition4} to \ref{lemma:boudingT2}, with specific emphasis on Lemmas~\ref{lemma18}, \ref{lemma20}, \ref{lemma21}, \ref{lemma22}, and is culminated in Equation~\eqref{eq1172} of Lemma~\ref{lemma:boudingT2}. In addition,  our proof executes this conversion without amplifying the dimensionality. In particular, this is achieved by orthogonally projecting the estimation error into the null space of the trend filtering operator and its orthogonal complement, while controlling for dimensionality, as demonstrated in Lemmas ~\ref{proposition2growingd}, \ref{lemma23growingd}, \ref{lemma30}, and \ref{lemma29}.  As a result, we maintain a minimax rate for arbitrary errors, irrespective of moment conditions. This part of our proof technique represents one of the significant technical contributions of this work.

\section{PROOF SKETCH}
Below we summarize our proofs for Theorems \ref{theorem1} and \ref{theorem2}, with the detailed proofs provided in the Appendix. 

\textbf{Step 1}. The first step in our analysis is to show that the convergence rate of our estimator depends on a local Rademacher complexity, which involves a constraint enforced by the total variation and one constraint consisting of a ball of the form \(\{ \delta : \Delta_n^2(\delta) \leq r\}\). 

\textbf{Step 2}. The geometry of the set \(\{ \delta : \Delta_n^2(\delta) \leq r\}\) is extremely complicated because the function \(\Delta_n(\cdot)\) does not satisfy the triangle inequality. To deal with this, we use the formula \(\|\delta\|_2^2 \leq \max\{\|\delta\|_{\infty}, 1\}\Delta_n^2(\delta)\). The idea behind this is to replace the set \(\{ \delta : \Delta_n^2(\delta) \leq r\}\) with a set of the form \(\{ \delta : \|\delta\|_2^2 \leq \tilde{r}\}\) for an appropriate \(\tilde{r} > 0\). However, to do this we must control \(\|\delta\|_{\infty}\) for \(\delta\) in the set \(\{\delta\in \mathcal{F}^{(r)}(V) - f_0: \Delta_n^2(\delta) \leq t^2\}\). This in itself is challenging, and it is done by orthogonally projecting the vectors \(\delta(X)\) into the null space of the trend filtering operator resulting in \(p_j\) and projecting onto the orthogonal complement leading to \(q_j\), where $j=1,\dots, d$. This has to be done for each dimension $j$. 

\textbf{Step 3}. Controlling the \(\ell_{\infty}\) of the projections onto both spaces is nontrivial, but we achieve this in Lemmas ~\ref{lemma18}, \ref{lemma19} and \ref{lemma21}. The next step is to obtain bounds on  
\(\|p\|_{\infty}\) and \(\|q\|_{\infty}\), where $p = p_1 + \dots + p_d$ and $q = q_1 + \dots + q_d$, then obtain the \(\Delta_n^2(p)\) and \(\Delta_n^2(q)\) afterwards. Lemmas ~\ref{proposition2}, \ref{lemma22}, and \ref{lemma:boudingT2} derive these steps. The higher-order term \(\Delta_n^2(q)\) would need conversion to \(\ell_2(q)\) to obtain the optimal rate. Lemma \ref{lemma:boudingT2} deals with that.

\section{ASSUMPTIONS FOR THEOREM~\ref{theorem1}}
\label{app_assumptions_thm1}
Before providing those main results, we give some detail description of assumptions as needed.



\begin{assumption}
\label{assumption1}
    Assume the inputs are contained in $[0,1]$ in each dimension such that 
    \beq 
    0\le X^i_j \le 1,  \quad i = 1, \dots, n, \quad j = 1, \dots, d. 
    \beq  Then we 
assume the maximum width $w$ for any consecutive $X^i_j$ satisfies, with high probability, that $w  = O(\frac{\log n}{n})$, where $w$ is given by \eqref{w_assump},  
\begin{equation}
\label{w_assump}
w = \max_{\substack{i=1,\dots, n-1\\ j = 1, \dots, d}} \abs{X^{(i)}_j - X^{(i-1)}_j},
\end{equation} where $X^{(1)}_j < X^{(2)}_j < \cdots < X^{(n)}_j$.
\end{assumption}

\begin{remark}
\label{remark_data_domain}
The $[0,1]$ assumption is minor, as we can always scale the inputs.  The same constraint is used in the analysis in \cite{sadhanala2019additive}. The
assumption for $w$ would be satisfied with probability approaching 1 by Lemma 5 in \cite{wang2015trend}, for $X^1_j, \dots, X^n_j$ that are i.i.d. from a 
uniform distribution $[0,1]$. In the literature of nonparametric statistics, it is common to assume covariates $X^i \in \R^d$ that are uniformly distributed on $[0,1]^d$, for instance, see Definition 3.4 of \cite{gyorfi2002distribution}. The latter model is actually a particular case of the class of models that we consider.

\end{remark} 







\begin{assumption}
\label{assumption4}
For the $f_0$ in \eqref{eq45}, we assume that $ 
\sum_{i=1}^n f_{0j}(X^i_j) = 0, \quad \text{for } j = 1 , \dots,  d.
$ 
\end{assumption}
Our Assumption~\ref{assumption4} is consistent with Assumption B1 by \cite{sadhanala2019additive}, and it is required for identifiability. 
\begin{assumption}
\label{assumptionA}
There exists a constants $L > 0$ and $\ubar{f} > 0$ such that for any
function $\delta: [0,1]^d \rightarrow \R $ with $\norm{\delta}_{\infty} \le L$ we have that
\beq 
\abs{F_{Y^i|X^i}( f^i_0 + \delta^i) - F_{Y^i|X^i}( f^i_0)} \ge \ubar{f} \abs{\delta^i},
\beq     
where $F_{Y^i|X^i}$ is the conditional CDF of $Y^i$, and $\delta^i = \delta(X^i)$. Note for a given quantile level $\tau$, for the input $X^i$, we define the evaluation of $f_0$ at $X^i$ as $f^i_0 = F^{-1}_{Y^i|X^i}(\tau)$, where $F^{-1}_{Y^i|X^i}(\tau)$ represents the $\tau$ conditional quantile of $Y^i$ given $X^i$.
\end{assumption} 

Similar conditions are assumed by Assumption A of \cite{madrid2022risk} and Assumption 4 of \cite{shen2022estimation}. 
Assumption \ref{assumptionA} ensures that the conditional quantile of $Y^i$ given $X^i$ is uniquely defined and there is a uniform linear growth of the CDF around a neighborhood of the quantile.

\section{SUPPLEMENTARY LEMMAS FOR THEOREM~\ref{theorem1}}
To complete our proofs for the theorems, we provide additional definitions.
\begin{definition} 
The $r$th order $TV$-ball of radius $V$ is defined as: 
\label{def:J}
\beq 
\label{eq197}
B_{TV^{(r)}}(V)  = \set{f: TV^{(r)}(f)  \le V}. 
\beq
We also denote
\beq 
\label{ball_tv_d}
B_{TV^{(r)}}^d(V) = \bigg\{\sum_{j=1}^d f_j: \sum_{j=1}^d TV^{(r)}(f_j) \le V, j = 1, \dots, d\biggr\}. 
\beq

The $r$th order trend filtering operator norm ball of radius $V$ as 
\beq 
\label{def:ball of trend filtering}
B_{D^{(X,r)}_n}(V) = \set{\theta \in \R^n: \norm{D^{(X,r)}_n(\theta)}_1 \le V}.
\beq    

We abbreviate
$B_n(V)$ for the $\|\cdot\|_n$-ball of radius $V$, $B_{L_2}(V)$ for
the $\|\cdot\|_2$-ball of radius $V$, and $B_\infty(V)$ for the
$\|\cdot\|_\infty$-ball of radius $V$.
\end{definition}

\begin{definition}
\label{definition213}
Let $Y\in \R^n$ be a vector of independent random variables and for a given quantile level $\tau \in (0,1)$.    
For any finite function class $\gF$, we define the Empirical Rademacher width as
\bea 
\gR_n(\gF) = \ept{\sup_{f \in \gF} \frac{1}{n}\sum_{i=1}^n \xi^i f^i(X^i)\mid \set{X^i}_{i=1}^n},
\bea    
and its Rademacher complexity 
\bea 
\gR(\gF) = \ept{\sup_{f \in \gF} \frac{1}{n}\sum_{i=1}^n \xi^i f(X^i)}. 
\bea
\end{definition}

\begin{definition}
  For $f \in \gF^{(r)}(V)$ , we define the empirical loss function as 
\beq 
\label{295}
\widehat M(f) = \sum_{i=1}^n \widehat M^i(f), 
\beq 
where each $\widehat M^i(f)$ is defined as 
\beq 
\label{348}
\widehat M^i(f) =  \rho_{\tau}(Y^i  - f(X^i)) - \rho_{\tau}(Y^i  - f_0(X^i)).
\beq   

We also define the population loss function as 
\beq 
\label{305}
M(f) \coloneqq \sum_{i=1}^n  M^i(f), 
\beq where  
\beq 
 M^i(f) \coloneqq  \ept{\rho_{\tau}(Y^i  - f(X^i)) - \rho_{\tau}(Y^i  - f_0(X^i))|X^i}.
\beq
\end{definition}

It is well known that for our model in \eqref{eq45}, conditioned on $X^i$, the true $\tau$-quantile of $Y^i$, $f_0(X^i)$ satisfies  
\beq
\label{eq44}
f_0 = \arg \min_{f\in \gF^{(r)}(V)  }M(f),
\beq where the expectation is taken for $Y^i$ conditioned on $X^i$. 

Note, from \eqref{eq44}, we know $f_0 \in \argmin M(f)$, and $M(f_0) = 0$. 

\begin{definition}
\label{definition3}
The constrained quantile additive estimator defined in \eqref{eq298} will be
\label{def}
    \beq 
\widehat f = \argmin_{f \in \gF^{(r)}(V)} \widehat M(f),
\beq  
\end{definition}

By Definitions in \eqref{eq44} and \eqref{295}, we have 
\beq 
\widehat M(\widehat f) - \widehat M(f ^*) \le 0 
\beq as the basic inequality.

\begin{lemma}
\label{lemma12}
For any $f, f_0 \in \gF^{(r)}(V)$, denote $\delta(X^i) = f(X^i) - f_0(X^i)$, then we have 
    \beq 
M^i( f) - M^i(f_0) &= \int_{0}^{\delta(X^i)} \paren{F_{Y^i|X^i}(f_0(X^i) + z) - F_{Y^i|X^i}(f_0(X^i))}dz.
\beq 
\end{lemma}

\begin{proof}
        
         By equation (B.3) in \cite{belloni2011},we have 
\beq
\rho_{\tau}(w-v) - \rho_{\tau}(w) = -v(\tau - \ind{w\le 0}) + \int_{0}^v \set{\ind{w \le z} - \ind{w \le 0}} dz. 
\beq 

Given any $f$ and $X^i$, let $w = Y^i  - f_0(X^i)$, $v = f(X^i) - f_0(X^i)$. 
Then taking conditional expectations on
above equation with respect to $Y^i$ conditioned on $X^i$ on both sides, we have 
\beq 
M^i( f) - M^i(f_0) &= \ept{  \rho_{\tau}(Y^i  - f(X^i)) - \rho_{\tau}(Y^i  - f_0(X^i))|X^i} \\&= \ept{-\delta(X^i)(\tau-\ind{Y^i  - f(X^i) \le 0})|X^i} \\
 & \quad + \ept{\int_{0}^{\delta(X^i)} \paren{\ind{Y^i  - f_0(X^i) \le z} - \ind{Y^i  - f_0(X^i)\le 0}dz|X^i}}\\
 &= 0 + \int_{0}^{\delta(X^i)} \paren{F_{Y^i|X^i}( f_0(X^i) + z) - F_{Y^i|X^i}( f_0(X^i))}dz \\
&= \int_{0}^{\delta(X^i)} \paren{F_{Y^i|X^i}( f_0(X^i) + z) - F_{Y^i|X^i}( f_0(X^i))}dz,  \\
\beq for all $i=1, \dots, n$. 
\end{proof}

\begin{lemma}
\label{lemma13}
Let $\Delta^2(\delta) \coloneqq n \Delta_n^2(\delta)$, 
assume Assumption~\ref{assumptionA} holds, then there exists $c_0$ such that for all $\delta = f- f_0$, we have 
\beq 
\label{bound_two_losses}
M(f_0 + \delta) \ge c \Delta^2(\delta),
\beq where $c$ is a positive constant that depends on $L$ and $\ubar{f}$ in Assumption~\ref{assumptionA}.
\end{lemma}

\begin{proof}
Supposing that $\abs{\delta(X^i)} \le L$, by Assumption~\ref{assumptionA} and Lemma~\ref{lemma12}, we have 
\beq 
M^i( f) - M^i(f_0) &\ge \int_{0}^{\delta(X^i)} \ubar{f} \abs{z} dz\\
&= \frac{\ubar{f}\delta(X^i)^2}{2}
\beq 
Supposing that $\delta(X^i) > L$, by Lemma~\ref{lemma12}, we have 
\beq 
M^i( f) - M^i(f_0) &\ge \int_{L/2}^{\delta(X^i)} \paren{F_{Y^i|X^i}( f_0(X^i) + z) - F_{Y^i|X^i}( f_0(X^i))}dz\\
&\ge \int_{L/2}^{\delta(X^i)} \paren{F_{Y^i|X^i}( f_0(X^i) + L/2) - F_{Y^i|X^i}( f_0(X^i))}dz\\
&= (\delta(X^i) - L/2)\paren{F_{Y^i|X^i}( f_0(X^i) + L/2) - F_{Y^i|X^i}( f_0(X^i))}\\
&\ge \frac{\delta(X^i)}{2}\frac{L \ubar{f}}{2}\coloneqq c \abs{\delta(X^i)},\\ 
\beq 
where the first two inequalities follow because $F_{Y^i|X^i}$ is monotone, where $c$ is a positive constant that depends on $L$ and $\ubar{f}$ in Assumption~\ref{assumptionA}. Then for the case $\delta(X^i) <-L$, it can be handled similarly. The conclusion follows combining the three different cases. 
\end{proof}


\begin{lemma}[Symmetrization]
\label{lemma10}
For $t>0$, it holds that 
\beq 
\ept{\sup_{ f\in \gF^{(r)}(V), M(f)\le t^2 }\paren{M(f) - \widehat M(f)}\mid \set{X^i}_{i=1}^n} \le 2 \ept{\sup_{f\in \gF^{(r)}(V), M(f)\le t^2} \sum_{i=1}^n \xi^i \widehat M^i(f)\mid \set{X^i}_{i=1}^n}, 
\beq where $\xi_1,\dots, \xi_n$ are independent Rademacher variables independent of $\set{Y^i}_{i=1}^n$.
\end{lemma}
\begin{proof}
Let $\tilde Y^1, \dots \tilde Y^n$ be an independent and identically distributed copy of $Y^1, \dots, Y^n$.
Let $\tilde M^i$ the version of $\widehat M^i$ corresponding to $\tilde Y^1, \dots \tilde Y^n$. 
    If we define 
    \beq 
    \label{eq64}
    X_f = \sum_{i=1}^n \paren{\tilde M^i(f) - \widehat M^i(f)}
    \beq
    then we have 
    \beq 
    \sup_{f\in \gF^{(r)}(V), M(f)\le t^2} \ept{X_f|\set{Y^i}_{i=1}^n, \set{X^i}_{i=1}^n } \le \ept{\sup_{f\in \gF^{(r)}(V), M(f)\le t^2}X_f|\set{Y^i}_{i=1}^n, \set{X^i}_{i=1}^n  }.
    \beq 
    The expectation is taken with respect to $\tilde Y^1, \dots \tilde Y^n$ conditioned on $Y^1, \dots, Y^n,X^1, \dots, X^n$. 

Then,
the left-hand side of the above, can be further written as     
\beq 
\label{eq65}
{}&
    \sup_{f\in \gF^{(r)}(V), M(f)\le t^2} \ept{X_f|\set{Y^i}_{i=1}^n, \set{X^i}_{i=1}^n } \\ &=     \sup_{f\in \gF^{(r)}(V), M(f)\le t^2} \paren{\ept{\sum_{i=1}^n \tilde M^i (f) \mid \set{Y^i}_{i=1}^n, \set{X^i}_{i=1}^n} - \ept{\sum_{i=1}^n \widehat M^i (f) \mid \set{Y^i}_{i=1}^n, \set{X^i}_{i=1}^n} } \\ &=     \sup_{f\in \gF^{(r)}(V), M(f)\le t^2} \sum_{i=1}^n \paren{M^i (f) -  \widehat M^i (f)}.
\beq 

Combing \eqref{eq64} and \eqref{eq65}, we get 
\beq 
\sup_{f\in \gF^{(r)}(V), M(f)\le t^2} \sum_{i=1}^n \paren{M^i (f) -  \widehat M^i (f)} \le \ept{\sup_{f\in \gF^{(r)}(V), M(f)\le t^2}\sum_{i=1}^n \paren{\tilde M^i(f) - \widehat M^i(f)}|\set{Y^i}_{i=1}^n, \set{X^i}_{i=1}^n}.
\beq 
Further take the expectation with respect of $Y^1, \dots, Y^n$ on both sides, and similar to the proof of the Lemma 10 in \cite{madrid2022risk}, we can get 
\beq 
{}&\ept{\sup_{f\in \gF^{(r)}(V), M(f)\le t^2} \sum_{i=1}^n \paren{M^i (f) -  \widehat M^i (f)}\mid \set{X^i}_{i=1}^n} \\&\le \ept{\sup_{f\in \gF^{(r)}(V), M(f)\le t^2} \sum_{i=1}^n \paren{\tilde M^i (f) -  \widehat M^i (f)} \mid \set{X^i}_{i=1}^n} \\
 &= \ept{\sup_{f\in \gF^{(r)}(V), M(f)\le t^2} \sum_{i=1}^n \xi^i\paren{\tilde M^i (f) -  \widehat M^i (f)} \mid \set{X^i}_{i=1}^n} \\
&\le \ept{\sup_{f\in \gF^{(r)}(V), M(f)\le t^2} \sum_{i=1}^n \xi^i\tilde M^i (f)\mid \set{ X^i}_{i=1}^n} + \ept{\sup_{f\in \gF^{(r)}(V), M(f)\le t^2} \sum_{i=1}^n -  \xi^i\widehat M^i (f) \mid \set{ X^i}_{i=1}^n} \\
 &= 2\ept{\sup_{f\in \gF^{(r)}(V), M(f)\le t^2} \sum_{i=1}^n \xi^i\widehat M^i (f)\mid \set{ X^i}_{i=1}^n},
    \beq 

where the first equality follows because $\xi_1(\tilde M^i (f) -  \widehat M^i (f), \dots, \xi_n(\tilde M_n (f) -  \widehat M_n (f))$ and $(\tilde M^i (f) -  \widehat M^i (f), \dots, (\tilde M_n (f) -  \widehat M_n (f))$ have the same distribution. The second equality follows because $-\xi_1, \dots, -\xi_n$ are also independent Rademacher variables. 
\end{proof}

\begin{lemma}[Contraction principle] 
\label{lemma:contraction}
With the notation from before we have that, for $t>0$, 
\beq 
&\ept{\sup_{f\in \gF^{(r)}(V), M(f)\le t^2} \sum_{i=1}^n \xi^i\widehat M^i (f) \mid \set{ X^i}_{i=1}^n} \le n\gR_n(\set{f- f_0: f\in \gF^{(r)}(V)}\cap \set{f-f_0: M(f)\le t^2}).
\beq 
\end{lemma}
\begin{proof}
Based on the definition of $M^i(f)$, 
\beq 
\widehat M^i(f) =   \rho_{\tau}(Y^i  - f(X^i)) - \rho_{\tau}(Y^i  - f_0(X^i)).
\beq    
From Lemma~\ref{lemma27}, $\widehat M^i(f)$ are 1-Lipschitz continuous functions, thus, 
\beq 
&\ept{\sup_{f\in \gF^{(r)}(V), M(f)\le t^2} \sum_{i=1}^n \xi^i \widehat M^i(f) \mid \set{ X^i}_{i=1}^n} \\ &= \ept{\ept{\sup_{f\in \gF^{(r)}(V), M(f)\le t^2} \sum_{i=1}^n \xi^i \widehat M^i(f) \bigm\mid \set{ X^i}_{i=1}^n, \set{ Y^i}_{i=1}^n}\mid \set{ X^i}_{i=1}^n} 
\\
& \le \ept{\ept{\sup_{f\in \gF^{(r)}(V), M(f)\le t^2} \sum_{i=1}^n \xi^i f^i \bigm\mid \set{ X^i}_{i=1}^n,\set{ Y^i}_{i=1}^n}|\set{ X^i}_{i=1}^n} \\
& = \ept{\sup_{f\in \gF^{(r)}(V), M(f)\le t^2} \sum_{i=1}^n \xi^i f^i \bigm\mid \set{ X^i}_{i=1}^n} \\
& \le \ept{\sup_{f\in \gF^{(r)}(V), M(f)\le t^2} \sum_{i=1}^n \xi^i (f^i - f^i_0) \bigm\mid \set{ X^i}_{i=1}^n} + \ept{\sum_{i=1}^n \xi^i f^i_0 \mid \set{ X^i}_{i=1}^n}\\
& = \ept{\sup_{f\in \gF^{(r)}(V), M(f)\le t^2} \sum_{i=1}^n \xi^i (f^i - f^i_0)(X^i)\mid \set{ X^i}_{i=1}^n}.
\beq 

The third line holds since we have used the 1-Lipschitz condition, so the expectation inside does not have $y$ anymore. 
\end{proof}

\begin{lemma}
\label{lemma27}
The  function $\widehat M^i(f)$ is 1-lipschitz. 
\end{lemma}
\begin{proof}
    For any $f, g$, and fixed $\tau$, we have 
    \beq 
    \label{eq618}
    {}&\abs{\widehat M^i(f) - \widehat M^i(g)} \\
    &= \abs{\rho_{\tau}(Y^i  - f(X^i)) - \rho_{\tau}(Y^i  - g(X^i))}\\
    &= \abs{\max \set{\tau(Y^i  - f(X^i)), (1-\tau)(Y^i  - f(X^i))} - \max \set{\tau(Y^i  - g(X^i)), (1-\tau)(Y^i  - g(X^i))}}. 
    \beq 
It is easy to see the conclusion satisfied if $Y^i  - f(X^i)$ and $Y^i  - g(X^i)$ have the same signs.

Assume $Y^i  - f(X^i) > 0$ and $Y^i  - g(X^i) \le 0$, we have Equation~\eqref{eq618} equals 
\beq 
{}&\abs{\max \set{\tau(Y^i  - f(X^i)), (1-\tau)(Y^i  - f(X^i))} - \max \set{\tau(Y^i  - g(X^i)), (1-\tau)(Y^i  - g(X^i))}} \\
&= \abs{\tau(Y^i  - f(X^i)) -  (1-\tau)(Y^i  - g(X^i))} \\ 
&= \abs{\tau(g(X^i) - f(X^i)) + (2\tau-1)(Y^i  - g(X^i))}.
\beq 

If $(2\tau-1)\le 0$, since $Y^i - f(X^i) > 0$ and $g(X^i) > f(X^i)$,  then we know $0 \le (2\tau-1)(Y^i  - g(X^i)) \le (2\tau-1)(f(X^i) - g(X^i)) $, we get 
\beq 
\label{lip_inq1}
\abs{\widehat M^i(f) - \widehat M^i(g)} \le &
\abs{\tau(g(X^i) - f(X^i)) + (2\tau-1)(f(X^i) - g(X^i))} \\ 
&= \abs{(\tau-1)(f(X^i) - g(X^i))} \le \abs{f(X^i) - g(X^i)},
\beq 
On the other hand, if $(2\tau-1)\ge 0$, if further 
$\abs{(2\tau-1)(Y^i -g(X^i))} \ge \tau(g(X^i) - f(X^i))$, and $\tau(g(X^i) - f(X^i)) > \tau(g(X^i) - Y^i) \geq 0$, science $Y^i < f(X^i) \leq g(X^i)$, then we have 

\beq 
\label{lip_inq2}
\abs{\widehat M^i(f) - \widehat M^i(g)} \le &
\abs{\tau(g(X^i) - Y^i) + (2\tau-1)(Y^i  - g(X^i) )} \\ 
&= \abs{(\tau-1)(Y^i  - g(X^i))} \le \abs{(\tau-1)(f(X^i) - g(X^i))} \le \abs{f(X^i) - g(X^i)},
\beq

if $\abs{(2\tau-1)(Y^i -g(X^i))} \le \tau(g(X^i) - f(X^i))$, then we have 

\beq 
\label{lip_inq3}
\abs{\widehat M^i(f) - \widehat M^i(g)} \le &
\abs{\tau(g(X^i) - f(X^i)) + (2\tau-1)(f(X^i) - g(X^i))} \\ 
&= \abs{(\tau)(f(X^i) - g(X^i))} \le \abs{f(X^i) - g(X^i)}.
\beq 

Thus, by \eqref{lip_inq1}, \eqref{lip_inq2} and \eqref{lip_inq3}, we have 
\bea 
\abs{\widehat M^i(f) - \widehat M^i(g)} \le \abs{f(X^i) - g(X^i)}. 
\bea

For the other case, we assume $Y^i  - f(X^i) \leq 0$ and $Y^i  - g(X^i) > 0$, 

Since 
\bea 
\abs{\widehat M^i(f) - \widehat M^i(g)} = \abs{\widehat M^i(g) - \widehat M^i(f)}, 
\bea 
 we thus can derive the same conclusion by applying the above analysis, thus, we have 
\beq 
\abs{\widehat M^i(f) - \widehat M^i(g)} \le \abs{f(X^i) - g(X^i)}, 
\beq 
 which indicates that $\widehat M^i(f)$ is a 1-lipschitz function. 
\end{proof}

\begin{proposition}
\label{prop1}
Let $\gF^{(r)}(V)$ be a function class, then the following inequality is true for any $t  > 0$, 
\beq 
\sP \paren{M(\widehat f) > t^2\mid \set{X^i}_{i=1}^n} \le \frac{2n}{t^2}\gR_n(\set{f- f_0: f\in \gF^{(r)}(V)}\cap \set{f-f_0: M(f)\le t^2}).
\beq 
\end{proposition}

\begin{proof}

Suppose that 
\beq 
M(\widehat f) > t^2. 
\beq 
First, define $g: \bracket{0,1} \rightarrow \R$ as $g(u) = M\paren{(1-u)f_0 + u\widehat f}$. Clearly, $g$ is a continuous function with $g(0) = 0$, and $g(1) =  M (\widehat{f})$.  Therefore, there exists $\mu_{\widehat f} \in \bracket{0,1}$ such that $g(\mu_{\widehat f}) = t^2$. Hence, letting $\tilde f = (1-\mu_{\widehat f}) f_0 + \mu_{\widehat f}\widehat f$, we observe that by the convexity of $\widehat M$ and the basic inequality, we have 
\beq 
\widehat M (\tilde f) = \widehat M\paren{(1-\mu_{\widehat f}) f_0 + \mu_{\widehat f}\widehat f} \le (1-\mu_{\widehat f})\widehat M(f_0) + \mu_{\widehat f} \widehat M(\widehat f) \le 0.
\beq 
Furthermore, $\gF^{(r)}(V)$ from \eqref{eq298} is a convex set, and $f_0$ and  $\widehat f$ are belong to $\gF^{(r)}(V)$, thus we know $\tilde f$ belongs to $\gF^{(r)}(V)$, furthermore, we have $M(\tilde f) = t^2$ by construction. This implies that 
\beq 
\sup_{f\in \gF^{(r)}(V), M(f)\le t^2 }M(f) - \widehat M(f) \ge M(\tilde f) - \widehat M(\tilde f) \ge M(\tilde f).  
\beq
The first inequality holds by the supreme, and the second inequality is held by $\widehat M(\tilde f) \le 0$. 
In the results shown above, it is true that 
\beq 
M(\widehat f) > t^2 \implies \sup_{f\in \gF^{(r)}(V), M(f)\le t^2  }M(f) -\widehat M( f) > t^2.
\beq 
Therefore, 
    \beq 
     \sP \paren{M(\widehat f) > t^2\mid \set{X^i}_{i=1}^n} &\le \sP \biggl(\sup_{f\in \gF^{(r)}(V), M(f)\le t^2 }M(f) - \widehat M(f) > t^2\mid \set{X^i}_{i=1}^n\biggr) \\&\le  \frac{1}{t^2}\ept{\sup_{f\in \gF^{(r)}(V), M(f)\le t^2 }M(f) - \widehat M(f)\mid \set{X^i}_{i=1}^n} \\&\le 
     \frac{2n}{t^2}\gR_n\paren{\set{f- f_0: f\in \gF^{(r)}(V)}\cap \set{f-f_0: M(f)\le t^2} }.
    \beq 
The third inequality is by Lemma~\ref{lemma10} and ~\ref{lemma:contraction}. 
\end{proof}

\begin{theorem}
\label{theorem7}
Suppose the distribution of $Y^1, \dots, Y^n$ obey Assumption~\ref{assumptionA}, then the following inequality is true for any $t > 0$, \beq 
\sP\paren{\Delta^2(\widehat f - f_0) > t^2\mid \set{X^i}_{i=1}^n} & \le \frac{cn}{t^2}\gR_n\paren{\set{f-f_0: f \in \gF^{(r)}(V)} \cap \set{f-f_0: \Delta^2( f - f_0)\le t^2}},
 \beq 
 where $c$ is the same constant in \cref{lemma13} that only depends on the distribution of $X^1, \dots, X^n, $ and the distribution of $Y^1, \dots, Y^n$, and $\gF^{(r)}(V)$ is a convex set with $\widehat f, f_0 \in \gF^{(r)}(V)$, and $\gR_n$ is defined in \eqref{definition213}. 
\end{theorem}

\begin{proof}
It is true that 
\beq 
    \sP\paren{\Delta^2(\widehat f - f_0) > t^2\mid \set{X^i}_{i=1}^n} & \le  \sP \paren{M(\widehat f) > ct^2\mid \set{X^i}_{i=1}^n} \\ 
    &\le \frac{cn}{t^2}\gR_n\paren{\set{f- f_0: f\in \gF^{(r)}(V)}\cap \set{f-f_0: M(f)\le t^2} } \\ 
    & \le  \frac{cn}{t^2}\gR_n\paren{\set{f-f_0: f \in \gF^{(r)}(V)} \cap \set{f-f_0: \Delta^2( f - f_0)\le t^2} }, 
\beq 
where the first inequality follows from Lemma~\ref{lemma13}, 
 the second inequality follows from Proposition ~\ref{prop1}, and the third inequality follows from the fact that $\set{f-f_0: M(f)\le t^2} \subset    \set{f-f_0: \Delta^2( f - f_0)\le t^2} $.
\end{proof}

We would need to upper bound the quantity 
\beq 
\gR_n\paren{\set{f-f_0: f \in \gF^{(r)}(V)} \cap \set{f-f_0: \Delta^2( f - f_0)\le t^2} }, 
\beq 
where 
\beq 
\gF^{(r)}(V) - f_0 = \biggl\{f - f_0 = \sum_{j=1}^d (f_j - f_{0j}): 
f_j \in \gH_j, \sum_{j=1}^d TV^{(r)}(f_j) \le V, \sum_{i=1}^n f_j(X^i_j) = 0 \biggr\}, 
\beq and  
\beq 
\Delta^2(f) = \sum_{i=1}^n \min\Bigl\{\Bigl|\sum_{j=1}^df_j(X^i_j)\Bigr|, \Bigl(\sum_{j=1}^d f_j(X^i_j)\Bigr)^2\Bigr\}. 
\beq

\begin{theorem}
\label{theorem8}
Suppose $f_j \in \gH_j $, the space spanned by the falling factorial basis, and for the fixed points of $X^1, \dots, X^n, Y^1, \dots, Y^n$ obey Assumption~\ref{assumptionA}, 
\begin{enumerate}[label=(\alph*)]
\item Let 
\beq 
f_j = (f_j(X^1_j), \dots, f_j(X^n_j)),
\beq $f = \sum_{j=1}^df_j$, then the following inequality is true for any $t > 0$, 
\beq 
\sP\paren{\Delta^2(\widehat f - f_0) > t^2\mid \set{ X^i}_{i=1}^n} & \le \frac{cn}{t^2}\gR_n\paren{\set{f-f_0: f \in \gF^{(r)}(V)} \cap \set{f-f_0: \Delta^2( f - f_0)\le t^2}},
 \beq 
  where  $\gF^{(r)}(V)$ is defined in \eqref{eq298}, where $c$ is a constant that only depends on the distribution of $X^1, \dots, X^n, Y^1, \dots, Y^n$.
\item Let 
\beq 
\theta_j = (\theta^1_j, \dots, \theta^n_j),
\beq $\theta^i_j = f_j(X^i_j)$, and let $\theta_0$ is given by
\beq
\label{362}
\theta_0 = \min_{\theta \in K^{(X,r)}(V)}\sum_{i=1}^n\ept{\rho_{\tau}(Y^i -  \theta^i) - \rho_{\tau}(Y^i -  \theta_0^i)|X^i},
\beq
then the following inequality is true for any $t > 0$, 
\beq 
\sP\paren{\Delta^2(\widehat f - f_0) > t^2 \mid \set{ X^i}_{i=1}^n} &=
\sP\paren{\Delta^2(\widehat \theta - \theta_0) > t^2 \mid \set{ X^i}_{i=1}^n} \\  & \le \frac{cn}{t^2}\gR_n\paren{\set{\theta-\theta_0: 
\theta \in K^{(X,r)}(V)} \cap \set{\theta-\theta_0: \Delta^2( \theta - \theta_0)\le t^2}},
 \beq 
 where $K^{(X,r)}(V)$ is defined in \eqref{space_K}, $\widehat \theta$ is an estimator in $K^{(X,r)}(V)$,
 and $c$ is a constant that only depends on the distribution of $X^1, \dots, X^n, $ and the distribution of $Y^1, \dots, Y^n$.
\end{enumerate}



\end{theorem}
\begin{proof}
    For $(a)$, the conclusion (the inequality) is derived based on Theorem~\ref{theorem7}. For $(b)$, the first equality holds because the equivalent formulation of additive trend filtering from Equation~\eqref{eq298} and Equation~\eqref{eq303}, then a similar conclusion is derived for the second inequality. 
\end{proof}

By the definition of $TV^{(r)}$, $TV^{(r)}$ is a seminorm $TV^{(r)}$, since its domain is contained in the space of $(r-1)$ times weakly differentiable functions, and its null space contains all $(r-1)$th order polynomials.

\begin{definition}{}
\label{definition4} 
Let $f, f_0 \in \gF^{(r)}(V)$. Let $\delta = f - f_0 \in \gF^{(r)}(V)$, let $\delta = \sum_{j=1}^d\delta_j$, with each $\delta_j = f_j - f_{0j}$, for $j=1, \dots, d$.
\begin{enumerate}
    \item Define a subset that is contained in the space of $(r-1)$times weakly differentiable functions. 
\bea 
\gW = \set{g: [0,1] \rightarrow \R, TV^{(r)}(g) \le V}.
\bea 
Clearly, we have $f_j \in \gW$, $f_{0j} \in \gW$, and $\delta_j \in \gW$.
\item Define a subspace of $\R^n$ that spanned by functions in $\gW$ evaluated at points $X_j = (X^1_j, \dots, X^n_j)$ as
\beq
\label{ref:defofw_j}
W_j = \Span{(g(X^1_j), \dots, g(X^n_j)), g \in \gW}.
\beq   

For any $\delta_j$,  let the vector $\delta_j(X_j)$ denote $\delta_j$ evaluated at points $X_j = (X^1_j, \dots, X^n_j)$:
\bea 
\delta_j(X_j) \coloneqq (\delta_j(X^1_j), \dots,  \delta_j(X^n_j)).
\bea

Let the $n$-dimensional vector obtained by adding $\delta_j(X_j)$ be denoted as 
\bea 
\delta(X) \coloneqq  \Bigl(\sum_{j=1}^d\delta_j(X^1_j), \dots,  \sum_{j=1}^d\delta_j(X^n_j)\Bigr) = (\delta(X^1), \dots,  \delta(X^n)).
\bea  
This implies that $\delta(X) = \sum_{j=1}^d \delta_j(X_j), \delta_j \in \gW$. We will see $\delta(X^i)$ equals to  $\sum_{j=1}^d\delta_j(X^i_j)$.

\item Let $\phi_{\ell}(\nu)$ be the  $\ell$th degree  polynomials $\ell \leq (r - 1)$, Define $P_j$ be the space spanned by all $(r-1)$th order polynomials evaluated at points $X_j = (X^1_j, \dots, X^n_j)$ as 
\beq 
\label{def:p_j}
P_j \coloneqq \Span{(\phi_{\ell}(X^1_j), \dots, \phi_{\ell}(X^n_j)), l = 0, \dots, r-1}.
\beq 
\item Define the orthogonal complement of $P_j$ in terms of $\norm{\cdot}_n$ as 
\beq
\label{def:q_j}
Q_j \subseteq W_j: Q_j= P_j^{\perp}, \quad Q_j \oplus P_j = W_j.
\beq 
 
In other words, for any $f_j(X_j) \in P_j$  and $g_j(X_j) \in Q_j$, it holds 
$\innerprod{f_j}{g_j}_n = 0$. 
\item Define the orthogonal projection operators applied to $W_j$ for $P_j$ and $Q_j$ to be
\beq
\label{ref:defofpip_j}
\Pi_{P_j}: \Pi_{P_j}t \in P_j,  \quad \Pi_{Q_j}: \Pi_{Q_j}t \in Q_j, \quad \text{with}\quad t \in W_j.
\beq 
\item  
For any $\delta \in \gF^{(r)}(V)$, with the vector $\delta(X) = (\delta(X^1), \dots,  \delta(X^n)) \in \R^n$, with $\Pi_{P_j}$ and $\Pi_{Q_j}$ defined in \eqref{ref:defofpip_j}, we define $\Pi_P$ 
and $\Pi_Q$ as 
\beq
\label{ref:defofpip}
\Pi_P: \Pi_P \delta(X)= \Pi_{P_1}\delta_1(X_j^{(1)}) + \dots +  \Pi_{P_d}\delta_d(X_d),
\beq 
and 
\beq
\label{ref:defofpiq}
\Pi_Q: \Pi_Q \delta(X)= \Pi_{Q_1}\delta_1(X_j^{(1)}) + \dots +  \Pi_{Q_d}\delta_d(X_d). 
\beq 
Then we denote 
\bea 
&p_j \coloneqq \Pi_{P_j}\delta_j(X_j) \in \R^n, \quad q_j \coloneqq \Pi_{Q_j}\delta_j(X_j)\in \R^n. 
\bea 
And let $p^i = (p^i_1, \dots, p^i_d), \quad q^i = (q^i_1, \dots, q^i_d),$
where $p^i_1,\dots, p^i_d$ are the elements of $p^i$, and $q^i_1, \dots, q^i_d$ are the elements of $q^i$. 
And denote 
\bea 
&p \coloneqq p_1 + \dots + p_d \in \R^n,\quad q \coloneqq q_1 + \dots + q_d\in \R^n. 
\bea 
So we have $\delta_j(X_j) = p_j + q_j$,  $\Pi_{P}\delta(X) = p$, and $\Pi_{Q}\delta(X) = q$. 
\end{enumerate}
\end{definition}

\begin{lemma}
\label{lemma15}
    Let $v_j \in \R^{n\times r}$ be a matrix whose $i$th row is given by 
\beq 
(1, X^{(i)}_j, (X^{(i)}_j)^2, \dots, (X^{(i)}_j)^{r-1}) \in \R^r,
\beq where $X^{(i)}_j, i = 1, \dots, n$ are the sorted data points in $[0,1]$.  By the definition of $P_j$ in Definition~\ref{definition4},
it holds that 
\beq P_j =
\Span{ \text{columns of } v_j}. 
\beq
Furthermore, for $D^{(X_j,r)}_n$, it holds  $\Span{ \text{columns of } v_j} = \text{null}(D^{(X_j,r)}_n)$.
\end{lemma}
\begin{proof} 
Clearly, $P_j$ equals to the column space of $v_j$. 


We know $D^{(X_j, r)}_n \in \R^{(n-r) \times n}$ as outlined in \ref{eq:higher_order} is the discrete difference operator. 
It is evident that for any $\nu \in \Span{\text{columns of } v_j}$, we have $D^{(X_j, r)}_n\nu = 0$. 
This follows from Lemma 1 and Assumption C1 in \cite{sadhanala2019additive}. 
Lemma 1 in \cite{sadhanala2019additive} expresses the $TV^{(r)}$ of a function $f$ in terms of the falling factorial basis. 
Since the $TV^{(r)}$ of an $(r-1)$th order polynomial is zero, the result follows.
Furthermore, the dimension of $\Span{ \text{columns of } v_j}$ is equal to $r$. 
By the definition of $D^{(X_j, r)}_n$, the dimension of $\text{row}(D^{(X_j, r)}_n)$ is equal to $n-r$. Then the dimension of the $\text{null}(D^{(X_j, r)}_n)$ is equal to $n-(n-r) = r$. Thus, we have $\Span{ \text{columns of } v_j} = \text{null}(D^{(X_j,r)}_n)$.


\end{proof}

   
  

\begin{lemma}
\label{lemma11}
For $\gF^{(r)}(V)$ defined in Definition~ \eqref{space_F}, and for the $f_0 \in \gF^{(r)}(V)$, 
\begin{enumerate}[label=(\alph*)]
\item For any $t>0$, let a space of functions be defined as 
\beq
\label{eq672}
\gD(t^2)\coloneqq \set{\delta = f - f_0: \Delta(\delta)\le t^2, f \in \gF^{(r)}(V)}, 
\beq  
We have  
\beq
\label{decomposition_t1_t2}
\gR_n(\set{\gF^{(r)}(V) - f_0} \cap \gD(t^2)) & \le 
T_1 + T_2,
\beq  with $T_1$ and $T_2$ are given below
\bea 
T_1 \coloneqq \ept{\sup_{\delta \in \gF^{(r)}(2V)\cap \gD(t^2)}  \xi^{\top} \Pi_{P}\delta(X)|\set{X^i}_{i=1}^n},  
T_2 \coloneqq \ept{\sup_{\delta \in \gF^{(r)}(2V)\cap \gD(t^2)}  \xi^{\top} \Pi_{Q}\delta(X)|\set{X^i}_{i=1}^n}.
\bea 

\item For $K^{(X,r)}(V)$ in Definition~\eqref{space_K} and for $\gD(t^2)\coloneqq \set{\delta \in \R^n, \Delta(\delta)\le t^2}$, we have for the same $T_1$ and $T_2$  
\beq 
\gR_n(\set{K^{(X,r)}(V) - \theta_0} \cap \gD(t^2)) \le  T_1 + T_2.  
\beq 

\end{enumerate}

\end{lemma} 
\begin{proof}
With the definitions and notations in Definition~\ref{definition4}.

For $(a)$, we have 
    \beq 
    \label{eq499}
\sup_{\substack{\delta \in \gF^{(r)}(V) - f_0: \Delta^2(\delta)
\le t^2}} \xi^{\top}\delta(X) 
&\le \sup_{\delta \in \gF^{(r)}(2V)\cap \gD(t^2)}  \xi^{\top} \Pi_{P}\delta(X)  \\ &+ \sup_{\delta \in \gF^{(r)}(2V)\cap \gD(t^2)}  \xi^{\top} \Pi_{Q}\delta(X),
\beq
taking expectation conditioned on $\set{X^i}_{i=1}^n$ on both sides, 
we have 
\beq 
\label{eq513}
&\ept{\sup_{\delta \in \gF^{(r)}(V) - f_0: \Delta^2(\delta)
\le t^2} \xi^{\top}\delta(X)|\set{X^i}_{i=1}^n} 
\\& \le \underbrace{\ept{\sup_{\delta \in \gF^{(r)}(2V)\cap \gD(t^2)}  \xi^{\top} \Pi_{P}\delta(X)|\set{X^i}_{i=1}^n}}_{T_1}  + \underbrace{\ept{\sup_{\delta \in \gF^{(r)}(2V)\cap \gD(t^2)}  \xi^{\top} \Pi_{Q}\delta(X)|\set{X^i}_{i=1}^n}}_{T_2}.
\beq  
For $(b)$, We have 
    \beq 
    \label{eq507}
\sup_{\delta \in K^{(X,r)}(V) - \theta_0: \Delta^2(\delta)
\le t^2} \xi^{\top}\delta(X) 
&\le \sup_{\delta \in B_{D^{(X,r)}_n}(2V)\cap \gD(t^2)}  \xi^{\top} \Pi_{P}\delta(X)  \\ &+ \sup_{\delta \in B_{D^{(X,r)}_n}(2V)\cap \gD(t^2)}  \xi^{\top} \Pi_{Q}\delta(X),
\beq 

taking expectation conditioned on $\set{X^i}_{i=1}^n$ on both sides, 
we have 
\beq 
\label{eq531}
&\ept{\sup_{\delta \in \gF^{(r)}(V) - f_0: \Delta^2(\delta)
\le t^2} \xi^{\top}\delta(X)|\set{X^i}_{i=1}^n} 
\\&\le \underbrace{\ept{\sup_{\delta \in \gF^{(r)}(2V)\cap \gD(t^2)}  \xi^{\top} \Pi_{P}\delta(X)|\set{X^i}_{i=1}^n}}_{T_1} + \underbrace{\ept{\sup_{\delta \in \gF^{(r)}(2V)\cap \gD(t^2)}  \xi^{\top} \Pi_{Q}\delta(X)|\set{X^i}_{i=1}^n}}_{T_2}.
\beq  

Thus, the rest proof will proceed to bound each terms individually.  
\end{proof}

\begin{lemma}
\label{lemma17}
Let $\delta \in \gF^{(r)}(V)$. Then 
\beq 
\norm{\delta}_n^2 \le \max\set{\norm{\delta}_{\infty}, 1}\Delta^2(\delta)n^{-1}. 
\beq 
\end{lemma}

\begin{proof}
We notice that 
\beq 
    n\norm{\delta}_n^2 &= \sum_{i: \abs{\delta(X^i)} \le 1} \abs{\delta(X^i)}^2 + \sum_{i: \abs{\delta(X^i)} > 1} \abs{\delta(X^i)}^2 \\
    &\le \sum_{i: \abs{\delta(X^i)} \le 1} \abs{\delta(X^i)}^2 + \norm{\delta}_{\infty}\sum_{i: \abs{\delta(X^i)} > 1} \abs{\delta(X^i)}\\
    &\le \max \set{\norm{\delta}_{\infty}, 1} \paren{\sum_{i: \abs{\delta(X^i)} \le 1}\abs{\delta(X^i)}^2 + \sum_{i: \abs{\delta(X^i)} > 1} \abs{\delta(X^i)}} \\
    &= \max \set{\norm{\delta}_{\infty}, 1} \Delta^2(\delta).
\beq 
Thus we conclude that 
\beq 
\norm{\delta}_n \le \max\set{\norm{\delta}_{\infty}^{1/2}, 1}\Delta(\delta)n^{-1/2}. 
\beq 
\end{proof}

\begin{lemma}
\label{lemma18}
    If $p_j \in P_j$ with $\norm{p_j} = 1$, where $p_j$, $P_j$ follow the definitions and notations in Definition~\ref{definition4}, then we have 
\beq 
\norm{p_j}_{\infty} \le \frac{c_3\log n}{n^{1/2}},  
\beq 
where $c_3$ depends on the degree $r$ of $P_j$.
\end{lemma}

\begin{proof}
   Let $\phi_{\ell}(\nu)$ be the  $\ell$th degree  orthogonal polynomials $\ell \leq (r - 1)$ which have a domain
in $[0, 1]$ and satisfy 
\beq 
\int_{0}^{1}\phi_{\ell}(\nu)\phi_{\ell'}(\nu)dt =  \ind{l=l'}, 
\beq 
then by Lemma~\ref{lemma15}, $p_j$ can be expressed as 
\beq 
p_{ij} = \sum_{l=0}^{r-1}a_l \phi_{\ell}(X^{(i)}_j), 
\beq for $i=1, \dots, n$ and $a_0, \dots, a_{r-1} \in \R$. 

Let $g_j: \R \rightarrow \R$ be defined as 
\beq 
\label{eq:g_j}
g_j(\nu) = \sum_{l=0}^{r-1} a_l \phi_{\ell}(\nu), 
\beq 
which means $g_j(X^{(i)}_j) = p_{ij}$ for all $j = 1, \dots d$. 

If we split the interval $[0,1]$ into $n$ pieces, and let  $A_1 = [0, X^{(1)}_j)$, $A_i = [X^{(i-1)}_j, X^{(i)}_j)$ for all $i >1$,  and define 
\beq 
w = \max_{\substack{i=1,\dots, n-1\\ j = 1, \dots, d}} \abs{X^{(i)}_j - X^{(i-1)}_j}.
\beq  
then we have,

\beq 
{}&\abs{\sum_{l=0}^{r-1}a_l^2 - \sum_{i=1}^n g_j(X^{(i)}_j)^2(X^{(i)}_j - X^{(i-1)}_j)} \\ &= \abs{ \int_{0}^{1} g_j(u)^2du - \sum_{i=1}^n g_j(X^{(i)}_j)^2(X^{(i)}_j - X^{(i-1)}_j)}\\
&= \abs{\sum_{i=1}^n \int_{A_i}g_j(u)^2du -  \sum_{i=1}^n \int_{A_i} g_j(X^{(i)}_j)^2 du}\\
&\le \sum_{i=1}^n \int_{A_i}\abs{g_j(u)^2 - g_j(X^{(i)}_j)^2}du\\
&\le  \sum_{i=1}^n \int_{A_i} \norm{ (g_j^2)'}_{\infty}\abs{u - X^{(i)}_j}du\\ 
&\le  \sum_{i=1}^n \norm{  (g_j^2)'}_{\infty} \frac{(u-X^{(i)}_j)^2}{2}\mid_{u= X^{(i)}_j}^{X^{(i-1)}_j}\\ 
& \lesssim \norm{  (g_j^2)'}_{\infty}\frac{\log^2 n}{n}. 
\beq 
The first equality follows from the definition of $g_j$ in \eqref{eq:g_j}.

The last inequality follows since
\beq 
w  \lesssim \frac{c\log n}{n}, 
\beq by the Assumption~\ref{assumption1}. 

By applying this, we have  
\beq
\abs{X^{(i-1)}_j - X^{(i)}_j} \le \frac{\log n}{n}. 
\beq 

At the same time, we have 
\beq 
  \left( g(\nu)^2 \right)^{'} = \bigg(\sum_{l=0}^{r-1}a_l^2\phi_{\ell}(\nu)^2 + \sum_{l\ne l'}a_l a_{l'} \phi_{\ell}(\nu)\phi_{\ell'}(\nu)\bigg)^{'}
\beq 
Thus, the above inequality is continued as 
\beq 
\abs{\sum_{l=0}^{r-1}a_l^2 - \sum_{i=1}^n g_j(X^{(i)}_j)^2(X^{(i)}_j - X^{(i-1)}_j)} 
&\lesssim \frac{\log^2 n}{n}\norm{  (g_j^2)'}_{\infty} \\
&\le \frac{c_2\log^2 n\norm{a}^2_{\infty}}{n} \le \frac{c_2\log^2 n}{n}\sum_{l=0}^{r-1}a_l^2, 
\beq for some constant $c_2>0$ that only depends on $r$. Thus for large enough $n$,
we can have $\sum_{l=0}^{r-1}a_l^2 \leq \sum_{i=1}^n g_j(X^{(i)}_j)^2(X^{(i)}_j - X^{(i-1)}_j)$, or  
\bea 
\sum_{l=0}^{r-1}a_l^2 \lesssim \sum_{i=1}^n g_j(X^{(i)}_j)^2(X^{(i)}_j - X^{(i-1)}_j) + \frac{c_2\log^2 n}{n}\sum_{l=0}^{r-1}a_l^2, 
\bea which further implies 
\bea 
\sum_{l=0}^{r-1}a_l^2 \lesssim \frac{1}{1-\frac{c_2\log^2 n}{n}} g_j(X^{(i)}_j)^2(X^{(i)}_j - X^{(i-1)}_j). 
\bea 
Thus, overall, we have 

\beq
\label{eq816}
\sum_{l=0}^{r-1}a_l^2 &\lesssim \frac{1}{1-\frac{c_2\log^2 n}{n}}\sum_{i=1}^n g_j(X^{(i)}_j)^2(X^{(i)}_j - X^{(i-1)}_j)\\
& \le \frac{1}{1-\frac{c_2\log^2 n}{n}}\sum_{i=1}^n g_j(X^{(i)}_j)^2 \max_j \abs{X^{(i)}_j - X^{(i-1)}_j} \\
& \lesssim \frac{1}{1-\frac{c_2\log^2 n}{n}}\frac{\log n}{n}\sum_{i=1}^n g_j(X^{(i)}_j)^2 \\
&\leq \frac{1}{1-\frac{c_2\log^2 n}{n}}\frac{\log n}{n}.
\beq
The last inequality holds by $\norm{p_j}_2 = 1$.
Furthermore, for large $n$, we have 
\bea 
\frac{1}{1-\frac{c_2\log^2 n}{n}} \leq  c_2\log n.
\bea 
Thus, \eqref{eq816} can be further bounded as 
\bea 
\sum_{l=0}^{r-1}a_l^2 &\lesssim  \frac{c_2\log^2 n}{n},
\bea 
where we have absorbed the constants that depend on $r$ and other constants into a single $c_2$.  
Therefore, 
\beq 
\norm{p_j}_{\infty} & = \max_{i = 1, \dots, n} \abs{p_{ij}} = \max_{i = 1, \dots, n} \abs{\sum_{l=0}^{r-1}a_l \phi_{\ell}(X^{(i)}_j)}
\\ & \le \max_{i = 1, \dots, n}\sum_{l=0}^{r-1}\abs{g_{j\ell}(X^{(i)}_j)}\abs{a_l} \\ &\le \norm{a}_{\infty} \max_{x \in [0,1]} \sum_{l=0}^{r-1}\abs{g_{j\ell}(x)}\\  &\le   \frac{c_2^{1/2}(\log n)}{n^{1/2}}\max_{x\in[0,1]} \sum_{l=0}^{r-1}\abs{g_{j\ell}(x)}\\
&\le \frac{c_3^{1/2}(\log n)}{n^{1/2}},
\beq  where the third inequality holds by using $\norm{a}_{\infty} \le \sqrt{\sum_{l=0}^{r-1} a_l^2}$ and the last inequality holds because of $\max_{x\in[0,1]} \sum_{l=0}^{r-1}\abs{g_{j\ell}(x)} = O(1)$ then the claim follows. 
\end{proof}

The following \cref{lem:hinv}, \cref{lem:dpinv} and \cref{lemma19} will show that the $\norm{q_j}_{\infty} \le V$.

We reserve the letter $S$ save for $X_j$. 
We will first define, for $r \geq 2$, 
\bea 
\Xi^{(r)} = \text{diag}&\left(S^{(r)} - S^{(1)}, \dots,S^{(n)} - S^{(n-r+1)}\right). 
\bea 
We then define the falling factorial basis matrix,
$H \in \R^{n\times n}$, by
\begin{equation}
\label{eq:h}
H_{ik} = h_k(S^{(i)}), \;\;\; k,i=1,\ldots n,
\end{equation} here $h_k$ denotes the falling factorial basis over $S$, as defined in Definition~\ref{f_f_r}. 

\begin{lemma}[Lemma 2 in \cite{wang2014falling}]
\label{lem:hinv}
If $H^{(r-1)}$ is the $(r-1)$th order falling factorial basis matrix
defined over the ordered inputs
\( S^{(1)} < \dots < S^{(n)} \), and $D^{(S, r)}_n$ is the $(r)$th order discrete
difference operator defined over the same inputs, then
\begin{equation}\label{eq:invH}
(H^{(r-1)})^{-1} = \left[\begin{array}{c} C' \\
\frac{1}{(r-1)!} \cdot D^{(S, r)}_n
\end{array}\right],
\end{equation}
for an explicit matrix $C' \in \R^{r\times n}$.  If we let $A_i$
denote the $i$th row of a matrix $A$, and $e_i$ be element of the canonical basis of subspace of $\R^{n}$,  then $C'$ has first
row $C'_1 = e_1^\top$, and subsequent rows
\begin{equation*}
C'_{i} =
\left[ \frac{1}{(i-2)!} \cdot (\Xi^{(i)})^{-1} \cdot D^{(S,i-1)}_n\right],
\;\;\; i=2,\ldots r.
\end{equation*}
Then the last $n-r$ rows
of $(H^{(r-1)})^{-1}$ are given  by $D^{(S, r)}_n/(r-1)!$.
\end{lemma}

\begin{lemma}[Lemma 13 in \cite{wang2015trend}]
		\label{lem:dpinv}
		Let $\Pi$ be the projection onto $\text{row}(D^{(S,r)}_n)$, the $(r)$th order discrete difference operator has pseudoinverse   
		$$(D^{(S,r)}_n)^\ddagger =  \Pi H_2^{(r-1)} / (r-1)!,$$
		where 
		\smash{$H_2^{(r-1)} \in \R^{n\times (n-r)}$} is the last $n-r$   
		columns of the $r-1$th order falling factorial basis matrix $H^{(r-1)}$.  
	\end{lemma}
	
	\begin{proof} The proof follows a same approach to Lemma 13 in \cite{wang2015trend}; for the reader's convenience, we provide it here again. 
		We abbreviate $D=D^{(S, r)}_n$, and consider the linear system
		\begin{equation}
		\label{eq:ddt}
		DD^\top x = Db
		\end{equation}
		in $x$, where $b\in\R^n$ is arbitrary. We seek an expression for
		$x=(DD^\top)^{-1} D^\top = (D^\dag)^\top b$, and this will tell us the 
		form of $D^\dag$.

        Define
		\begin{equation*}
		\tilde D = \left[\begin{array}{c} C \\ D \end{array}\right] \in
		\R^{n\times n},
		\end{equation*}
		where $C \in \R^{r\times n}$ is equal to $(r-1)!C'$ where $C'$ is the matrix that collects the first row of each lower order difference operator, defined in Equation~\eqref{eq:invH}. From \cref{lem:hinv}, we know that
		\begin{equation*}
		\tilde D^{-1} = H/(r-1)!,
		\end{equation*}
		where $H=H^{(r-1)}$ is falling factorial basis matrix of order $r-1$,
		evaluated over inputs
\( S^{(1)} < \dots < S^{(n)} \).  With this in mind, consider the
		expanded linear system
		\begin{equation}
		\label{eq:ddt2}
		\left[\begin{array}{cc} CC^\top & CD^\top \\
		DC^\top & DD^\top \end{array}\right]
		\left[\begin{array}{c} w \\ x \end{array}\right] =
		\left[\begin{array}{c} a \\ Db \end{array}\right].
		\end{equation}
		The second equation reads
		\begin{equation*}
		DC^\top w + DD^\top x = Db,
		\end{equation*}
		and so if we can choose $a$ in \eqref{eq:ddt2} so that at the
		solution we have $w=0$, then $x$ is the solution in \eqref{eq:ddt}.
		The first equation in \eqref{eq:ddt2} reads
		\begin{equation*}
		CC^\top w + CD^\top x = a,
		\end{equation*}
		i.e.,
		\begin{equation*}
		w = (CC^\top)^{-1} ( a - CD^\top x ).
		\end{equation*}
		That is, we want to choose
		\begin{equation*}
		a = CD^\top x =CD^\top(DD^\top)^{-1} D b = C\Pi b, 
		\end{equation*}
		where $\Pi$ is the projection onto row space of $D$.
		Thus we can reexpress \eqref{eq:ddt2} as
		\begin{equation*}
		\tilde D\tilde D^\top
		\left[\begin{array}{c} w \\ x \end{array}\right] =
		\left[\begin{array}{c} C \Pi b \\ Db \end{array}\right]
		= \tilde D \Pi b
		\end{equation*}
		and, using \smash{$\tilde D^{-1} = H/(r-1)!$},
		\begin{equation*}
		\left[\begin{array}{c} w \\ x \end{array}\right] = H^\top
		\Pi b/(r-1)!.
		\end{equation*}
		Finally, writing $H_2$ for the last $n-r$ columns of $H$, we
		have $x = H_2^\top \Pi b/(r-1)!$, as desired.
	\end{proof}

\begin{lemma}
\label{lemma19}
Let $TV^{(r)}$ be defined in \ref{def:J}.  
For $\delta \in \gF^{(r)}(V)$, let $\delta = \sum_{j=1}^d\delta_j$, with $\delta_j \in B_{TV^{(r)}}(V)$ that is defined in Definition~\eqref{eq197}, 
with the definitions and notations in Definition~\ref{definition4}. Let the vector $\delta_j(X_j)\in \R^n$ be the evaluation of $\delta_j$ on the ordered inputs $X_j^{(1)}, \dots, X_j^{(n)}$. 
 Let  $D^{(X_j,r)}_n$ be the discrete trend filtering operator defined in Definition~\eqref{eq:higher_order}, let $q_j = \Pi_{Q_j} \delta_j(X_j)$, where $Q_j$ is row space of $D^{(X_j,r)}_n$, then we have 
    \beq 
    \norm{q_j}_{\infty} \le V\log n.
    \beq 

 

\end{lemma}
\begin{proof} 


We know the $\Pi_{Q_j}$ is the projection onto the $\text{row}(D^{(X_j, r)}_n)$.  
Now let  $M_j = (D^{(X_j,r)}_n)^{\ddagger} \in \R^{n \times (n-r)}$, the pseudoinverse of the $r$th order discrete difference operator. From \cref{lem:dpinv}, we know that 
    \beq 
    (D^{(X_j,r)}_n)^{\ddagger} = \Pi_{Q_j}H_{2,j}^{(r-1)}/(r-1)!, 
    \beq 
    where $H_{2,j}^{(r-1)} \in \R^{n \times (n-r)}$ contains the last $n-r$ columns of the falling factorial 
    basis matrix of order $(r-1)$, evaluated over $X^{(1)}_j, \dots, X^{(n)}_j$, such that for $i \in \set{1,\dots, n}$, and $s\in \set{1, \dots, n-r}$,  and $j \in \set{1, \dots, d}$,
    \beq 
    (H_{2,j}^{(r-1)})_{i,s} = h_{s,j}(X_j^{(i)}), 
    \beq 
    where 
    \beq 
    h_{s,j}(x) = \prod_{l=1}^{r-1}(x - x_j^{s+l})\ind{x\ge x_j^{s+l}},
    \beq where we reserve letter $x_j$ for $X_j$.  
    Then for $e_i$ an element of the canonical basis of subspace of $\R^{n}$, and $P_j = \text{null}(D^{(X_j,r)}_n)$, we have 
    \beq 
    \norm{e_i^{\top} M_j}_{\infty} &\le \norm{\Pi_{Q_j}e_i}_1 \norm{H_{2,j}}_{\infty}/(r-1)!\\
    &\le (\norm{e_i}_1 + \norm{\Pi_{P_j}e_i}_1) \norm{H_{2,j}}_{\infty}/(r-1)! \\ 
    &\le (1 + \norm{\Pi_{P_j}e_i}_1)/(r-1)!.
    \beq 
    The first inequality follows from Holder’s inequality, the second from the triangle inequality and the
last by the definition of $H_{2,j}^{(r-1)}$, with each entry is less equal to 1. 
Now, we let $\nu_1, \dots, \nu_r$ be an orthonormal basis of $P_j$. Then we have 
\beq 
\norm{\Pi_{P_j} e_i}_1 = \norm{\sum_{j=1}^r (e_i^{\top}\nu_j)\nu_j}_1 \le \sum_{j=1}^r \norm{\nu_j}_{\infty} \norm{\nu_j}_1 \le \sum_{j=1}^r\norm{\nu_j}_{\infty} n^{1/2}.
\beq 
Based on Lemma~\ref{lemma18}, we have obtain that 
\beq 
\label{eq:m_j}
\norm{M_j}_{\infty} = \max_{s = 1\dots n-r } \max_{i = 1, \dots, n}  (M_j)_{i,s} =  \max_{i = 1, \dots, n} \norm{e_i^{\top} M_j}_{\infty} = O(\log n). 
\beq
Then by the equivalence of Problems~\eqref{eq303} and \eqref{eq298}, 
for $\delta_j \in B_{TV^{(r)}}(V)$, we have 
\beq 
\label{eq3840}
D^{(X_j,r)_n} \delta_j(X_j) \in  B_{D^{(X,r)}_n}(V), 
\beq where $B_{D^{(X,r)}_n}(V)$ is defined in Definition~\eqref{def:ball of trend filtering}. 
Then we get for any $\delta_j \in B_{TV^{(r)}}(V)$, we have 
\beq 
\norm{\Pi_{Q_j} \delta_j(X_j)}_{\infty} = \norm{(D^{(X_j,r)}_n)^{\ddagger} D^{(X_j,r)}_n \delta_j(X_j)}_{\infty} \le \norm{M_j}_{\infty} \norm{D^{(X_j,r)}_n \delta_j(X_j)}_1 \le V \log n .
\beq
The last inequality follows from the \eqref{eq3840} and \eqref{eq:m_j}. 
\end{proof}

\begin{lemma}
\label{lemma20} For $\delta \in \gF^{(r)}(V)$,  let $\delta = \sum_{j=1}^d\delta_j$, with $\delta_j \in B_{TV^{(r)}}(V)$ that is defined in Definition~\eqref{def:J}. 
Let $g_{j1}, \dots, g_{jr}$ be the orthonormal basis for $P_j$ such that $\norm{g_{j\ell}} = 1, \text{ for } l \in \set{1,\dots, r}$, with the definitions and notations in Definition~\ref{definition4}.
Denote 
$ 
\alpha_{jl} = \innerprod{\delta_j(X_j)}{ g_{j\ell}}_2,
$
and put all $\alpha_{jl}$ into a vector 
\beq 
\alpha = (\alpha_{11},\dots, \alpha_{1r},\dots, \alpha_{d1},\dots, \alpha_{dr}). 
\beq 
Then we have 
\beq 
\norm{\alpha}_2 \le \frac{\norm{p}_2}{\lambda_{\min}(\Gamma^{\top}\Gamma)^{1/2}} = \frac{\norm{p}_n}{\lambda_{\min}(\frac{1}{n}\Gamma^{\top}\Gamma)^{1/2}} \le \frac{\max\set{\norm{\delta}_{\infty}^{1/2}, 1}\Delta(\delta)n^{-1/2} + V\log n}{\lambda_{\min}(\frac{1}{n}\Gamma^{\top}\Gamma)^{1/2}}, 
\beq 
where $\Gamma \in \mathbb{R}^{n\times rd}$ is a matrix constructed from the basis  $g_{j1}, \dots, g_{jr}$ for each $j \in \{1, \dots, d\}$, such that the columns of $\Gamma$ consist of $g_{j1}, \dots, g_{jr}$ for all $j \in \{1, \dots, d\}$. 
\end{lemma}
\begin{proof}


Let  $\text{col}(v_j)$ be the column space of $v_j$, where $v_j$ is defined in \cref{lemma18}. 
Let $g_{j1}, \dots, g_{jr}$ be the orthonormal basis for  $\text{col}(v_j)$  by taking Gram-Schmidt Procedure, since $P_j = \text{col}(v_j)$, so $g_{j1}, \dots, g_{jr}$ is a set of basis for $P_j$ that 
\beq 
   \norm{g_{j\ell}} = 1, \quad l \in \set{1,\dots, r}.
\beq 

Recall
\beq 
\alpha_{jl} = \innerprod{\delta_j(X_j)}{ g_{j\ell}}_2,
\beq 
and put all $\alpha_{jl}$ into a vector 
\beq 
\alpha = (\alpha_{11},\dots, \alpha_{1r},\dots, \alpha_{d1},\dots, \alpha_{dr}) \in \R^{rd}.
\beq 
Then we have 
\beq 
\label{ref:prepresentation}
    p = \Pi_{P_1}\delta_1(X_j^{(1)}) + \dots +\Pi_{P_d}\delta_d(X_d)  = \sum_{j=1}^d\sum_{l=1}^r \innerprod{\delta_j(X_j)}{ g_{j\ell}}_2 g_{j\ell} = \Gamma\alpha, 
    \beq where $\Gamma \in \R^{n\times rd}$ the basis matrix constructed by basis $g_{j1}, \dots, g_{jr}$, with $j=1,\ldots, d$. 
We have 
\beq 
\norm{p} = \norm{\Gamma\alpha},
\beq and 
\beq 
\norm{p} \ge \lambda_{\min}(\Gamma^{\top}\Gamma)^{1/2}\norm{\alpha}, 
\beq 
thus, we have 
\beq 
\norm{\alpha} \le \frac{\norm{p}}{\lambda_{\min}(\Gamma^{\top}\Gamma)^{1/2}} = \frac{\norm{p}_n}{\lambda_{\min}(\frac{1}{n}\Gamma^{\top}\Gamma)^{1/2}}.
\beq 
By Lemma~\ref{lemma17}, Lemma~\ref{lemma19}, the triangle inequality, we have 
\beq 
\norm{p}_n \le \norm{q}_n + \norm{\delta}_n \le \max\set{\norm{\delta}_{\infty}^{1/2}, 1}\Delta(\delta)n^{-1/2} + V\log n
\beq 

Thus, we have 
\beq 
\label{eq:alpha}
\norm{\alpha} \le \frac{\max\set{\norm{\delta}_{\infty}^{1/2}, 1}\Delta(\delta)n^{-1/2} + V\log n}{\lambda_{\min}(\frac{1}{n}\Gamma^{\top}\Gamma)^{1/2}}.
\beq 
 \end{proof}


\begin{lemma}
\label{lemma21}
    Let $t>0$ and for $\delta \in \gF^{(r)}(V)$ with $\Delta^2(\delta) \le t^2$ and $\delta_j \in B_{TV^{(r)}}(V)$. Then, with the definitions and notations in Definition~\ref{definition4}, it holds that
    \beq 
\norm{p_j}_{\infty} \le  \gamma(t,d,n) \coloneqq  \frac{c_3\sqrt{d}}{\lambda_{\min}(\frac{1}{n}\Gamma^{\top}\Gamma)^{1/2}}
\paren{\frac{t \log n}{n} + \frac{t^2 \log n}{n} + \frac{V\log n}{n^{1/2}}},
\beq where $c_3>0$ is a constant depends on $r$.
\end{lemma}
\begin{proof}
 For $j=1, \dots, d$,  let $g_{j1}, \dots, g_{jr}$ be the orthonormal basis for $P_j$ in Definition~\ref{definition4} such that 
\beq 
   \norm{g_{j\ell}}_2 = 1, \quad l \in \set{1,\dots, r}.
\beq 
      
Then we have  
    \beq 
    \label{representation of polynomial}
    p = \Pi_{P_1}\delta_1(X_j^{(1)}) + \dots \Pi_{P_d}\delta_d(X_d)  = \sum_{j=1}^d\sum_{l=1}^r \innerprod{\delta_j(X_j)}{ g_{j\ell}}_2 g_{j\ell}.
    \beq 

Thus, for any $j$, we have for all $i$, if we express $g_{jl}$ in component form as 
\bea 
g_{jl} = (g_{jl}^1, g_{jl}^2, \dots, g_{jl}^n).
\bea 

\beq 
\label{eq1276}
    \abs{p_{ij}} &=\abs{\innerprod{\delta_j(X_j)}{ g_{j1}}_2 \cdot g_{j1}^i + \dots + \innerprod{\delta_j(X_j)}{ g_{jr}}_2 \cdot g_{jl}^i }\\
    &\le \sum_{l=1}^r\abs{\alpha_{jl}} \abs{g_{jl}^i}\le  \sum_{l=1}^r\abs{\alpha_{jl}} \paren{\max_{l}\norm{g_{jl}}_{\infty}} \\
    &\le \abs{\alpha}_1 \frac{c_3\log n}{n^{1/2}}\le \frac{c_3\sqrt{d}\log n}{n^{1/2}}\norm{\alpha}_2 \\
    &\le \frac{c_3\sqrt{d}\log n\paren{\max\set{\norm{\delta}_{\infty}^{1/2},1}\Delta(\delta) n^{-1} + n^{-1/2}V\log n}}{\lambda_{\min}(\frac{1}{n}\Gamma^{\top}\Gamma)^{1/2}}.
\beq 
The first inequality is by Triangle Inequality, and $\alpha_{jl}$ is defined in Lemma~\ref{lemma20}. The second inequality holds by the fact $X^i_j \in [0,1]$, and $g_{j1}, \dots, g_{jr}$ are the orthonormal basis for $P_j$, then by Lemma~\ref{lemma18}, $\norm{g_{j\ell}}_{\infty} \le \frac{c_3\log n}{n^{1/2}}$. The last inequality follows from the Lemma~\ref{lemma20}.  Also, $c_3>0$ is a constant.

Furthermore, we have 
\beq 
\norm{\delta}_{\infty} & = \max\set{ \max_{\set{i:\abs{\delta(X^i)}\ge 1}}\abs{\delta(X^i)}, \max_{\set{i:\abs{\delta(X^i)}< 1}}\abs{\delta(X^i)}} \\
&\le \max\set{\sum_{i\in\set{i:\abs{\delta(X^i)}\ge 1}} \abs{\delta(X^i)},1}\\
&\le \max\set{\sum_i \abs{\delta(X^i)}\ind{\abs{\delta(X^i)}\ge 1} + \sum_i \abs{\delta(X^i)}^2\ind{\abs{\delta(X^i)}\le 1},1} \\
&= \max\set{\Delta^2(\delta),1}\le \max\set{t^2,1}.
\beq 

Thus, we have $\norm{\delta}_{\infty}^{1/2} \le \max\set{\Delta(\delta),1} \le \max\set{t,1}.$ 

Finally, we have 
\beq 
\label{eq1295}
\sum_{l=1}^r
\innerprod{\delta_j(X_j)}{ g_{j\ell}}_2 &\le 
\sum_{l=1}^r \abs{\alpha_{jl}} \\
&\le \frac{c_3\sqrt{d}\paren{\max\set{\norm{\delta}_{\infty}^{1/2},1}\Delta(\delta) n^{-1/2} + V\log n}}{\lambda_{\min}(\frac{1}{n}\Gamma^{\top}\Gamma)^{1/2}}\\& \le 
\frac{c_3\sqrt{d}}{\lambda_{\min}(\frac{1}{n}\Gamma^{\top}\Gamma)^{1/2}}
\paren{\frac{t}{n^{1/2}} + \frac{t^2 }{n^{1/2}} + V\log n},
\beq 
where the second inequality follows from \eqref{eq:alpha} and $\alpha$ has length $rd$; the last inequality follows from the condition $\Delta(\delta) \le t$, and $\norm{\delta}_{\infty}^{1/2} \le \max\set{t,1}$, and $\max\set{t^2, t} \le t^2 + t$ for $t\ge 0$.
Also from \eqref{eq1276} we conclude that 

\beq 
\norm{p_j}_{\infty} &\le \frac{c_3\sqrt{d}\log n\paren{\max\set{\norm{\delta}_{\infty}^{1/2},1}\Delta(\delta) n^{-1} + n^{-1/2}V\log n}}{\lambda_{\min}(\frac{1}{n}\Gamma^{\top}\Gamma)^{1/2}}\\ &\le \frac{c_3\sqrt{d}\log n\paren{\max\set{\Delta^2(\delta),\Delta(\delta)} n^{-1} + n^{-1/2}V\log n}}{\lambda_{\min}(\frac{1}{n}\Gamma^{\top}\Gamma)^{1/2}} \\& \le 
\frac{c_3\sqrt{d}}{\lambda_{\min}(\frac{1}{n}\Gamma^{\top}\Gamma)^{1/2}}
\paren{\frac{t \log n}{n} + \frac{t^2 \log n}{n} + \frac{V\log n}{n^{1/2}}}. 
\beq 
 



\end{proof}

\begin{lemma}[Bounding The $T_1$]
\label{proposition2} Let $t>0$ and for $\delta \in \gF^{(r)}(V)$ with $\Delta^2(\delta) \le t^2$ and $\delta_j \in B_{TV^{(r)}}(V)$. Consider
 the definitions and notations in Definition~\ref{definition4}. 
Let $g_{j1}, \dots, g_{jr}$ be an orthonormal basis for $P_j$. 
It holds that 
\beq 
\label{T1_bound}
\ept{\sup_{\delta \in \gF^{(r)}(2V)\cap \gD(t^2)}   \xi^{\top} \Pi_{P}\delta(X)|\set{X^i}_{i=1}^n} \le   \frac{c_3d^{3/2}\sqrt{\log d}}{\lambda_{\min}(\frac{1}{n}\Gamma^{\top}\Gamma)^{1/2}}\paren{\frac{t}{n^{1/2}} + \frac{t^2 }{n^{1/2}} + V\log n},
\beq which goes to zero for large $n$, where $c_3$ is a constant depends on $r$ and $R$, where $\Gamma$ is defined in \cref{lemma20}.
\end{lemma}
\begin{proof}
By the Definition~\ref{definition4} of $\Pi_{P}$, 
\beq 
\label{inq:t1}
\xi^{\top} \Pi_{P}\delta(X) 
&= \xi^{\top} (p_1 + \dots + p_d) 
\\&\le \abs{\sum_{l=1}^r \sum_{j=1}^d\xi^{\top}g_{j\ell} {\delta_j(X_j)}^{\top}g_{j\ell}} \le \sum_{l=1}^r \sum_{j=1}^d \abs{{\delta_j(X_j)}^{\top}g_{j\ell}}\abs{ \xi^{\top}g_{j\ell}}\\
&\le d \bracket{ \paren{\max_{l = 1, \dots, r, j = 1, \dots, d } \abs{\xi^{\top}g_{jl}}} \paren{\max_{j = 1, \dots, d}\sum_{l=1}^r  \abs{{\delta_j(X_j)}^{\top}g_{j\ell}}}}\\
    &\le c_3d^{3/2}  \paren{\max_{l = 1, \dots, r, j = 1, \dots, d} \abs{\xi^{\top}g_{jl}}} \paren{\frac{1}{\lambda_{\min}(\frac{1}{n}\Gamma^{\top}\Gamma)^{1/2}}
\paren{\frac{t}{n^{1/2}} + \frac{t^2 }{n^{1/2}} + V\log n}},
\beq where the first inequality is based on  \eqref{representation of polynomial} and last inequality is based on \eqref{eq1295}.

Thus, we have 
\beq 
{}&\ept{\sup_{\delta \in \gF^{(r)}(2V)\cap \gD(t^2)} \xi^{\top} \Pi_{P}\delta(X)|\set{X^i}_{i=1}^n} \\&\le \frac{c_3d^{3/2}}{\lambda_{\min}(\frac{1}{n}\Gamma^{\top}\Gamma)^{1/2}}
\paren{\frac{t}{n^{1/2}} + \frac{t^2 }{n^{1/2}} + V\log n} \paren{ \sum_{l = 1, \dots, r} \E\bracket{\max_{j = 1, \dots, d} \abs{\xi^{\top}g_{jl}}|\set{X^i}_{i=1}^n}}  \\&\le \frac{c_3d^{3/2}\sqrt{\log d}}{\lambda_{\min}(\frac{1}{n}\Gamma^{\top}\Gamma)^{1/2}}\paren{\frac{t}{n^{1/2}} + \frac{t^2 }{n^{1/2}} + V\log n}.
\beq 

The first inequality follows from line \eqref{inq:t1} and 
\bea 
\max_{l = 1, \dots, r, j = 1, \dots, d} \abs{\xi^{\top}g_{jl}} \le \sum_{l = 1, \dots, r} \max_{j = 1, \dots, d} \abs{\xi^{\top}g_{jl}}. \bea 
The last inequality holds since $\xi^{\top}g_{jl}$ are sub-Gaussian random variables with parameter $1$, and then by applying (2.66) from \cite{wainwright2019high}. 
Thus we have 
\beq 
\ept{\sup_{\delta \in \gF^{(r)}(2V)\cap \gD(t^2)}   \xi^{\top} \Pi_{P}\delta(X)|\set{X^i}_{i=1}^n} \le  \frac{c_3d^{3/2}\sqrt{\log d}}{\lambda_{\min}(\frac{1}{n}\Gamma^{\top}\Gamma)^{1/2}}\paren{\frac{t}{n^{1/2}} + \frac{t^2 }{n^{1/2}} + V\log n}.
\beq

\end{proof}

\begin{lemma}
\label{lemma22}
Let $t>0$  and  for $\delta \in \gF^{(r)}(V)$ with $\Delta^2(\delta) \le t^2$ and $\delta_j \in B_{TV^{(r)}}(V)$.  
For $q$ defined in \cref{definition4} \textbf{6} where $\Pi_{Q}\delta(X) = q$, define $\Delta_q^2 \coloneqq \sum_{i=1}^n \min\set{\abs{q^i}, (q^i)^2}$, where $q^i$ is defined in \cref{definition4} \textbf{6}. 
For $j=1, \dots, d$,  let $g_{j1}, \dots, g_{jr}$ be the orthonormal basis for $P_j$, with the definitions and notations in Definition~\ref{definition4}. Then, it holds that  
\beq 
\Delta_q^2 \le h(t,d, n) \coloneqq  \paren{2t^2 + 2t^2\gamma(t,d,n) + 4n c^2(r,\lambda_{\text{min}})d^3\paren{\frac{t^2 \log^2 n}{n^2} + \frac{t^4 \log^2 n}{n^2} + \frac{V^2\log^2 n}{n}} },
\beq 
with $\gamma(t,d,n) \coloneqq  \frac{c_3d^{3/2}}{\lambda_{\min}(\frac{1}{n}\Gamma^{\top}\Gamma)^{1/2}}
\paren{\frac{t \log n}{n} + \frac{t^2 \log n}{n} + \frac{V\log n}{n^{1/2}}}$ as the same quantity in Lemma~\ref{lemma21}. 
\end{lemma}
\begin{proof}
Based on Lemma~\ref{lemma21},  we have 
\beq 
\norm{p}_{\infty} = \sum_{j=1}^d\norm{ p_j}_{\infty} \le  \gamma(t,d,n) \coloneqq \frac{c_3d^{3/2}}{\lambda_{\min}(\frac{1}{n}\Gamma^{\top}\Gamma)^{1/2}}
\paren{\frac{t \log n}{n} + \frac{t^2 \log n}{n} + \frac{V\log n}{n^{1/2}}}. 
\beq 

  Based on triangle inequality, we have  
  \beq 
  \label{ref:triangleqp}
  \norm{q - \delta}_{\infty} = \norm{p}_{\infty} \le \gamma(t,d,n).
  \beq 
  
Also we find that 
\beq 
 \Delta_q^2 & = \sum_{i=1}^n \min\set{\abs{q^i}, (q^i)^2} \\ & \le  \sum_{i=1}^n \abs{q_{i}}\ind{\abs{\delta(X^i)} >1} + \sum_{i=1}^n q_{i}^2 \ind{\abs{\delta(X^i)} \le 1}. 
\beq 
Then we have 
\beq 
\Delta_q^2 &\le \sum_{i=1}^n (\abs{\delta(X^i)} + \gamma(t,d,n))\ind{\abs{\delta(X^i)} >1} + \sum_{i=1}^n (2{\delta(X^i)}^2 + 2(\gamma(t,d,n)^2)) \ind{\abs{\delta(X^i)} \le 1} \\
&\le t^2(\gamma(t,d,n)) + \sum_{i=1}^n (\abs{\delta(X^i)})\ind{\abs{\delta(X^i)} >1} + 2\sum_{i=1}^n \delta_{i}^2\ind{\abs{\delta(X^i)} \le 1} + 2\sum_{i=1}^n (\gamma(t,d,n)^2)\ind{\abs{\delta(X^i)} \le 1} \\
& \le 2t^2 + 2t^2\gamma(t,d,n) + 2n\gamma(t,d,n)^2\\
& \le \paren{2t^2 + 2t^2\gamma(t,d,n) + 6n c^2(r,\lambda_{\text{min}})d^3\paren{\frac{t^2 \log^2 n}{n^2} + \frac{t^4 \log^2 n}{n^2} + \frac{V^2\log^2 n}{n}} }.  
\beq 
Where the first inequality follows since $\norm{q - \delta}_{\infty} \le \gamma(t,d,n)$ and $a^2 + b^2 \ge 2ab$, 
where the second inequality follows from the fact that 
\beq 
\abs{\set{i \in \set{1,\dots, n}: \abs{\delta(X^i)} > 1}} \le t^2,
\beq  which holds because $\Delta^2(\delta) \le t^2$. The third inequality follows by the definition of 

$\Delta^2(\delta)\coloneqq \sum_{i=1}^n \min \set{\abs{\delta(X^i)}, \delta(X^i)^2} = \sum_{i=1}^n (\abs{\delta(X^i)}\ind{\abs{\delta(X^i)} >1}) + \sum_{i=1}^n (\delta(X^i)^2\ind{\abs{\delta(X^i)} \le 1})$. 
\end{proof}

\begin{lemma}[Bounding $T_2$] Let $t \asymp n^{-r/(2r+1)} V^{1/(2r+1)}\max\set{1, V^{(2r-1)/(4r+2)}} $, for $\delta \in \gF^{(r)}(V)$ with $\Delta^2(\delta) \le t^2$, and with the definitions and notations in Definition~\ref{definition4}. For a positive constant $c_4$ , it holds that 
\label{lemma:boudingT2}
    \beq 
\ept{\sup_{\delta \in \gF^{(r)}(2V)\cap \gD(t^2)}   \frac{1}{n}\xi^{\top} \Pi_{Q}\delta(X)|\set{X^i}_{i=1}^n} &\le c_4(d)^{1/2+1/2r}n^{-1/2} V^{1/2r}m^{1-1/(2r)},  
\beq with $m$ is a quantity depends on $t,n,d,r, \lambda_{\text{min}}$.
\end{lemma}

\begin{proof}
For $q$ defined in \cref{definition4} \textbf{6} where $\Pi_{Q}\delta(X) = q$, define $\Delta_q^2 \coloneqq \sum_{i=1}^n \min\set{\abs{q^i}, (q^i)^2}$, where $q^i$ is defined in \cref{definition4} \textbf{6}. First, by Lemma~\ref{lemma17},  we have 
\beq 
\label{eq152}
\norm{q}_n
&\le    \max\set{\norm{q}_{\infty}^{1/2},1}\Delta_qn^{-1/2} \\
&\le \max\set{d^{1/2}\max_{j=1\dots d}\norm{q_j}_{\infty}^{1/2},1}\Delta_qn^{-1/2}\\
&\le c(r, \lambda_{\text{min}},d) \max\set {V^{1/2},1}\left(t+d^{3/4}\paren{\frac{ V^{1/2}\sqrt{\log n} }{ n^{1/4}} + \frac{t^{1/2} \sqrt{\log n}}{n^{1/2}} + 
\frac{t \log n}{ n^{1/2}} }t \right. 
\\
&\left. \quad\quad\quad + d^{3/2}(\frac{t \log n}{n^{1/2}}+\frac{t^2 \log n}{n^{1/2}} + V\log n) \right) n^{-1/2}  \\
&\le c(r, \lambda_{\text{min}},d) \max\set {V^{1/2},1}\paren{\paren{1 + \frac{t^{1/2} \sqrt{\log n}}{n^{1/2}} + 
\frac{t \log n}{ n^{1/2}} }t +  \frac{t^2 \log n}{n^{1/2}} + V\log n} n^{-1/2}
\\& =: m,  
\beq where the second line we use triangle inequality, where in the third line we used~\cref{lemma19} to control $\norm{q_j}_\infty$,  ~\cref{lemma22} to control $\Delta_q$, and $\sqrt{a_1^2 + \dots + a_n^2} \le a_1 + \dots + a_n$ for positive numbers $a_1, \dots, a_n$, and $c(r, \lambda_{\text{min}},d)$ is introduced that depends on $d,r, \lambda_{\text{min}}$, where in the fourth line we assume $t \geq 1$ and $t \asymp n^{-r/(2r+1)} V^{1/(2r+1)}\max\set{1, V^{(2r-1)/(4r+2)}} $ and combine some lower order terms.

Then we have  

\beq 
\label{eq1172}
&\ept{\sup_{\delta \in \gF^{(r)}(2V)\cap \gD(t^2)}   \frac{1}{n}\xi^{\top} \Pi_{Q}\delta(X)|\set{X^i}_{i=1}^n} 
\\ 
& \leq \ept{\sup_{g \in \gF^{(r)}(2V)\cap B_n(m)   }   \frac{1}{n}\sum_{i=1}^n \xi^i g(X^i) |\set{X^i}_{i=1}^n} \\
&\le c_{\text{Dud}}\frac{1}{\sqrt{n}}\int_{0}^{m} \sqrt{\log N(\epsilon, \norm{\cdot}_{n}, \gF^{(r)}(2V))}d\epsilon\\
&\le c_4n^{-1/2}(d)^{1/2+1/2r} V^{1/2r}\int_{0}^{m}\epsilon^{-1/2r}d\epsilon\\
&=  c_4n^{-1/2}(d)^{1/2+1/2r} V^{1/2r}m^{1-1/(2r)}. 
\beq

The second line follows from \eqref{eq152}, the third line applies Dudley's entropy integral \cite{dudley1967sizes}, with $c_{\text{Dud}}$ a positive constant, and the fourth line follows from Lemma 15 in \cite{sadhanala2019additive}, specifically from the middle derivations and the third-to-last displayed equation on page 45, where $c_4$ is a positive constant.
    
\end{proof}

\section{PROOF OF THEOREM~\ref{theorem1}}
\begin{proof}

For $t>1$,
by Theorem~\ref{theorem7}, 

\beq 
\sP\paren{\frac{1}{n}\Delta^2(\widehat f - f_0) > t^2\mid \set{X^i}_{i=1}^n} & \le \frac{c}{t^2}\gR_n\paren{\set{f-f_0: f \in \gF^{(r)}(V)} \cap \set{f-f_0: \Delta^2( f - f_0)\le nt^2}} \\ 
& = \frac{c}{t^2}\ept{ 
\sup_{\delta \in \gF^{(r)}(V) - f_0: \Delta^2(\delta)
\le nt^2} \frac{1}{n}\xi^{\top}\delta(X)  |\set{X^i}_{i=1}^n},
\beq 
where $c$ is a positive constant that depends on $L$ and $\ubar{f}$ in Assumption~\ref{assumptionA}. 
We denote the numerator of the right-hand side as $T_n(t)$ that depends on $t$ and $n$.
By \eqref{decomposition_t1_t2} in Lemma~\ref{lemma11}, we decompose the $T_n(t)$ into two terms. Equation \eqref{T1_bound} in Lemma~\ref{proposition2} gives the bound for the first term, and Equation~\eqref{eq1172} in Lemma~\ref{lemma:boudingT2} gives the bound for the second term. Then we have the upper bound  of $T_n(t)$ as  
\beq 
\label{eq1149}
T_n(t) & \le  cc_4n^{-1/2}(d)^{1/2+1/2r} V^{1/2r}m_1(t)^{1-1/(2r)} \\
& \le C_1n^{-r/(2r+1)} V^{1/(2r+1)} m_1(t) + C_1n^{-2r/(2r+1)} V^{2/(2r+1)}, 
\beq 
where in the first inequality, we used \eqref{eq1172}, which is of higher order compared to \eqref{T1_bound}. The second inequality follows from \cref{Lemma_11}, where we set \( a = n^{-1/2} d^{1/2+1/2r} V^{1/2r} \), \( b = m_1(t) \), and \( w = 1/r \). Here, \( C_1 \) is a constant that depends on \( d \), $c$ and $c_1$, and \( m_1(t) \) depends on \( t, n, d, r, \lambda_{\text{min}} \), as derived in \eqref{eq152}. Specifically, 
\begin{equation}
\begin{aligned}
&m_1(t) = c(r, \lambda_{\text{min}},d) \max\left\{ V^{1/2}, 1 \right\} \left( \left( 1 + t^{1/2} \sqrt{\log n} + t \log n \right)t + t^2 \log^2 n + \log n V n^{-1/2} \right), 
\end{aligned}
\end{equation} 
where \( c(r, \lambda_{\text{min}}, d) \) depends on \( r, \lambda_{\text{min}}, d \).


Define the function
\beq 
g_n(t) = \frac{T_n(t)}{t^2},
\beq 
setting 
\beq 
t \asymp n^{-r/(2r+1)} V^{1/(2r+1)}\max\set{1, V^{(2r-1)/(4r+2)}} , 
\beq 
we obtain the result of Theorem~\ref{theorem1}, by verifying that
\beq
\label{limit_derivation}
\lim_{c_1\rightarrow \infty}\sup_{n\ge 1}g_n(c_1 t) &= 
\lim_{c_1\rightarrow \infty}\sup_{n\ge 1} \frac{T_n(c_1t)}{c_1^2t^2
} \\ & \leq  \lim_{c_1\rightarrow \infty}\sup_{n\ge 1} \frac{C_1n^{-r/(2r+1)} V^{1/(2r+1)} m_1(c_1t)}{c_1^2t^2} \\&\quad + \lim_{c_1\rightarrow \infty}\sup_{n\ge 1} \frac{C_1n^{-2r/(2r+1)} V^{2/(2r+1)}}{c_1^2t^2}
\\
&\le \lim_{c_1\rightarrow \infty}\sup_{n\ge 1}  c(r,\lambda_{\text{min}},d)\Bigl(c_1^{-1} +t^{1/2}c_1^{-1}(\log n)^{1/2}  +  tc_1^{-1} \log n \Bigr. \\
&\quad \Bigl.+ tc_1^{-1}(\log n)^2+ c_1^{-1}Vn^{-1/2}t^{-1}\Bigr)\\
&  =0, 
\beq with such a choice of $t$. 
\end{proof}
\begin{lemma}[Lemma 11 in \cite{sadhanala2019additive}]
\label{Lemma_11}    

For any \( a, b \geq 0 \), and any \( w \),
\[
ab^{1 - w/2} \leq a^{1/(1+w/2)} b + a^{2/(1+w/2)}.
\]

\end{lemma}
\begin{proof}
Note that either \( ab^{1 - w/2} \leq a^{1/(1+w/2)} b \) or \( ab^{1 - w/2} \geq a^{1/(1+w/2)} b \), and in the latter case we get \( b \leq a^{1/(1+w/2)} \), so \( ab^{1 - w/2} \leq a^{2/(1+w/2)} \). 
\end{proof}

\section{DISCUSSION OF INFLUENCE OF $\tau$ FOR THEOREM~\ref{theorem1} }
\label{app_discussion_tau}
Note that the influence of $\tau$ on the $O_{pr} \paren{n^{-2r/(2r+1)} V^{2/(2r+1)}\max\set{1, V^{(2r-1)/(2r+1)}}}$ rate is solely through a constant $\tilde c$ hiden in the symbbol $O_{pr}$. The constant $\tilde c$ satisfies $\tilde c = O((L\ubar{f})^{-1})$, where $L$ and $\ubar{f}$ are defined as in Assumption~\ref{assumptionA}, with $F_{Y^i|X^i}(f^i_0) = \tau$. Please also refer to Lemma~\ref{lemma13}, where $c$ is given by $\tilde c^{-1}$. Hence, in principle, we could let $\tau \rightarrow 0$ or $\tau \rightarrow 1$, but in that case, the rate in \eqref{rate_fix_d} would have to be inflated by the factor of $(L\ubar{f})^{-1}$. 

\section{ASSUMPTION FOR THEOREM~\ref{theorem2}}
\label{app_discussion_assumption6}

Furthermore, to produce an error rate linearly dependent on the growing dimension $d$, we include the assumption below. This appeared as Assumption A3 in \cite{sadhanala2019additive}. 
\begin{assumption}
\label{assumption6}
The input points $X^i, i = 1, \dots, n$ are i.i.d. from a continuous distribution $\mathscr{F}$ supported on $[0, 1]^d$, that decomposes as $\mathscr{F} = \mathscr{F}_1 \times \cdots \times \mathscr{F}_d$, where the
density of each $\mathscr{F}_j$ is lower and upper bounded by positive constants $b_1$ and $b_2$, respectively,  $\text{ for } j = 1,\dots, d$. 
\end{assumption}


\begin{remark}
Assumption \ref{assumption6} appeared as  Assumption A.3 in  \cite{sadhanala2019additive}. As mentioned in \cite{sadhanala2019additive}, this condition is fairly restrictive, since it requires the input distribution \( \mathscr{F} \) to have independent coordinates. The reason we use this assumption: when \( \mathscr{F} = \mathscr{F}_1 \times \cdots \times \mathscr{F}_d \), additive functions enjoy a key decomposability property in terms of the (squared) \( L_2 \) norm defined with respect to \( \mathscr{F} \). In particular, if \( f = \sum_{j=1}^d f_j \) has components with \( L_2 \) mean zero, denoted by
\[ 
\bar{f}_j = \int_0^1 f_j(x_j) d\mathscr{F}_j(x_j) = 0, \quad j = 1, \ldots, d,
\]
then we have
\beq 
\label{norm_decomp}
\left\| \sum_{j=1}^d f_j \right\|_{L_2}^2 = \sum_{j=1}^d \| f_j \|_{L_2}^2.
\beq 

This is explained by the fact that each pair of components \( f_j, f_l \) with \( j \ne l \) are orthogonal with respect to the \( L_2 \) inner product, since
\[ 
\langle f_j, f_l \rangle_{L_2} = \int_{[0,1]^2} f_j(x_j) f_l(x_l) d\mathscr{F}_j(x_j) d\mathscr{F}_l(x_l) = \bar{f}_j \bar{f}_l = 0.
\]

The above orthogonality, and thus the decomposability property in \eqref{norm_decomp}, is only true because of the product form \( \mathscr{F} = \mathscr{F}_1 \times \cdots \times \mathscr{F}_d \). Such decomposability is not generally possible with the empirical norm (the inner products between components do not vanish even if all empirical means are zero). In the proof of Theorem~\ref{theorem2}, we move from considering the empirical norm of the error vector to the \( L_2 \) norm, in order to leverage the property in~\eqref{norm_decomp}, which eventually leads to an error rate that has a polynomial dependence on the dimension \( d \). 
\end{remark}

\section{SUPPLEMENTARY LEMMAS FOR THEOREM~\ref{theorem2}}
\label{app_lemmas_theorem2}
\vspace{2pt}
\begin{lemma} Let $f: [0,1]^d \rightarrow \R$ with $f = \sum_{j=1}^d f_j$. Let Assumption~\ref{assumption6} holds, furthermore, 
if 
\beq 
\int_{0}^1 f_j(x_j) d\mathscr{F}_j(x_j) = 0, 
\beq for $j=1,\dots, d$, then 
it holds that 
\beq 
\innerprod{f}{f}_{L_2} = \sum_{j=1}^d\innerprod{f_j}{f_j}_{L_2}. 
\beq     
\end{lemma}
\begin{proof}
    By the definition of $L_2$ inner product, and Assumption~\ref{assumption6},we have 
    \beq 
    \innerprod{f}{f}_{L_2} &= \int_{[0,1]^d} f(x)^2 d\mathscr{F}(x) \\
   &= \int_{[0,1]^d} \paren{\sum_{j=1}^df_j(x_j)}^2 d\mathscr{F}_1(X^1)\times \dots \times d\mathscr{F}_d(x_j)\\
    &= \sum_{j=1}^d\int_{[0,1]} f_j(x_j)^2 d\mathscr{F}_j(x_j) + 2\sum_{j\neq k} \int_{[0,1]^2} f_j(x_j)f_k(x_k)d\mathscr{F}_j(x_j)d\mathscr{F}_k(x_k) \\ 
    &= \sum_{j=1}^d\int_{[0,1]} f_j(x_j)^2 d\mathscr{F}_j(x_j) + 2\sum_{j\neq k} \int_{[0,1]} f_j(x_j)d\mathscr{F}_j(x_j) \int_{[0,1]} f_k(x_k)d\mathscr{F}_k(x_k) \\
    &= \sum_{j=1}^d\innerprod{f_j}{f_j}_{L_2}. 
    \beq 
    Thus we conclude that 
    $\innerprod{f_j}{f_k}_2 = 0$ for $j\neq k$ and 
    \beq 
    \norm{\sum_{j=1}^d f_j}_{L_2}^2 = \sum_{j=1}^d \norm{f_j}_{L_2}^2. 
    \beq 
\end{proof}

\begin{lemma}[Bounding the $T_1$ with growing $d$]
\label{proposition2growingd} 
For $\delta \in \gF^{(r)}(V)$, and let $\Pi_{P}$ follow the definitions and notations in Definition~\ref{definition4}.  it holds that \beq 
\ept{\sup_{\delta \in \gF^{(r)}(2V)\cap \gD(t^2)}   \xi^{\top} \Pi_{P}\delta(X)|\set{X^i}_{i=1}^n} \le   \frac{c_3d^{3/2}\sqrt{\log d}}{\lambda_{\min}(\frac{1}{n}\Gamma^{\top}\Gamma)^{1/2}}\paren{\frac{t}{n^{1/2}} + \frac{t^2 }{n^{1/2}} + V\log n}, 
\beq which goes to zero for large $n$.
\end{lemma}
\begin{proof}
    The proof is similar to the proof of Lemma~\ref{proposition2}, thus we omit that. 
\end{proof}

\begin{lemma}
\label{lemma30}
Let $t_n = c\sqrt{d} n^{-r/(2r+1)} V^{1/(2r+1)}$, where $c$ is a positive constant. This value is chosen such that, if we set $m = t_n$, then
\beq \frac{\gR_n(B_{TV^{(r)}}^d(2V) \cap B_{L_2}(m))}{m} \le \frac{m}{c} \beq
holds with high probability.
\end{lemma}
\begin{proof}
Let $\beta$ be a $d$ dimensional new variable satisfying $\norm{\beta}_2 \le m$. We then use it to control the radius of the $L_2$ ball.
Let first use the decomposability property in $L_2$ norm by Assumption~\ref{assumption6}, which is 
\beq
\norm{\sum_{j=1}^d f_j}_{L_2}^2 = \sum_{j=1}^d \norm{f_j}_{L_2}^2.  
\beq 
We let $m$ to denote the upper bound such that
\beq 
\norm{\sum_{j=1}^d f_j}_{L_2}^2 \le m^2 \implies \sum_{j=1}^d \norm{f_j}_{L_2}^2 \le m^2.
\beq

Then we have two equivalent function spaces:
\beq 
\label{eq_func_decomp}
\set{f_j: \sum_{j=1}^d \norm{f_j}_{L_2}^2 \le m^2} = \set{f_j: \norm{f_j}_{L_2} \le \abs{\beta_j}, \norm{\beta}_2 \le m}.
\beq   
To get the local critical radius of $B_{TV^{(r)}}^d(2V)$, by $L_2$ orthogonality of the components of functions in $B_{TV^{(r)}}^d(2V)$, we first have  
\beq 
\label{eq1729}
\sup_{g \in B_{TV^{(r)}}^d(2V)\cap B_{L_2}(m)}   \frac{1}{n}\sum_{i=1}^n \xi^i g(X^i)& \le  \sup_{\norm{\beta}_2\le m } \sup_{\substack{g_j \in B_{TV^{(r)}}(2V)\cap B_{L_2}(\abs{\beta_j})\\ j = 1, \dots, d}}  \frac{1}{n}\sum_{i=1}^n \xi^i \sum_{j=1}^d g_j (X^i_j) \\&\le 
\sup_{\norm{\beta}_2\le m} \sum_{j=1}^d \sup_{g_j \in B_{TV^{(r)}}(2V)\cap B_{L_2}(\abs{\beta_j})}  \frac{1}{n}\sum_{i=1}^n \xi^i  g_j (X^i_j).
\beq 

The first inequality follows from \eqref{eq_func_decomp}.


Then we have with probability at least $1-1/n^2$, it holds that 
\begin{equation}
\label{eq_159_hold}
\begin{aligned}
&\sup_{g_j \in B_{TV^{(r)}}(2V)\cap B_{L_2}(\abs{\beta_j})}  \frac{1}{n}\sum_{i=1}^n \xi^i g_j (X^i_j)  \\
&\le 
c \gR_n(B_{TV^{(r)}}(2V)\cap B_{L_2}(\abs{\beta_j})) + c\sqrt{\frac{\log n}{n}}\paren{\sup_{g_j \in B_{TV^{(r)}}(2V)\cap B_{L_2}(\abs{\beta_j})} \norm{g_j}_n}\\
&\le c \paren{\gR(B_{TV^{(r)}}(2V)\cap B_{L_2}(\abs{\beta_j})) + \frac{\log n}{n}+\sqrt{\frac{\log n}{n}}\paren{\sup_{g_j \in B_{TV^{(r)}}(2V)\cap B_{L_2}(\abs{\beta_j})} \norm{g_j}_n}}\\
&\le c \paren{\gR(B_{TV^{(r)}}(2V)\cap B_{L_2}(\abs{\beta_j})) + \frac{\log n}{n}+\sqrt{\frac{\log n}{n}}\sqrt{2} \paren{\max\set{\abs{\beta_j}, r_{nj}}}}\\
&\le c\frac{V^{1/2r}\abs{\beta_j}^{1-2/r}}{\sqrt{n}} + c\sqrt{\frac{\log n}{n}}\paren{\max\set{\abs{\beta_j}, r_{nj}}}.
\end{aligned}
\end{equation}

The first line holds by Theorem 3.6 in \cite{wainwright2019high}, where $\gR_n$ is defined in \eqref{definition213}, 
the second line holds by Lemma A.4 in \cite{bartlett2005local}, and third inequality holds by Lemma 3.6 in \cite{bartlett2005local}, where $r_{nj}$ is the critical radius of the function class $B_{TV^{(r)}}(2V)$, the smallest $b$ such that 
\beq 
\frac{\gR(B_{TV^{(r)}}(2V)\cap B_{L_2}(b))}{b} \le \frac{b}{c}. 
\beq 
By Lemma~\ref{lemma29}, we have $r_{nj} = n^{-r/(2r+1)} V^{1/(2r+1)}$, 
and the last inequality in \eqref{eq_159_hold} follows by Lemma~\ref{lemma29}, also uses $\frac{\log n}{n } \le r_{nj}\sqrt{\frac{\log n}{n }}$ for $n$ sufficiently large.  Call an event based on the result of \eqref{eq_159_hold} that simultaneously holds for all $j=1,\dots, d$
\beq 
\label{event_e}
\gE = \set{\sup_{g_j \in B_{TV^{(r)}}(2V)\cap B_{L_2}(\abs{\beta_j})}  \frac{1}{n}\sum_{i=1}^n \xi^i g_j (X^i_j) \leq c\frac{V^{1/2r}\abs{\beta_j}^{1-2/r}}{\sqrt{n}} + c\sqrt{\frac{\log n}{n}}\paren{\max\set{\abs{\beta_j}, r_{nj}}} \text{ for }  j = 1, \dots, d}.  
\beq 
By a union bound, $\sP(\gE) \geq 1-d/n^2$.  

Meanwhile, on $\gE^c$, by Lemma~\ref{lemma19},  
we have $\norm{g_j (X^i_j)}_{\infty} \le V \log n$.

Thus, 
back to \eqref{eq1729}, we have 
\beq 
\sup_{g \in B_{TV^{(r)}}(2V)\cap B_{L_2}(m)}   \frac{1}{n}\sum_{i=1}^n \xi^i g(X^i)&\le  \sup_{\norm{\beta}_2 \le m}\sum_{j=1}^d \paren{c\frac{V^{1/2r}\abs{\beta_j}^{1-1/2r}}{\sqrt{n}} + \sqrt{\frac{\log n}{n}}\paren{\max{\abs{\beta_j}, r_{nj}}}}+ \frac{dV \log n}{n^2}\\
&\le c\paren{\frac{V^{1/2r}d^{(2r+1)/(4r)}m^{1-1/2r}}{\sqrt{n}} + \sqrt{\frac{d\log n}{n}}m +  dr_{nj}^2} + \frac{dV \log n}{n^2}.
\beq

In the second line, we use Holder's inequality $a^{\top}b \leq \norm{a}_p \norm{b}_q$ for the first term, with $p = 4/(2 + 1/r)$, 
and $q = 4/(2 - 1/r)$; for the second term, we use $\max\set{a, b} \le a + b$ for  $a>0$ and $b>0$. We also use the fact in \eqref{eq_func_decomp}  to get the bound $\norm{\beta}_1 \le 
\sqrt{d}m$, and the
fact that $r_{nj} \sqrt{\log n /n}
\le r_{nj}^2$ for large enough $n$.
Thus, we have 
\beq 
\label{rademacher_162}
\gR_n(B_{TV^{(r)}}(2V)\cap B_{L_2}(m)) &\le \ept{\sup_{g \in B_{TV^{(r)}}(2V)\cap B_{L_2}(m)}   \frac{1}{n}\sum_{i=1}^n \xi^i g(X^i)|\set{X^i}_{i=1}^n}\\
&\le c\paren{\frac{V^{1/2r}d^{(2r+1)/(4r)}m^{1-1/2r}}{\sqrt{n}} + \sqrt{\frac{d\log n}{n}}m +  dr_{nj}^2} + \frac{dV \log n}{n^2}\\ 
& \lesssim c\paren{\frac{V^{1/2r}d^{(2r+1)/(4r)}m^{1-1/2r}}{\sqrt{n}} + \sqrt{\frac{d\log n}{n}}m +  dr_{nj}^2},
\beq 
the third line follows by seeing the $\frac{dV \log n}{n^2}$ is a lower order term given $V = \Omega(1)$. 
Denoting the right-hand side of first line of \eqref{rademacher_162} as $T_r$, we  have 


\beq
\label{eq1772}
T_r & \le  c \paren{n^{-1/2}(d)^{1/2+1/4r} V^{1/2r}m^{1-1/(2r)} + \sqrt{\frac{d\log n}{n}}m +  dr_{nj}^2} \\
& \le c n^{-r/(2r+1)}d^{1/2} V^{1/(2r+1)} m + \sqrt{\frac{d\log n}{n}}m + c d^{(2+2/r)/(2+1/r)}n^{-2r/(2r+1)} V^{2/(2r+1)}, 
\beq 
where in the second inequality, we have used the similar bound in Equation~\eqref{eq1149}, it can be verified that for $m = c\sqrt{d}n^{-r/(2r+1)} V^{1/(2r+1)}$, the upper bound in \eqref{eq1772} is at most $m^2/c$. Therefore, this is an upper bound on the critical radius of $B_{TV^{(r)}}(2V)$, which completes the proof. 
\end{proof}

\begin{lemma}[Bounding $T_2$ with growing $d$]
\label{lemma23growingd}
Let
\bea 
t \asymp d^{(r+1)/(2r+1)}n^{-r/(2r+1)} V^{1/(2r+1)}\max\set{1, V^{(2r-1)/(4r+2)}} \log n, 
\bea  and for $\delta \in \gF^{(r)}(V)$ with $\Delta^2(\delta) \le t^2$, and let $p_j$, $q_j$, $p$, $q$, $P_j$, $Q_j$, $\Pi_{P_j}$, $\Pi_{Q_j}$, $\Pi_{P}$, and $\Pi_{Q}$  follow the definitions and notations in Definition~\ref{definition4}. For a positive constant $c$, 
it holds that
\begin{equation}
\begin{aligned}
&\ept{\sup_{\delta \in \gF^{(r)}(2V)\cap \gD(t^2)}   \frac{1}{n}\xi^{\top} \Pi_{Q}\delta(X)|\set{X^i}_{i=1}^n} \\&\le cd^{(2r+2)/(2r+1)}n^{-2r/(2r+1)} V^{2/(2r+1)}\max\set{1, V^{(2r-1)/(2r+1)}} \log n. 
\end{aligned}
\end{equation} 
\end{lemma}

\begin{proof}Let $t_n$ be defined as a value of $t$ satisfying
\beq \frac{\gR_n(B_{TV^{(r)}}^d(2V) \cap B_{L_2}(t))}{t} \le \frac{t}{c}, \beq
where $B_{TV^{(r)}}^d$ is given in \eqref{ball_tv_d}.
We write $l$ as \beq 
\label{def_l}
\sqrt{2}m \vee \left(t_n + \frac{\log n}{n}\right),\beq where $m$ is defined in Equation~\eqref{eq152}.
Let $d$ dimensional  $\beta$ be a new variable satisfying $\norm{\beta }_2 \le l$. We then use it to control the radius of the $L_2$ ball and we aim to have control over it. Then we have two equivalent function spaces:
\beq 
\set{f_j: \sum_{j=1}^d \norm{f_j}_{L_2}^2 \le l^2} = \set{f_j: \norm{f_j}_{L_2} \le \abs{\beta_j}, \norm{\beta}_2 \le l}.
\beq   

Define event $\gE_1 \coloneq \set{B_n(m) \subseteq B_{L_2}(\sqrt{2}m\bigvee  (t_{n} + \frac{\log n}{n}))}$, and $\gE_1^c$ is its complement.

First, we can upper bound 
$\sup_{\delta \in \gF^{(r)}(2V)\cap \gD(t^2)}   \frac{1}{n}\xi^{\top} \Pi_{Q}\delta(X)$, by splitting into two parts.
Therefore,  we have 
    \beq  
    \mathbb  E & \bracket{\sup_{\delta \in \gF^{(r)}(2V)\cap \gD(t^2)}   \frac{1}{n}\xi^{\top} \Pi_{Q}\delta(X)|\set{X^i}_{i=1}^n} \\ & \leq  \ept{\sup_{\delta \in \gF^{(r)}(2V)\cap \gD(t^2)}   \frac{1}{n}\xi^{\top} \Pi_{Q}\delta(X)\ind{\gE_1}|\set{X^i}_{i=1}^n} + \ept{\sup_{\delta \in \gF^{(r)}(2V)\cap \gD(t^2)}   \frac{1}{n}\xi^{\top} \Pi_{Q}\delta(X)\ind{\gE_1^c}|\set{X^i}_{i=1}^n} \\ &\le 
\ept{\sup_{g \in \gF^{(r)}(2V)\cap B_n(m)}   \frac{1}{n}\sum_{i=1}^n \xi^i g(X^i)\ind{\gE_1}|\set{X^i}_{i=1}^n} + \frac{c_1dV\log n}{n} \\&\le 
\ept{\sup_{g \in \gF^{(r)}(2V)\cap B_{L_2}(\sqrt{2}m\bigvee  (t_{n} + \frac{\log n}{n}))}   \frac{1}{n}\sum_{i=1}^n \xi^i g(X^i)|\set{X^i}_{i=1}^n}  + \frac{c_1dV\log n}{n} \\
&\leq \ept{\sup_{\norm{\beta}_2\le l } \sup_{\substack{g_j \in B_{TV^{(r)}}(2V)\cap B_{L_2}(\abs{\beta_j})\\ j = 1, \dots, d}}  \frac{1}{n}\sum_{i=1}^n \xi^i \sum_{j=1}^d g_j (X^i_j)|\set{X^i}_{i=1}^n} + \frac{c_1dV\log n}{n}\\&\le 
\ept{\sup_{\norm{\beta}_2\le l} \sum_{j=1}^d \sup_{g_j \in B_{TV^{(r)}}(2V)\cap B_{L_2}(\abs{\beta_j})}  \frac{1}{n}\sum_{i=1}^n \xi^i  g_j (X^i_j)|\set{X^i}_{i=1}^n} + \frac{c_1dV\log n}{n}.
\beq

The first inequality follows the same as the proof we have shown in Lemma~\ref{lemma:boudingT2}, see \eqref{eq152}.
Lemma 3.6 in \cite{bartlett2005local} gives that $\gE_1$ holds with probability at least $1-1/n$.  
On the event $\gE_1^c$, by Lemma~\ref{lemma19}, we have $\norm{q_j}_{\infty} \le V \log n$, then we have the following bound, 
\bea 
\sup_{\delta \in \gF^{(r)}(2V)\cap \gD(t^2)}   \frac{1}{n}\xi^{\top} \Pi_{Q}\delta(X)\ind{\gE_1^c}\le c_1dV\log n.
\bea

And the second inequality follows since Lemma 3.6 in \cite{bartlett2005local}, which gives the probability of $\gE^c$ as $1/n$. 

And the third inequality follows since the decomposability property in $L_2$ norm by Assumption~\ref{assumption6}, which is 
\beq
\norm{\sum_{j=1}^d f_j}_{L_2}^2 = \sum_{j=1}^d \norm{f_j}_{L_2}^2.  
\beq 

then we have 
\beq 
\norm{\sum_{j=1}^d f_j}_{L_2}^2 \le l^2 \implies \sum_{j=1}^d \norm{f_j}_{L_2}^2 \le l^2,
\beq where $l$ is introduced in \eqref{def_l}.

Thus the third inequality holds. 

The last inequality follows by $\sup(\sum_{j=1}^d a_j) \le \sum_{j=1}^d \sup a_j$.

Following the proof steps that lead to Equation~\eqref{eq1772} in Lemma~\ref{lemma30}, we obtain that 
\beq 
\label{eq1671}
{}& \ept{\sup_{\delta \in \gF^{(r)}(2V)\cap \gD(t^2)}   \frac{1}{n}\xi^{\top} \Pi_{Q}\delta(X)|\set{X^i}_{i=1}^n} \\ &\le c n^{-r/(2r+1)}d^{(1+1/r)/(2+1/r)} V^{1/(2r+1)} l + \sqrt{\frac{d\log n}{n}}l + c d^{(2+2/r)/(2+1/r)}n^{-2r/(2r+1)} V^{2/(2r+1)}\\&= S_1 + S_2 + S_3.
\beq

Based on the proof of Lemma~\ref{lemma:boudingT2}, for $q$ defined in \cref{definition4} \textbf{6} where $\Pi_{Q}\delta(X) = q$, we have 
\beq 
\norm{q}_n \leq m. 
\beq 
Then we let 
\beq  
m &= \max\set {d^{1/2}V^{1/2},1}\left(t+d^{3/4}\paren{\frac{ V^{1/2}\sqrt{\log n} }{ n^{1/4}} + \frac{t^{1/2} \sqrt{\log n}}{n^{1/2}} + 
\frac{t \log n}{ n^{1/2}} }t \right. 
\\
&\left. \quad\quad\quad + d^{3/2}(\frac{t \log n}{n^{1/2}}+\frac{t^2 \log n}{n^{1/2}} + V\log n) \right) n^{-1/2}.
\beq 
We also get $t_{n} = c\sqrt{d}n^{-r/(2r+1)} V^{1/(2r+1)}$ from Lemma~\ref{lemma30} and $  
l =\sqrt{2}m \vee \left(t_n + \frac{\log n}{n}\right).$ 

For the first term $S_1$ in \eqref{eq1671},
we have 

\beq 
\label{bound_s1}
S_1 &\lesssim \max \Bigg( V^{\frac{2}{2r + 1}} c d^{\frac{2r + \frac{3}{2}}{2r + 1}} n^{- \frac{2r}{2r + 1}}, \\
& C_2 V^{\frac{1}{2r + 1}} d^{\frac{r + 1}{2r + 1}} n^{- \frac{r}{2r + 1}} \max(1, V^{1/2}) 
\Bigg( \frac{V d^{2}}{n^{1/2}} + \frac{\sqrt{d} t^{3/2} \sqrt{\log n}}{n^{1/4}} + \sqrt{d} t^{2} \log^2 n \\
& \quad + \sqrt{d} t^{2} \log n + \frac{d^{2} t \log n}{n^{1/2}} \Bigg) \Bigg).
\beq

which follows by the definition of $l$ in \eqref{def_l}. 
For the second term $S_2$, we have 

\beq
\label{bound_s2}
S_2 &\lesssim \max \Bigg( C_2 \max(1, V^{1/2}) \Bigg( \frac{V d^{5/2} \sqrt{\log n}}{n} + \frac{d t^{3/2} \log n}{n^{3/4}} + \frac{d t^{2} \log n^{2/5}}{n^{1/2}} \\
& \quad + \frac{d t^{2} \log n^{2/3}}{n^{1/2}} + \frac{d^{5/2} t \log n^{2/3}}{n} \Bigg), \quad \frac{V^{1/(2r + 1)} d^{(2r + 1.5)/(2r + 1)} n^{- r/(2r+1)} \sqrt{\log n}}{n^{1/2}} \Bigg).
\beq

which follows by the definition of $l$ in \eqref{def_l}. 
For the third term $S_3$, we have 
\beq 
\label{bound_s3}
S_3 = 
c d^{(2r+2)/(2r+1)}n^{-2r/(2r+1)} V^{2/(2r+1)}.
\beq 
Thus, if we choose 
\beq 
t \asymp d^{(r+1)/(2r+1)}n^{-r/(2r+1)} V^{1/(2r+1)}\max\set{1, V^{(2r-1)/(4r+2)}} \log n,
\beq 


then we have the upper bound of $S_1 + S_2 + S_3$ is all $O(t^2)$.

Using the bounds in \eqref{bound_s1}, \eqref{bound_s2}, and \eqref{bound_s3}, then we come back to \eqref{eq1671},  get 
 \beq  
    \mathbb  E & \bracket{\sup_{\delta \in \gF^{(r)}(2V)\cap \gD(t^2)}   \frac{1}{n}\xi^{\top} \Pi_{Q}\delta(X)|\set{X^i}_{i=1}^n} \\ & \leq cd^{(2r+2)/(2r+1)}n^{-2r/(2r+1)} V^{2/(2r+1)}\max\set{1, V^{(2r-1)/(2r+1)}} \log n + \frac{c_1dV\log n}{n}\\
    &\lesssim cd^{(2r+2)/(2r+1)}n^{-2r/(2r+1)} V^{2/(2r+1)}\max\set{1, V^{(2r-1)/(2r+1)}} \log n. 
\beq

Since the second term $\frac{c_1dV\log n}{n}$ is linearly dependent on $d$, and we see that it is also a lower order term compared to the first term, so we omit that in the third line.

\end{proof}

\begin{lemma}
\label{lemma29}
Let $Q_j$ be defined in Definition~\ref{definition4}. For a positive constant $c$, let $m = cn^{-r/(2r+1)} V^{1/(2r+1)}$, then the local Rademacher complexity of function class $B_{TV^{(r)}}(2V)\cap Q_j\cap B_{L_2}(m)$ satisfies that 
\beq 
\gR(B_{TV^{(r)}}(2V)\cap Q_j\cap B_{L_2}(m)) \le cn^{-2r/(2r+1)} V^{2/(2r+1)}. 
\beq 
\end{lemma}
\begin{proof}
    Consider the empirical local Rademacher complexity, 
\beq 
\gR_n(B_{TV^{(r)}}(2V)\cap B_{L_2}(m) = \ept{\sup_{g_j \in B_{TV^{(r)}}(2V) \cap B_{L_2}(m)}\frac{1}{n}\sum_{i=1}^n \xi^i g_j(X^i_j)|\set{X^i}_{i=1}^n}.
\beq 
Define event $\gE_2 \coloneqq \set{B_{TV^{(r)}}(2V)\cap B_{L_2}(m) \subseteq B_{TV^{(r)}}(2V)\cap B_n(\sqrt{2}m)}$, and let $\gE_2^c$ be its complement. 

Corollary 2.2 of \cite{bartlett2005local} gives $\gE_2$ holds with high probability $1- \eta$, with $\eta$ is a small positive number. 
Thus, we have 
\bea 
&\gR_n(B_{TV^{(r)}}(2V)\cap B_{L_2}(m)) \\ 
& \leq \ept{\sup_{g_j \in B_{TV^{(r)}}(2V) \cap B_{L_2}(m)}\frac{1}{n}\sum_{i=1}^n \xi^i g_j(X^i_j)\ind{\gE_2}|\set{X^i}_{i=1}^n} + \ept{\sup_{g_j \in B_{TV^{(r)}}(2V) \cap B_{L_2}(m)}\frac{1}{n}\sum_{i=1}^n \xi^i g_j(X^i_j)\ind{\gE_2^c}|\set{X^i}_{i=1}^n}.   
\bea 
To bound the first term, notice that
\beq 
\ept{\sup_{g_j \in B_{TV^{(r)}}(2V) \cap B_{L_2}(m)}\frac{1}{n}\sum_{i=1}^n \xi^i g_j(X^i_j)\ind{\gE_2}|\set{X^i}_{i=1}^n} &\le \ept{\sup_{g_j \in B_{TV^{(r)}}(2V) \cap B_n(\sqrt{2}m)}\frac{1}{n}\sum_{i=1}^n \xi^i g_j(X^i_j)|\set{X^i}_{i=1}^n} \\&\le c_{\text{Dud}}\frac{1}{\sqrt{n}}\int_{0}^{\sqrt{2} m} \sqrt{\log N(\epsilon, \norm{\cdot}_{n}, B_{TV^{(r)}}(2V))}d\epsilon\\
&\le c_4n^{-1/2} V^{1/2r}\int_{0}^{m}\epsilon^{-1/2r}d\epsilon\\
&= c_4n^{-1/2} V^{1/2r}m^{1-1/(2r)},
\beq 
where the second line applies Dudley's entropy integral \cite{dudley1967sizes}, with $c_{\text{Dud}}$ as a positive constant, and the third line follows from Lemma 15 in \cite{sadhanala2019additive},  specifically from the middle derivations and the third-to-last displayed equation on page 45, which uses the fact of $TV$ ball, where $B_{TV^{(r)}}$ is defined in Definition~\ref{def:J}. And $c_4$ is a positive constant. Also, on the event $\gE_2^c$, by Lemma~\ref{lemma19}, we have $\norm{q_j}_{\infty} \le V \log n$, then we have the following bound, 
\beq 
\ept{\sup_{g_j \in B_{TV^{(r)}}(2V) \cap B_{L_2}(m)}\frac{1}{n}\sum_{i=1}^n \xi^i g_j(X^i_j)\ind{\gE_2^c}|\set{X^i}_{i=1}^n} &\le cV\eta\log n.
\beq 

Thus, we have 
\beq 
\gR_n(B_{TV^{(r)}}(2V)\cap B_{L_2}(m)) 
&\le c_5n^{-1/2} V^{1/2r}m^{1-1/(2r)}.
\beq


Therefore, we can
upper bound the local Rademacher complexity, splitting the expectation over two events, 
\beq 
\gR(B_{TV^{(r)}}(2V)\cap B_{L_2}(m)) = \ept{\gR_n(B_{TV^{(r)}}(2V)\cap B_{L_2}(m))} \le c_5n^{-1/2} V^{1/2r}m^{1-1/(2r)}, 
\beq 
where $c_5$ is a positive constant. 
Therefore, we have an upper bound on the critical radius $r_{nj}$ is thus given by the solution of
\beq 
\frac{c_5n^{-1/2}V^{1/2r}m^{1-1/(2r)}}{m}  = \frac{m}{c}, 
\beq 
for $m$, which is $m = cn^{-r/(2r+1)} V^{1/(2r+1)}$, this completes the proof. 





\end{proof}

\section{PROOF OF THEOREM~\ref{theorem2}}
\label{app_theorem2}

\begin{proof}
For $t>1$, by Theorem~\ref{theorem7}, 
\beq 
\label{prob_rate_growing_d}
\sP\paren{\frac{1}{n}\Delta^2(\widehat f - f_0) > t^2} & \le \frac{c}{t^2}\gR_n\paren{\set{f-f_0: f \in \gF^{(r)}(V)} \cap \set{f: \Delta^2( f - f_0)\le nt^2}} \\ 
& = \frac{c}{t^2} \ept{ 
\sup_{\delta \in \gF^{(r)}(V) - f_0: \Delta^2(\delta)
\le t^2} \xi^{\top}\delta(X)  |\set{X^i}_{i=1}^n}, 
\beq 

we then have 
\beq 
\sup_{\delta \in \gF^{(r)}(V) - f_0: \Delta^2(\delta)
\le t^2} \xi^{\top}\delta(X) 
&\le \underbrace{\sup_{\delta \in \gF^{(r)}(2V)\cap \gD(t^2)}  \xi^{\top} \Pi_{P}\delta(X)}_{T_1}  \\ &+ \underbrace{\sup_{\delta \in \gF^{(r)}(2V)\cap \gD(t^2)}  \xi^{\top} \Pi_{Q}\delta(X)}_{T_2}.
\beq

Taking expectation conditioned on $\set{X^i}_{i=1}^n$ on both sides, 
we have 
\beq 
&\ept{\sup_{\delta \in \gF^{(r)}(V) - f_0: \Delta^2(\delta)
\le t^2} \xi^{\top}\delta(X)|\set{X^i}_{i=1}^n} 
\\&\le \underbrace{\ept{\sup_{\delta \in \gF^{(r)}(2V)\cap \gD(t^2)}  \xi^{\top} \Pi_{P}\delta(X)|\set{X^i}_{i=1}^n}}_{T_1}  + \underbrace{\ept{\sup_{\delta \in \gF^{(r)}(2V)\cap \gD(t^2)}  \xi^{\top} \Pi_{Q}\delta(X)|\set{X^i}_{i=1}^n}}_{T_2},
\beq  

the result follows by bounds of each two terms above based on Lemma~\ref{proposition2growingd} and Lemma~\ref{lemma23growingd}. 
We see that by choosing 
\beq 
t \asymp d^{(r+1)/(2r+1)}n^{-r/(2r+1)} V^{1/(2r+1)}\max\set{1, V^{(2r-1)/(4r+2)}} \log n, 
\beq 
we have 
\bea 
T_1 &= O\paren{c_3d^{3/2}\paren{\frac{t}{n^{1/2}} + \frac{t^2 }{n^{1/2}} + V\log n}} \\ 
&= O\paren{c_3d^{3/2}d^{(2r+2)/(2r+1)}n^{-(2r+1/2)/(2r+1)} V^{2/(2r+1)}\max\set{1, V^{(2r-1)/(2r+1)}} \log n},
\bea and 
\bea 
T_2 = O\paren{cd^{(2r+2)/(2r+1)}n^{-2r/(2r+1)} V^{2/(2r+1)}\max\set{1, V^{(2r-1)/(2r+1)}} \log n}.  
\bea 

If $d$ does not grow too quickly, $T_1$ remains a lower order term relative to $T_2$ due to its dependence on $n$. Thus, our analysis only focuses on $T_2$. 

Define the function
\beq 
g_n(t) = \frac{G_n(t)}{t^2},
\beq 

where $G_n(t)$ equals to the numerator of the right-hand side of \eqref{prob_rate_growing_d} that depends on $t$ and $n$.
Then we see that given an small positive $\epsilon$, there is a positive constant $c_1$ that depends on $\epsilon$ such that

\beq
\label{limit_derivation_growing_d}
\lim_{c_1\rightarrow \infty}\sup_{n\ge 1}g_n(c_1 t) &= 
\lim_{c_1\rightarrow \infty}\sup_{n\ge 1} \frac{T_n(c_1t)}{c_1^2t^2
} \\ & \leq  \lim_{c_1\rightarrow \infty}\sup_{n\ge 1} \frac{C_1c_1d^{(2r+2)/(2r+1)}n^{-2r/(2r+1)} V^{2/(2r+1)}\max\set{1, V^{(2r-1)/(2r+1)}} \log n}{c_1^2t^2}
\\&=O\paren{\frac{1}{c_1}} < \epsilon, 
\beq
by setting 
$$
t^2 \asymp d^{\frac{2r+2}{2r+1}} n^{-\frac{2r}{2r+1}} V^{\frac{2}{2r+1}}\max\left\{ 1, V^{\frac{2r-1}{2r+1}}\right\}.
$$
Thus, we conclude that 
\bea 
\Delta_n^2(\widehat f - f_0) = 
& O_{pr} \Big( d^{\frac{2r+2}{2r+1}} n^{-\frac{2r}{2r+1}} V^{\frac{2}{2r+1}} \cdot \max\left\{ 1, V^{\frac{2r-1}{2r+1}} \right\} \Big).
\bea which is the conclusion in \cref{theorem2}. 

\end{proof}   

\section{SOLVING PROBLEM~\eqref{73}}
\label{app_backfitting_update}
\subsection{Reformat of Problem~\eqref{73}}
We  reformulate the optimization problem \eqref{73} as
\beq
\label{74}
& \underset{\theta_j \in \mathbb{R}^{n}, z \in \mathbb{R}^{n}}{\text{min}}
& & \sum_{i=1}^n   \rho_\tau (u_j^i - \theta^i_j) + \lambda \norm{D^{(X_j,r)}_nS_jz_j}_1,  \\
& \text{subject to}
& & z_j = \theta_j, \quad \boldsymbol{1}^{\top} \theta_j =0,
\beq 
and the augmented Lagrangian can then be written as
\beq 
\label{75}
L(\theta_j, z_j, s_j, \nu_j) = \sum_{i=1}^n  \rho_\tau (u_j^i - \theta^i_j) + \lambda \norm{D^{(X_j,r)}_nS_jz_j}_1 + \frac{\eta}{2} \| \theta_j - z_j + s_j \|^2 + \frac{\omega}{2} ( \boldsymbol{1}^{\top} \theta_j + \nu_j )^2,
\beq 
where \( \eta \) and \( \omega \) is the penalty parameter that controls step size in the update. 

Thus we initialize the variables $\theta_j^{(0)} = 0$ and $z_{j}^{(0)} = 0$ for $j = 1, \dots, d$. We can solve \eqref{74} iteratively until convergence, with $m$th iteration, we have:
\begin{align}
\label{76}
\theta_j^{(m)} &= \arg\min_{\theta_j \in \mathbb{R}^{n}} \left\{ \sum_{i=1}^n  \rho_\tau (u_j^{i} - \theta^i_j) + \frac{\eta}{2} \| \theta_j - z_j^{(m-1)} + s_j^{(m-1)} \|^2+ \frac{\omega}{2} ( \boldsymbol{1}^{\top} \theta_j + \nu_j^{(m-1)} )^2 \right\}, \\
\label{77}
 \quad z_j^{(m)} &= \arg\min_{z_j \in \mathbb{R}^n} \left\{ \frac{1}{2} \| \theta_j^{(m)} + s_j^{(m-1)} - z_j \|^2 + \frac{\lambda}{\eta} \|D^{(X_j,r)}_nS_jz_j\|_1 \right\}, \\ 
s_j^{(m)} &\leftarrow s_j^{(m-1)} + \theta_j^{(m)} - z_j^{(m)}, \label{eq:s_update} \\ 
\nu_j^{(m)} &\leftarrow \nu_j^{(m-1)} + \boldsymbol{1}^{\top}\theta_{j}^{(m)}. \label{eq:nu_update}
\end{align}

To solve \eqref{76}, we first set 
\beq 
u^{i}_j &\leftarrow Y^i -  \sum_{l<j} \theta^{i(m)}_l - \sum_{l>j} \theta^{i(m-1)}_l, \label{eq:r_update} 
\beq  then we do the following update, 
\beq 
\theta^{i(m)}_j &\leftarrow
\begin{cases}
    \dfrac{1}{\eta+\omega} \left(\tau -\omega \sum_{l\neq j} \theta^{l(m-1)}_j -\omega \nu^{(m-1)} + \eta(z^{i(m-1)}_j - s^{i(m-1)}_j)\right), \\ 
    \quad \text{if } u_j^i > \dfrac{1}{\eta+\omega} \left(\tau -\omega \sum_{l\neq j} \theta^{l(m-1)}_j -\omega \nu^{(m-1)} + \eta(z^{i(m-1)}_j - s^{i(m-1)}_j)\right), \\
    \dfrac{1}{\eta+\omega} \left(\tau-1 -\omega \sum_{l\neq j} \theta^{l(m-1)}_j -\omega \nu^{(m-1)} + \eta(z^{i(m-1)}_j - s^{i(m-1)}_j)\right), \\
    \quad \text{if } u_j^i < \dfrac{1}{\eta+\omega} \left(\tau-1 -\omega \sum_{l\neq j} \theta^{l(m-1)}_j -\omega \nu^{(m-1)} + \eta(z^{i(m-1)}_j - s^{i(m-1)}_j)\right), \\
    u_j^i, \quad \text{otherwise}.
\end{cases} \label{eq:theta_update} 
\beq 
To solve \eqref{77}, we do two cases. 
For $r=1$, we use dynamic programming by \cite{johnson2013dynamic}. For $r>1$, we use ADMM algorithm from  \citep{ramdas2016fast}, see Below~\ref{admm_subproblem_z} for more details.

\subsection{Solve Problem \eqref{77} by ADMM}
\label{admm_subproblem_z}
The minimization Problem \eqref{77} is a univariate trend filtering problem, on unevenly spaced inputs. For $r>1$, the solution to this problem has been well studied, and we solve it using methods described in \cite{ramdas2016fast}, implemented in the  \texttt{trendfilter} function in the \texttt{glmgen} R package.

The standard ADMM approach (e.g., \cite{boyd2011distributed}) is based on rewriting 
problem \eqref{77} as  
\begin{equation}
\label{eq:tf1}
\min_{z_j \in \R^n, \, \alpha_j \in \R^{n-k-1}} \,
\frac{1}{2} \|\theta_j + s_j -z_j\|_2^2 + \lambda/\eta \|\alpha_j\|_1 \;\;\st\;\; 
\alpha_j = D^{(X_j,r)}_nS_jz_j. 
\end{equation}
The augmented Lagrangian can then be written as
\begin{equation*}
g(z_j,\alpha_j,u_j) = \frac{1}{2} \|\theta_j + s_j-z_j\|_2^2 + \lambda/\eta \|\alpha_j\|_1 +
\frac{\rho}{2} \|\alpha_j - D^{(X_j,r)}_nS_jz_j + u_j \|_2^2 -
\frac{\rho}{2}\|u_j\|_2^2,
\end{equation*}
with updates as
\begin{align}
z_j^{(k)} &\leftarrow 
\left(I + \rho (D^{(X_j,r)}_n)^{\top} D^{(X_j,r)}_n\right)^{-1}
\left(\theta_j^{(k)} + s_j^{(k-1)} + \rho (D^{(X_j,r)}_n)^{\top} (\alpha_j^{(k-1)} + u_j^{(k-1)})\right), \label{eq:z_update} \\ 
\alpha_j^{(k)} &\leftarrow 
S_{\lambda/(\eta\rho)} \left(D^{(X_j,r)}_nS_jz_j^{(k)} - u_j^{(k-1)}\right), \label{eq:alpha_update} \\ 
u_j^{(k)} &\leftarrow u_j^{(k-1)} + \alpha_j^{(k)} - D^{(X_j,r)}_nS_jz_j^{(k)}, \label{eq:u_update} 
\end{align}

The $\alpha$-update, where $S_{\lambda/(\eta\rho)}$ denotes coordinate-wise soft-thresholding at the level $\lambda/(\eta\rho)$,

\subsection{Computational Complexity}

In Algorithm~\ref{alg-backfitting}, the backfitting process is conducted, involving $d$ distinct fits of quantile trend filtering. For each fit of quantile trend filtering. The ADMM procedure involves three updates (primal, dual, and slack). The total number of ADMM iterations is denoted by $m$. For each update, updating the primal variable \( \theta_j \) in \eqref{eq:theta_update} in \( O(n) \) time and computing \( u_j \) in \eqref{eq:r_update} in \( O(n) \) time. The slack variable \( z_j \), introduced in Equation~\eqref{74}, is updated via dynamic programming in \( O(n) \) time for \( r=1 \). For \( r>1 \), \( z_j \) can be updated in \( O(n(r+2)^2) \) time using another ADMM approach.
The
update of auxiliary variable 
takes time $O(n-r-1)$. The dual update 
taking time $O(n(r+2))$. The update of $s_j$ in \eqref{eq:s_update} takes time $O(1)$. The update of $\nu_j$ in \eqref{eq:nu_update} takes time $o(n)$. There are $d$ components. And therefore one full iteration of
standard backfitting ADMM updates can be done in linear time $O(dmn)$(considering $r$ 
as a constant).  

\section{CONVERGENCE OF THE UPDATE IN \eqref{eq70}}
\label{appendix_convergence}
We demonstrate the convergence of the update in Equation \eqref{eq:theta_update} through the following theorem.

\begin{theorem} \label{nonconvex0}
    Let $\{\theta^{(t)} = (\theta^{(t)}_1, \ldots, \theta_d^{(t)})\}_{t=0, 1, \ldots}$ represent the parameter updates at the $t$th iteration 
    in BCD in Algorithm~\ref{alg-backfitting}.
    Under the conditions of Theorem \ref{theorem2}, it holds that every cluster point of the sequence $\theta^{(t)}$ generated by the BCD method is a coordinate minimum point of Equation \eqref{eq70}.
\end{theorem}

We have the problem~\eqref{eq70}
\bea  
&
\min_{\substack{\theta_1, \dots, \theta_d \in \R^n}} \sum_{i=1}^n \rho_{\tau}(Y^i -  \sum_{j=1}^d \theta^i_j) + \lambda \sum_{j=1}^d  \norm{D^{(X_j,r)}_nS_j\theta_j}_1 \\
&\text{subject to}\quad \boldsymbol{1}^{\top} \theta_j =0, \;\;\; j=1,\ldots,d. 
\bea 
Define the following functions, 
\bea 
f(\theta_1, \dots, \theta_d;\lambda)\coloneqq f_0(\theta_1, \dots, \theta_d) +  \sum_{j=1}^d f_j(\theta_j).
\bea 

\bea 
f(\theta_1, \dots, \theta_d;\lambda)\coloneqq \min_{\substack{\theta_1, \dots, \theta_d \in \R^n}} \sum_{i=1}^n \rho_{\tau}(Y^i -  \sum_{j=1}^d \theta^i_j) + \lambda \sum_{j=1}^d  \norm{D^{(X_j,r)}_nS_j\theta_j}_1.
\bea 
\bea
f_0(\theta_1, \dots, \theta_d) \coloneqq \min_{\substack{\theta_1, \dots, \theta_d \in \R^n}} \sum_{i=1}^n \rho_{\tau}(Y^i -  \sum_{j=1}^d \theta^i_j),  
\bea 
\bea 
f_j(\theta_j) \coloneqq  \norm{D^{(X_j,r)}_nS_j\theta_j}_1. 
\bea 
We observe that 
\begin{itemize}
    \item[(1)] $f_0$ is continuous on $\text{dom } f_0$.
    \item[(2)] For each $j \in \{1, \ldots, d\}$ and $(\theta_j)_{j \ne k}$, the function $\theta_j \mapsto f(\theta^1, \ldots, \theta_d)$ is quasiconvex and hemivariate.
    \item[(3)] $f_0, f_1, \ldots, f_d$ are lower semicontinuous .
    \item[(4)] $\text{dom } f_0 = Y_1 \times \cdots \times Y_d$, for some $Y_j \subseteq \mathbb{R}^{n}, \, j = 1, \ldots, d$.
\end{itemize}

    
Then by Theorem 5.1 of \cite{tseng2001convergence}, we have the  every
cluster point is a coordinatewise minimum point of $f$.

\section{RELATED WORKS}
\label{app_review_of_things}
\subsection{Review of Additive Models}

Since introduced, 
additive models have been extensively
studied in various contexts of statistics and machine learning, including Cox regression \citep{cox1972regression}, logistic regression \citep{hastie1987non}, exponential family data distributions \citep{hastie2017generalized}, neural networks \citep{thielmann2024neural, agarwal2021neural,shen2021deep, jo2023neural} and and online machine learning~\citep{abbasi2011improved,li2017provably,ding2022syndicated,kang2022efficient}.
Due to their model structure, additive models effectively address the curse of dimensionality \citep{breiman1985estimating, hastie2017generalized}, allowing for more accurate and interpretable predictions. 

\subsection{Review of Trend Filtering}

In the field of image processing, total variation smoothing penalties were introduced by  \citep{rudin1992nonlinear}.  These penalties enable edge detection and permit sharp breaks in gradients, surpassing the limitations of conventional Sobolev penalties. 

Proposed independently by
\cite{rudin1992nonlinear,kim2009ell_1}, trend filtering, which functions as a total variation smoothing penalty, is a
relatively new approach to univariate nonparametric regression.  Notably, the work by \cite{sadhanala2019additive}
studies additive trend filtering estimates, involving regularizing each component function based on the total variation of its $r$th order discrete derivative. Overall, trend filtering is favored for its favorable theoretical and computational properties, largely due to the localized nature of the total variation regularization it employs. 


\subsection{Review of Quantile Regression}

Unlike classical mean regression, quantile regression estimates conditional quantiles of the response variable, making it robust to outliers and providing a more comprehensive view of the relationship between the response and covariates \citep{koenker2005quantile}.

The quantile trend filtering method, as proposed by \cite{brantley2020baseline} and studied in \cite{madrid2022risk}, produces a trend filtering estimator with smooth structures by imposing a penalty consisting of the total variation of the $r$th order discrete derivatives.  The quantile fused lasso is a special case of quantile trend filtering when $r=0$. The risk bound for quantile trend filtering has been studied by \cite{madrid2022risk}, which established the minimax optimality of univariate trend filtering under the quantile loss with minor assumptions of the data generation mechanism. 

\section{IMPORTANT DEFINITIONS AND EXAMPLES}
\subsection{Total Variation}
\label{app_discussion_TV}
The total variation of the \( (r-1) \)th weak derivative of a function \( f \), denoted as \( TV^{(r)}(f) \), is a key concept in this formulation. 
It is well-known that the total variation \( TV^{(r)}(f) \) is equivalent to the Riemann approximation of the integral \( \int_{[0,1]} |f^{(r)}(t)| \, dt \) if \( f \) is \( r \) times differentiable. Moreover, for \( (r-1) \)th order polynomials, the \( r \)th derivative is zero, making \( TV^{(r)}(g) = 0 \) for all \( g(x) = x^l \) where \( l = 0, \dots, r-1 \). 

\subsection{Discrete Difference Operator for Trend Filtering}

The \( r \)th order trend filtering estimate is derived from the following penalized least squares optimization problem:
\[
\hat{\theta} = \argmin_{\theta \in \mathbb{R}^n} \left( \frac{1}{2} \|y - \theta \|_2^2 + \frac{n^{r-1}}{(r-1)!} \cdot \lambda \bigl\|D^{(r)} \theta \bigr\|_1 \right),
\]
as described by \cite{kim2009ell_1} and \cite{tibshirani2014adaptive}. Here, \( D^{(r)} \) represents the discrete difference operator of order \( r \), which is a banded matrix with bandwidth \( r+1 \). The operator \( D^{(r)} \) can be understood as the discrete analogue of the \( r \)th order derivative operator, with the penalty term enforcing smoothness by penalizing the discrete \( r \)th derivative of the vector \( \theta \in \mathbb{R}^n \).

\subsubsection{Examples of the Discrete Difference Operator}
\label{app_examples_discrete_difference_operator}
To further clarify the structure of \( D^{(X,r)}_n \), we consider specific examples for \( r = 1 \) and \( r = 2 \), which approximate the first and second derivatives, respectively.

\paragraph{Example 1 \( (r = 1) \).}

Given four points \((X^1, X^2, X^3, X^4)\) and its sorted ones \( 0 < X^{(1)} < X^{(2)} < X^{(3)} < X^{(4)} < 1 \) within the interval \([0, 1]\), the first-order operator \( D^{(X,1)}_4 \) is:
\[
D^{(X,1)}_4 = \begin{pmatrix}
	-1 & 1 & 0 & 0 \\
	0 & -1 & 1 & 0 \\
	0 & 0 & -1 & 1 
\end{pmatrix} \in \R^{3 \times 4}, \quad 
D^{(X,1)}_4(\theta) = (\theta^2 - \theta^1, \theta^3 - \theta^2, \theta^4 - \theta^3).
\]

\paragraph{Example 2 \( (r = 2) \).}

For the second-order operator \( D^{(X,2)}_4 \), using the same four points as before, we have:
\[
D^{(X,2)}_4 = D^{(X,1)}_3 \, \text{diag}\left(\frac{1}{X^{(2)} - X^{(1)}}, \frac{1}{X^{(3)} - X^{(2)}}, \frac{1}{X^{(4)} - X^{(3)}}\right) D^{(X,1)}_4,
\]
which explicitly becomes:
\[
D^{(X,2)}_4 = \begin{pmatrix}
	\frac{1}{X^{(2)} - X^{(1)}} & \frac{-1}{X^{(2)} - X^{(1)}} + \frac{-1}{X^{(3)} - X^{(2)}} & \frac{1}{X^{(3)} - X^{(2)}} & 0 \\
	0 & \frac{1}{X^{(3)} - X^{(2)}} & \frac{-1}{X^{(3)} - X^{(2)}} + \frac{-1}{X^{(4)} - X^{(3)}} & \frac{1}{X^{(4)} - X^{(3)}}
\end{pmatrix}.
\]
This operator acts as a second-order difference operator, adjusted for the non-uniform spacing of the points.

\paragraph{Special Case \( (r = 0) \).}

For \( r = 0 \), the operator \( D^{(X,0)}_n \) is simply the identity matrix:
\[
D^{(X,0)}_n = I, \quad D^{(X,0)}_n(\theta) = \theta.
\]

\section{EXAMPLES USING ALGORITHM~\ref{alg-backfitting}}
\label{app_algorithm_example}
\vspace{5pt}
\begin{example}
    In Figure~\ref{fig:4}, we present the true quantile signal alongside the fitted values for each component in the model. The black curves represent the functions $f_{0j}(x) = a_jg_j(x) - b_j$ for $j = 1, \dots, 4$, where $a_j$ and $b_j$ are chosen such that $f_{0j}(X_j)$ has an empirical mean zero and an empirical norm $\|f_{0j}\|_n = 1$. Specifically:
    \begin{itemize}
        \item $g_{1}(x) = -\frac{1}{2}x^2$
        \item $g_{2}(x) = \frac{3}{2} \sin(4 \pi x) + \mathbbm{1}_{x \leq \frac{1}{2}} \cdot \sin(16 \pi x)$
        \item $g_{3}$ is a dummy dimension (where only one randomly assigned point takes a non-zero value)
        \item $g_{4}(x) = e^{3 x} \sin(4 \pi x)$.
    \end{itemize}
The blue curves in the figure plot the fitted $\widehat{f}_j$ values for $j = 1, \dots, 4$. The model is fitted on 1000 noisy data points, constructed similarly to our simulated experiments in Section 5. Here, the noise is generated from independent draws from a $t$ distribution with $3$ degrees of freedom denoted as  $t(3)$. This figure demonstrates that quantile additive trend filtering effectively captures varying levels of smoothness both within and between component functions, even under heavy-tailed error conditions.
\end{example}

\begin{example}
    Figure~\ref{fig:3} provides a 3D visualization of the true quantile signal $\sum_{j = 1}^{2}f_{0j}(X_j)$, the noisy data $y$, and the reconstructed signal $\sum_{j = 1}^{2}\widehat{f}_{j}$. These components, $f_{0j}$, $y$, and $\widehat{f}_{j}$ for $j = 1, \dots, 2$, are defined similarly to those in Example 1, and $X_j = (X^1_j \dots X^{2000}_j)$ constructed as described in our experiments in Section 5. The underlying component functions here are:
    \begin{itemize}
        \item $g_{1}(x) = \frac{1}{2} \cos(6 \pi x) + 0.1$
        \item $g_{2}(x) = -(x - \frac{1}{2})^2$.
    \end{itemize}
    With only two input dimensions, the model's output can be directly plotted along its inputs, allowing us to visually assess the model's effectiveness and the performance of our back-fitting algorithm.
\end{example}

\section{COMPUTATION SETTING}
\label{computation_setting}
In the main paper, we have provided detailed information on data reproduction in the experimental section referenced as \ref{experiment}. All experiments were conducted on a Linux-based system equipped with an Intel(R) Xeon(R) Platinum 8160 CPU running at 2.10 GHz. The system had 24 processor cores and a total memory capacity of 260 GB.
The experiments were performed using R version 4.2.2. 

\section{ADDITIONAL EXPERIMENTS RESULTS}
\label{add_exp_res}

\begin{table}[ht]
\centering
\caption {Average mean squared error, $\frac{1}{n}\sum_{i=1}^n(f^i_0 - \widehat f^i)^2$, averaging over 50 Monte Carlo simulations for the different methods considered. The best MSE is listed with bold text.}
\label{table2_more_result}
\begin{tabular}{|l|l|l|l|l|l|l|l|l|}
\hline
\multicolumn{1}{|c|}{n} & \multicolumn{1}{c|}{Scenario} & \multicolumn{1}{c|}{d} & \multicolumn{1}{c|}{$\tau$} & \multicolumn{1}{c|}{QATF1} & \multicolumn{1}{c|}{QATF0} & \multicolumn{1}{c|}{QS} & \multicolumn{1}{c|}{ATF1} & \multicolumn{1}{c|}{ATF0} \\ \hline
500                     & 1                             & 10                     & 0.2                         & \textbf{1.3440}            & 1.4921                     & 1.3603                  & NA                        & NA                        \\ 
1000                    & 1                             & 10                     & 0.2                         & \textbf{0.7628}            & 1.1676                     & 0.9103                  & NA                        & NA                        \\ 
2500                    & 1                             & 10                     & 0.2                         & \textbf{0.4034}            & 0.5036                     & 0.5405                  & NA                        & NA                        \\ \hline 
500                     & 1                             & 10                     & 0.8                         & \textbf{1.3648}            & 1.6115                     & 1.5026                  & NA                        & NA                        \\ 
1000                    & 1                             & 10                     & 0.8                         & \textbf{0.8058}            & 1.1484                     & 0.9332                  & NA                        & NA                        \\ 
2500                    & 1                             & 10                     & 0.8                         & \textbf{0.3812}            & 0.5585                     & 0.5304                  & NA                        & NA                        \\ \hline
500                     & 2                             & 10                     & 0.2                         & 3.1722                     & \textbf{3.1520}            & 3.3013                  & NA                        & NA                        \\ 
1000                    & 2                             & 10                     & 0.2                         & \textbf{3.0346}            & 3.2824                     & 3.1760                  & NA                        & NA                        \\ 
2500                    & 2                             & 10                     & 0.2                         & \textbf{1.3725}            & 1.6200                     & 1.6323                  & NA                        & NA                        \\ \hline 
500                     & 2                             & 10                     & 0.8                         & \textbf{4.4045}            & 4.5964                     & 4.6682                  & NA                        & NA                        \\ 
1000                    & 2                             & 10                     & 0.8                         & \textbf{2.7262}            & 3.1071                     & 2.7387                  & NA                        & NA                        \\ 
2500                    & 2                             & 10                     & 0.8                         & \textbf{1.4714}            & 1.6283                     & 1.7447                  & NA                        & NA                        \\ \hline
500                     & 3                             & 10                     & 0.2                         & \textbf{1.3607}            & 1.9141                     & 1.5422                  & NA                        & NA                        \\ 
1000                    & 3                             & 10                     & 0.2                         & \textbf{0.6426}            & 1.1087                     & 0.8326                  & NA                        & NA                        \\ 
2500                    & 3                             & 10                     & 0.2                         & \textbf{0.2518}            & 0.4391                     & 0.3849                  & NA                        & NA                        \\ \hline
500                     & 3                             & 10                     & 0.8                         & \textbf{2.2376}            & 2.4855                     & 2.4760                  & NA                        & NA                        \\ 
1000                    & 3                             & 10                     & 0.8                         & \textbf{1.7452}            & 2.0362                     & 1.9083                  & NA                        & NA                        \\ 
2500                    & 3                             & 10                     & 0.8                         & \textbf{1.1270}            & 1.2700                     & 1.2928                  & NA                        & NA                        \\ 
\hline 
500                     & 5                             & 10                     & 0.2                         & 2.4657                     & \textbf{2.3672}            & 2.7830                  & NA                        & NA                        \\ 
1000                    & 5                             & 10                     & 0.2                         & 1.6413                     & \textbf{1.4640}            & 1.7123                  & NA                        & NA                        \\ 
2500                    & 5                             & 10                     & 0.2                         & 0.9482                     & \textbf{0.6189}            & 0.9951                  & NA                        & NA                        \\ \hline 
500                     & 5                             & 10                     & 0.8                         & 2.7530                     & \textbf{2.5349}            & 2.8937                  & NA                        & NA                        \\ 
1000                    & 5                             & 10                     & 0.8                         & 1.5717                     & \textbf{1.5360}            & 1.6207                  & NA                        & NA                        \\ 
2500                    & 5                             & 10                     & 0.8                         & 0.9286                     & \textbf{0.6175}            & 0.9613                  & NA                        & NA                        \\ \hline
500  & 6        & 10 & 0.2    & \textbf{3.4553}                 & \textbf{3.4553}                 & 3.8648                  & NA                        & NA                                 \\ 
1000 & 6        & 10 & 0.2    & \textbf{2.7782}                 & 3.3125                     & 2.9650                  & NA                        & NA                                 \\ 
2500 & 6        & 10 & 0.2    & \textbf{2.2229}                 & 2.4699                     & 2.3843                  & NA                        & NA                                 \\ \hline
500  & 6        & 10 & 0.8    & \textbf{0.0245}                 & 0.0345                     & 2.6176                  & NA                        & NA                                 \\ 
1000 & 6        & 10 & 0.8    & \textbf{0.0243}                 & 0.0267                     & 1.7869                  & NA                        & NA                                 \\ 
2500 & 6        & 10 & 0.8    & \textbf{0.0245}                 & 0.0271                     & 1.4223                  & NA                        & NA                                 \\ \hline

\end{tabular}
\end{table}

\clearpage




\section{WORLD HAPPINESS COMPONENT PLOTS}
\label{sec:component_plots}

Provided in this section are the component plots for all variables considered in the statistical inference study on the World Happiness Data. 

\begin{figure}[htbp] 
    \centering 
    \includegraphics[width=\textwidth]{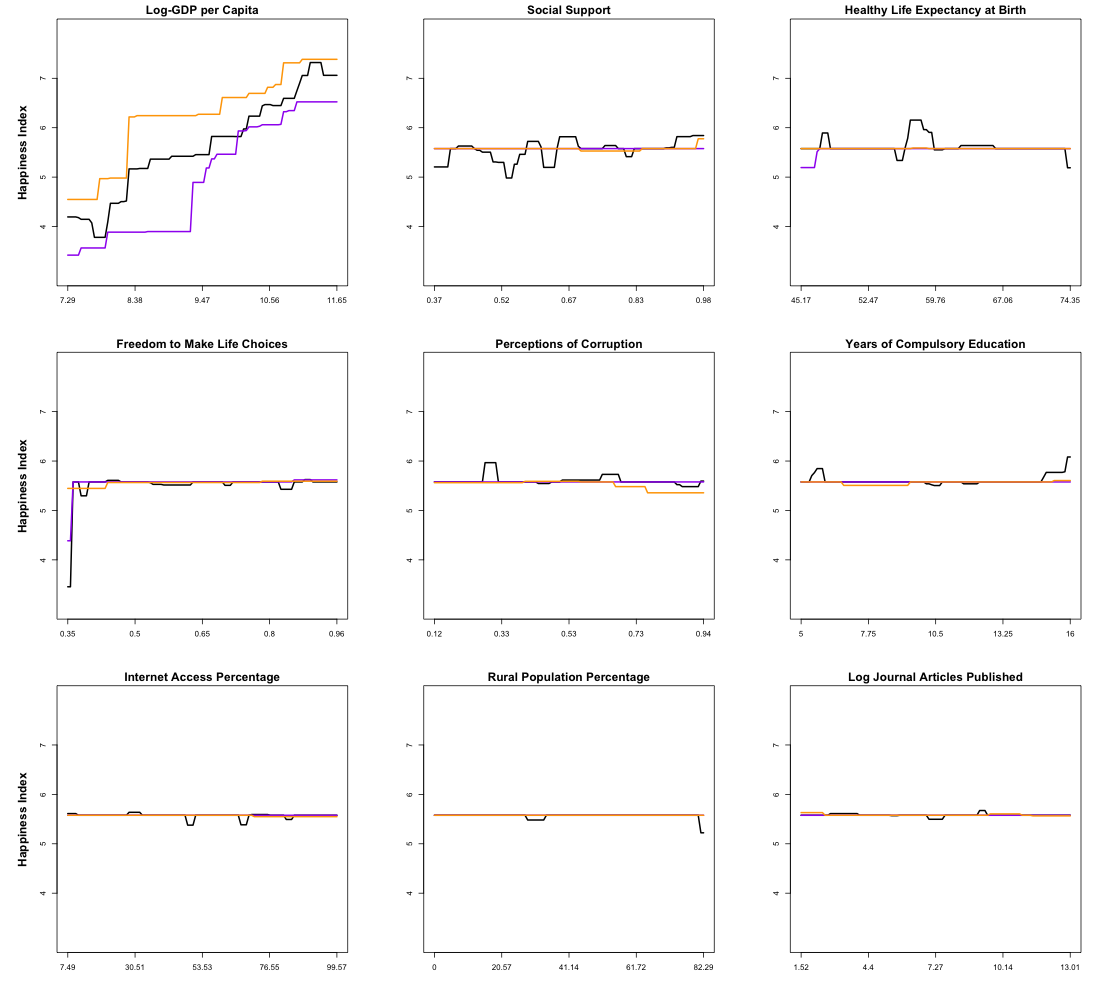} 
    \caption{Quantile Additive Trend Filtering component estimations with $\tau = 0.1, 0.5, \text{and  } 0.9$, plotted in purple, black, and orange, respectively, for most relevant components.}
    \label{fig:component_plots} 
\end{figure}




\end{document}